\documentclass[smallextended]{svjour3}
\usepackage{amsfonts,amsmath,amssymb}
\usepackage{mathrsfs,mathtools}
\usepackage{enumerate}
\usepackage{xspace}
\usepackage{mydef}
\usepackage{hyperref}
\usepackage{esint}
\usepackage{graphicx}
\usepackage{psfrag}
\usepackage{cite}
\usepackage[percent]{overpic}
\usepackage[normalem]{ulem}

\DeclareMathAlphabet{\mathpzc}{OT1}{pzc}{m}{it}

\newcommand{\slope}{\mathfrak{s}}
\newcommand{\Tg}{\mathcal{T}^{L,n}_{geo,\sigma}}
\newcommand{\Tmin}{\mathcal{T}^{L,q}_{min,\lambda}}

\newcommand{\M}{\mathcal{M}}

\newcommand{\eremk}{\hbox{}\hfill\rule{0.8ex}{0.8ex}}
\numberwithin{equation}{section}

\newcommand{\BL}{{\sf BL}}
\newcommand{\Co}{{\sf C}}
\newcommand{\Te}{{\sf T}}
\newcommand{\Mi}{{\sf M}}
\newcommand{\xh}{{\widehat x}}
\newcommand{\yh}{{\widehat y}}
\newcommand{\calA}{{\mathcal A}}
\newcommand{\calC}{{\mathcal C}}
\newcommand{\calG}{{\mathcal G}}
\newcommand{\calK}{{\mathcal K}}
\newcommand{\calL}{{\mathcal L}}

\newcommand{\calM}{{\mathcal M}}

\newcommand{\calT}{{\mathcal T}}

\newcommand{\bA}{{\boldsymbol A}}
\newcommand{\bmr}{{\boldsymbol r}}

\newcommand{\eps}{\varepsilon}

\newcommand{\bbN}{{\mathbb N}}

\newcommand{\bbR}{{\mathbb R}}
\newcommand{\bbS}{{\mathbb S}}

\newcommand{\EhpFEM}{Extended $hp$-FEM }
\newcommand{\EhpFEMp}{Extended $hp$-FEM}
\newcommand{\BK}{sinc BK-FEM }
\newcommand{\BKp}{sinc BK-FEM}

\newcommand{\bO}{{\mathbf O}}
\newcommand{\bP}{{\mathbf P}}
\begin{document}
\title{Exponential Convergence of $hp$ FEM for 
Spectral Fractional Diffusion in Polygons
\thanks{
The research of JMM was supported by the Austrian Science Fund (FWF) project F 65. 
Work performed in part while CS was visiting the Erwin Schr\"odinger Institute (ESI)
in Vienna in June-August 2018 during the ESI thematic period 
``Numerical Analysis of Complex PDE Models in the Sciences''. 
Research of CS supported
in part by the Swiss National Science Foundation, under Grant SNSF 200021-159940
}
}
\titlerunning{Exponential Convergence for Spectral Fractional Diffusion in Polygons}        

\author{Lehel Banjai \and
Jens M. Melenk \and
Christoph Schwab}

\institute{L. Banjai \at Maxwell Institute for Mathematical Sciences\\ School of Mathematical  
        \& Computer Sciences\\ Heriot-Watt University\\ Edinburgh EH14 4AS, UK \\ \email{l.banjai@hw.ac.uk}
\and 
J.M. Melenk \at Institut  f\"{u}r  Analysis und  Scientific Computing\\ 
Technische Universit\"{a}t Wien\\ A-1040  Vienna, Austria \\\email{melenk@tuwien.ac.at}
\and
C. Schwab \at Seminar for Applied Mathematics\\ ETH Z\"{u}rich, ETH Zentrum, 
HG  G57.1\\ CH8092 Z\"{u}rich, Switzerland \\ \email{christoph.schwab@sam.math.ethz.ch}}

\maketitle
\begin{abstract}
For the spectral fractional diffusion operator of order $2s\in (0,2)$
in bounded, curvilinear polygonal domains $\Omega\subset \bbR^2$ 
we prove exponential convergence of two classes of $hp$ discretizations 
under the assumption of analytic data (coefficients and source terms, 
without any boundary compatibility),
in the natural fractional Sobolev norm $\mathbb{H}^s(\Omega)$.
The first $hp$ discretization is based on 
writing the solution as a co-normal derivative
of a $2+1$-dimensional local, linear elliptic boundary value problem, 
to which an 
$hp$-FE discretization is applied.
A diagonalization in the extended variable reduces the numerical
approximation of the inverse of the spectral fractional diffusion operator
to the numerical approximation of a system of \emph{local, 
decoupled, second order reaction-diffusion equations in $\Omega$}.
Leveraging results on robust exponential convergence of 
$hp$-FEM for second order, linear reaction diffusion 
boundary value problems in $\Omega$, 
exponential convergence rates 
for solutions $u\in \mathbb{H}^s(\Omega)$ of $\calL^s u = f$
follow. 
Key ingredient in this $hp$-FEM are 
\emph{boundary fitted meshes with geometric mesh refinement towards $\partial\Omega$}.  

The second discretization is based on 
exponentially convergent numerical sinc quadrature approximations
of the Balakrishnan integral representation of $\calL^{-s}$ 
combined with $hp$-FE discretizations of a 
\emph{decoupled system of local, linear, 
      singularly perturbed reaction-diffusion equations in $\Omega$}.
The present analysis for either approach 
extends to (polygonal subsets $\widetilde{\calM}$
of) analytic, compact $2$-manifolds $\calM$, 
parametrized by a global, analytic chart $\chi$
with polygonal Euclidean parameter domain $\Omega\subset \bbR^2$.
Numerical experiments for model problems
in nonconvex polygonal domains and with incompatible data
confirm the theoretical results.

Exponentially small bounds on Kolmogoroff $n$-widths of
solutions sets for spectral fractional diffusion in polygons
are deduced.
\end{abstract}

\section{Introduction}
\label{S:introduction}
In recent years, the mathematical and numerical analysis of 
initial-boundary value problems for fractional differential operators
has received substantial attention. 
Their numerical treatment has to overcome several challenges. 
The first challenge arises from their nonlocal nature as integral operators.
A direct Galerkin discretization leads to fully populated system matrices, and 
compression techniques (see, e.g., \cite{KarkJMM19} and the references there) 
have to be brought to bear to make the discretization computationally tractable. 
An alternative to a direct Galerkin discretization of an integral operator, 
which is possible for the presently considered spectral fractional Laplacian, 
is to realize the nonlocal operator numerically
as a Dirichlet-to-Neumann operator for a local (but degenerate) 
elliptic problem.
While this approach, 
sometime referred to as ``Caffarelli-Silvestre'' extension (``CS-extension'' for short)
\cite{CS:07,ST:10} increases the spatial dimension by $1$, 
it permits to use the mathematical and numerical 
tools that were developed for local, integer order differential operators. 
In the present paper, we study several 
$hp$-FE discretizations of the resulting local (but degenerate) 
elliptic problem, to which we will refer as \emph{\EhpFEMp}. 

One alternative to the extension approach is
the representation of fractional 
powers of elliptic operators as Dunford-Taylor integrals 
proposed in \cite{BoPascFracRegAcc2017,BP:13}.
Discretizing such an integral leads to a sum 
of solution operators for \emph{local}, second order elliptic problems, 
which turn out to be singularly perturbed, 
but are amenable to established numerical techniques.
In the present paper, 
we also study this approach under the name \emph{sinc Balakrishnan FEM}
(\emph{\BKp} for short).

A second challenge arises from the fact that  the solutions of problems
involving fractional operators are typically not smooth, even for 
smooth input data 
(cf.\ the examples and discussion in \cite[Sec.~{8.4}]{BMNOSS17_732}).
Indeed, for the spectral fractional Laplacian, 
the behavior near a smooth boundary $\partial\Omega$
is 
$u \sim u_0 +  O(\operatorname{dist}(\cdot,\partial\Omega)^\beta)$ for 
some more regular $u_0$ and 
a $\beta\not\in{\mathbb N}$ \cite[Thm.~{1.3}]{CafStinga16}. 
Points of non-smoothness of $\partial\Omega$
introduce further singularities into the solution.  
The numerical resolution of both types of singularities requires 
suitably designed approximation spaces. 
For the spectral fractional Laplacian in two-dimensional
polygonal domains, 
we present a class of meshes in $\Omega$ with anisotropic, geometric
refinement
towards $\partial\Omega$ and with isotropic geometric refinement 
towards the corners of $\Omega$. 
We show that spaces of piecewise polynomials on such meshes 
can lead to exponential convergence. 

%
\begin{figure}[th]
\psfragscanon
\psfrag{A01}{\tiny $A_{J_1}^{(1)} = A_{0}^{(1)}$}
\psfrag{A11}{\tiny $A_{1}^{(1)}$}
\psfrag{A21}{\tiny $A_{2}^{(1)}$}
\psfrag{A31}{\tiny $A_{3}^{(1)}$}
\psfrag{Ajm11}{\tiny $A_{J_1-1}^{(1)}$}
\psfrag{G11}{\tiny $\Gamma_1^{(1)}$}
\psfrag{G21}{\tiny $\Gamma_2^{(1)}$}
\psfrag{G31}{\tiny $\Gamma_3^{(1)}$}
\psfrag{Gjm11}{\tiny $\Gamma_{J_1-1}^{(1)}$}
\psfrag{O01}{\tiny $\omega_{J_1}^{(1)}$}
\psfrag{A02}{\tiny $A_0^{(2)}$}
\psfrag{A12}{\tiny $A_1^{(2)}$}
\psfrag{G02}{\tiny $\Gamma_1^{(2)}$}
\psfrag{G12}{\tiny $\Gamma_2^{(2)}$}
\begin{overpic}[width=0.5\textwidth]{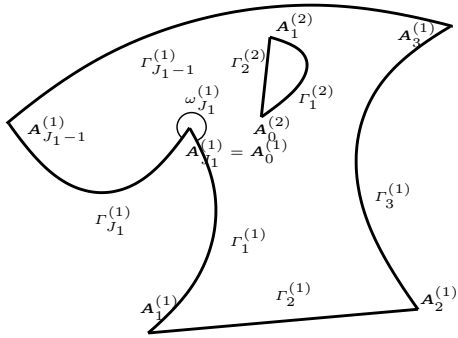}
\put(40,40){\tiny $\bA_{J_1}^{(1)} = \bA_{0}^{(1)}$}
\put(30,05){\tiny $\bA_{1}^{(1)}$}
\put(92,7){\tiny $\bA_{2}^{(1)}$}
\put(87,66){\tiny $\bA_{3}^{(1)}$}
\put(5,45){\tiny $\bA_{J_1-1}^{(1)}$}
\put(50,20){\tiny $\Gamma_1^{(1)}$}
\put(60,8){\tiny $\Gamma_2^{(1)}$}
\put(82,30){\tiny $\Gamma_3^{(1)}$}
\put(30,60){\tiny $\Gamma_{J_1-1}^{(1)}$}
\put(20,25){\tiny $\Gamma_{J_1}^{(1)}$}
\put(40,52){\tiny $\omega_{J_1}^{(1)}$}
\put(55,45){\tiny $\bA_0^{(2)}$}
\put(60,68){\tiny $\bA_1^{(2)}$}
\put(65,52){\tiny $\Gamma_1^{(2)}$}
\put(50,60){\tiny $\Gamma_2^{(2)}$}
\end{overpic}
\psfragscanoff
\caption{\label{fig:curvilinear-polygon} Example of a curvilinear polygon}
\end{figure}

%
\subsection{Geometric Preliminaries}
\label{sec:GeoPrel}
%
As in \cite{banjai-melenk-schwab19-RD},
we consider a bounded Lipschitz domain 
$\Omega \subset \bbR^2$ that is a curvilinear
polygon as depicted in Fig.~\ref{fig:curvilinear-polygon}. 
The boundary $\partial \Omega$ is assumed to consist 
of $J \in \bbN$ closed curves $\Gamma^{(i)}$. 
Each curve $\Gamma^{(i)}$ in turn is assumed to comprise
$J_i \in \bbN$ many open, disjoint, \emph{analytic arcs} $\Gamma^{(i)}_j$, 
$j=1,\ldots,J_i$, with 
$
\overline{\Gamma^{(i)}} = \bigcup_{j=1}^{J_i} \overline{\Gamma^{(i)}_j}\;,
\quad i=1,\ldots,J\;.
$
The arcs $\Gamma^{(i)}_j$ are assumed further to admit 
\emph{nondegenerate, analytic parametrizations}, 
$$
\Gamma^{(i)}_j
=
\left\{ {\mathbf x}^{(i)}_j (\theta)\,  | \, \theta \in (0,1) \right\}\;,
\quad 
i=1,\ldots,J, \quad j=1,\ldots,J_i \;.
$$
The coordinate functions 
$x_j^{(i)}$, $y_j^{(i)}$ of ${\mathbf x}^{(i)}_j(\theta) = (x_j^{(i)}(\theta),y_j^{(i)}(\theta))$
are assumed to be (real) analytic functions of $\theta \in  [0,1]$ and 
such that 
%
$
\min_{\theta \in [0,1]}
|\dot{\mathbf x}_j^{(i)}(\theta)|^2
> 0$ for $
j=1,\ldots,J_i$,  
$i=1,\ldots,J$. 
The end points of the arcs $\Gamma^{(i)}_j$ are denoted as
$\bA^{(i)}_{j-1} = {\mathbf x}_j^{(i)}(0)$ and $\bA^{(i)}_j = {\mathbf x}_j^{(i)}(1)$. 
We enumerate these points counterclockwise by indexing cyclically
with $j$ modulo $J_i$, 
thereby identifying in particular $\bA^{(i)}_{0} := \bA^{(i)}_{J_i}$.
%
The interior angle at $\bA_j^{(i)}$ is denoted $\omega_j^{(i)} \in (0,2\pi)$. 
For notational simplicity, we assume henceforth that $J=1$, i.e., $\partial\Omega$ consists 
of a single component of connectedness. 
We write $\bA_j = \bA_j^{(1)}$, $\Gamma_j$ for $\Gamma_j^{(1)}$. 
%
\subsection{Spectral Fractional Diffusion}
\label{sec:FracDiff}
%
When dealing with fractional operators, care must be
exercised in stating the definition of the fractional
powers. Here, we consider the so-called 
\emph{spectral fractional diffusion operators}
as investigated in \cite{CS:07}. 
We refer to the surveys \cite{RosOton2016Surv,BonitoEtAl_FracSurv2017,AinsworthEtAl_FracSurv2018}
and the references there for a comparison of the different
definitions of fractional powers of the Dirichlet Laplacian.

We consider 
the linear, elliptic, self-adjoint, second order differential operator 
$w \mapsto  \calL w = - \DIV( A \GRAD w ) $,
in a bounded, curvilinear polygon $\Omega\subset \R^2$ as described
in Section~\ref{sec:GeoPrel}.
The diffusion coefficient $A\in L^\infty(\Omega,\GL(\R^2))$ 
is assumed symmetric,  uniformly positive definite.
The data $A$ and $f$ are assumed analytic in $\overline{\Omega}$.
We quantify analyticity of $A$ and $f$ 
by assuming that there are $C_A$, $C_f > 0$ 
such that 
\begin{equation}\label{eq:AnRegAc}
\forall p \in \bbN_0:\;\;
\| |D^p A| \|_{L^\infty(\Omega)} \leq C_A^{p+1} p!\;,
\quad 
\| |D^p f| \|_{L^\infty(\Omega)} \leq C_f^{p+1} p!\;.
\end{equation}
Here, the notation $|D^p A|$ signifies $\sum_{|\alpha| = p} | D^\alpha A |$, 
with the usual multi-index convention $D^\alpha$ 
denoting mixed weak derivatives of order $\alpha \in \bbN_0^2$
whose total order $|\alpha| = \alpha_1 + \alpha_2$.
Further, we employ standard notation for (fractional) 
Sobolev spaces $H^t(\Omega)$, consistent with the notation 
and definitions in 
\cite{mcLean}. 

We introduce the ``energy'' inner product $\blfa{\Omega}(\cdot,\cdot)$ 
on $H^1_0(\Omega)$ associated with the differential operator $\calL$ 
by 
\begin{equation}
\label{eq:blfOmega}
\blfa{\Omega}(w,v) = \int_\Omega \left( A \nabla w \cdot \nabla v 
\right) \diff x'
\;.
\end{equation}
The operator 
$\calL: H^1_0(\Omega) \to H^{-1}(\Omega)$
induced by this bilinear form is an isomorphism, 
due to the (assumed) positive definiteness of $A$.
Let 
$\{\lambda_k, \varphi_k \}_{k \in \mathbb{N}} \subset \R^+ \times H_0^1(\Omega)$
be a sequence of eigenpairs of $\calL$, normalized such that  
$\{\varphi_k\}_{k \in \mathbb{N}}$ is an orthonormal basis of $L^2(\Omega)$ 
and an orthogonal basis of $(H_0^1(\Omega), \blfa{\Omega}(\cdot,\cdot))$.
We introduce, for $\sigma \ge 0$, the domains of fractional powers of $\calL$
as 
\begin{equation}
\label{def:Hs}
  {\mathbb H}^\sigma(\Omega) 
= \left\{ v = \sum_{k=1}^\infty v_k \varphi_k: \| v \|_{{\mathbb H}^\sigma(\Omega)}^2 
= \sum_{k=1}^{\infty} \lambda_k^\sigma v_k^2 < \infty \right\}.
\end{equation}
We denote by 
${\mathbb H}^{-\sigma}(\Omega)$ the dual space of ${\mathbb H}^{\sigma}(\Omega)$. 
Denoting by $\langle \cdot,\cdot \rangle$ the 
${\mathbb H}^{-\sigma}(\Omega)\times  {\mathbb H}^{\sigma}(\Omega)$ 
duality pairing
that extends the standard $L^2(\Omega)$ inner product, we can identify 
elements $f \in {\mathbb H}^{-\sigma}(\Omega)$ 
with sequences $\{f_k\}_k$ 
(written formally as $\sum_k f_k \varphi_k$)
such that 
$\|f\|^2_{{\mathbb H}^{-\sigma}(\Omega)} = \sum_k |f_k|^2 \lambda_k^{-\sigma}   <\infty$.
With this identification, we 
can extend the definition of the norm in (\ref{def:Hs}) to $\sigma < 0$. 
Furthermore, the linear operator 
$\calL^s:\Hs \rightarrow \Hsd: v \mapsto \sum_{k=1}^\infty v_k \lambda_k^s \varphi_k$
is bounded and
the Dirichlet problem for the fractional diffusion in $\Omega$ may be stated as:
given a fractional order $s \in (0,1]$ and $f \in \Hsd$, 
find $u\in {\mathbb H}^{s}(\Omega)$ such that
\begin{equation}
\label{fl=f_bdddom}
    \calL^s u = f \quad \text{in } \Omega\;.
\end{equation}
The ellipticity estimate 
$\langle w, \calL^s w \rangle \geq \lambda_1^s \| w \|^2_{{\mathbb H}^{s}(\Omega)}$
valid for every $w\in {\mathbb H}^{s}(\Omega)$ 
implies the unique solvability of \eqref{fl=f_bdddom} 
for every $f\in {\mathbb H}^{-s}(\Omega)$.
The $hp$-FEM approximations of \eqref{fl=f_bdddom} 
developed and analyzed in the present work 
are not based on explicit or approximated eigenfunctions 
but instead on the localization of the operator $\calL^s$ in terms
of extension discussed in Sec.~\ref{sec:SpecFrcLap} and on the 
Dunford-Taylor integral discussed in Sec.~\ref{S:BalkrFrm}, 
the so-called Balakrishnan formula.  
\begin{remark}[compatibility condition] 
\label{remk:compatibility}
As discussed in \cite[Lemma~{1}, Rem.~{1}]{BMNOSS17_732} 
the spectral fractional Laplacian has the mapping property 
$\calL^s: {\mathbb H}^{s+\sigma}(\Omega) \rightarrow 
{\mathbb H}^{-s+\sigma}(\Omega)$, $\sigma \ge 0$. 
For smooth coefficients $A$ and $\partial\Omega$, the spaces 
${\mathbb H}^{s+\sigma}(\Omega)$, $\sigma \ge 0$, are 
subspaces of the Sobolev spaces $H^{s+\sigma}(\Omega)$. In fact, 
for $-s+\sigma > 1/2$, the spaces ${\mathbb H}^{-s+\sigma}(\Omega)$ 
are proper subspaces of $H^{-s+\sigma}(\Omega)$ as they 
encode some boundary conditions on $\partial\Omega$. E.g., for 
$f \in {\mathbb H}^{-s+\sigma}(\Omega)$ with $-s+\sigma \ge 1/2$ one has 
$f|_{\partial\Omega} = 0$. 
That is, $f \in H^{-s+\sigma}(\Omega)$ must satisfy additionally 
\emph{compatibility conditions} on $\partial\Omega$ 
to ensure $u \in H^{s+\sigma}(\Omega)$.  
\eremk
\end{remark}
%
\subsection{Contributions}
\label{S:contrib}
We briefly highlight the principal contributions of this work.
For the nonlocal, spectral fractional diffusion problem \eqref{fl=f_bdddom}
in bounded, curvilinear polygonal domains $\Omega$ as described in 
Section \ref{sec:GeoPrel} and with analytic data $A$ and $f$ 
as in \eqref{eq:AnRegAc}, and without any boundary compatibility,
we develop two $hp$-FEMs for \eqref{fl=f_bdddom}
that converge exponentially in terms of the number of degrees of freedom
$N_{DOF}$ in $\Omega$. The setting covers in particular also 
boundary value problems for fractional surface diffusion on analytic 
surface pieces as in the setting of Section \ref{S:FrcLapManif}.
Key insight in our error analysis is that 
either method, based on the extension of \eqref{fl=f_bdddom}
combined with a diagonalization procedure as in \cite[Sec. 6]{BMNOSS17_732}
or on a contour-integral representation of $\calL^s$ combined with
an exponentially converging sinc quadrature, reduce the numerical
solution of \eqref{fl=f_bdddom} to the numerical solution of
\emph{local, singularly perturbed second order reaction-diffusion 
problems in $\Omega$}. 
Drawing on analytic regularity and corresponding
$hp$-FEM in $\Omega$ for these reaction-diffusion problems with 
\emph{robust, exponential convergence} as developed in 
\cite{MelCS_RegSingPert,melenk-schwab98,melenk02,banjai-melenk-schwab19-RD}, 
we establish here exponential convergence rate bounds for
solutions of \eqref{fl=f_bdddom}. As we showed in \cite{melenk02,banjai-melenk-schwab19-RD},
the singular perturbation character of the reaction-diffusion problems in $\Omega$ mandates 
both, \emph{geometric corner mesh refinement} and 
\emph{anisotropic geometric boundary mesh refinement} to resolve the algebraic corner
and boundary singularities that occur in solutions to \eqref{fl=f_bdddom}.

Before proceeding to the main part of this paper, we briefly recall
the localization due to Caffarelli-Silvestre 
and the contour integral representation of Balakrishnan \cite{Balakr1960}.
\subsection{Caffarelli-Silvestre extension}
\label{sec:SpecFrcLap}
%
In \cite{CS:07} the (full space) fractional Laplacian $\calL^s$ was localized 
via a singular elliptic PDE depending on one extra variable and thus
represented as Dirichlet-to-Neumann problem for an elliptic problem
in a half-space. 
Cabr\'e and Tan \cite{CT:10} and Stinga and Torrea \cite{ST:10} 
extended this to bounded domains $\Omega$ and more general operators, 
thereby obtaining an extension posed on the
semi--infinite cylinder $\C := \Omega \times (0,\infty)$.
Their extension 
is given by the 
\emph{local boundary value problem}
\begin{equation}
\label{alpha_harm_intro}
  \begin{dcases}
    \mathfrak{L} \ue = -\DIV\left( y^\alpha {\boldsymbol{\mathfrak A}} \nabla \ue \right) 
= 0  
          & \text{ in } \C=\Omega \times (0,\infty), 
    \\
    \ue = 0  &\text{ on } \partial_L \C, 
    \\
    \partial_{\nu^\alpha} \ue = d_s f  &\text{ on } \Omega \times \{0\},
  \end{dcases}
\end{equation}
where ${\boldsymbol{\mathfrak A}} = {\rm diag} \{A,1\} \in L^\infty(\C,\GL(\bbR^{d+1}))$,
$\partial_L \C := \partial \Omega \times (0,\infty)$, 
$d_s: = 2^{1-2s} \Gamma(1-s)/\Gamma(s)>0$ 
and where $\alpha = 1-2s \in (-1,1)$ \cite{CS:07,ST:10}. 
The so--called conormal exterior derivative of $\ue$ at $\Omega \times \{ 0 \}$ is
\begin{equation}
\label{def:lf}
\partial_{\nu^\alpha} \ue = -\lim_{y \rightarrow 0^+} y^\alpha \ue_y.
\end{equation}
The limit in \eqref{def:lf} is in the distributional sense \cite{CT:10,CS:07,ST:10}. 
Fractional powers of $\calL^s$ in \eqref{fl=f_bdddom} and 
the Dirichlet-to-Neumann operator of problem \eqref{alpha_harm_intro} 
are related by \cite{CS:07,CafStinga16}
\begin{equation} \label{eq:identity}
  d_s \calL^s u = \partial_{\nu^\alpha} \ue \quad \text{in } \Omega\;.
\end{equation}
We write $x = (x',y)\in \C$ with $x' \in \Omega$ and $y > 0$. 
For $D \subset \mathbb{R}^{d} \times \mathbb{R}^+$, 
we define $L^2(y^\alpha,D)$ as the Lebesgue space with the measure $y^\alpha \diff x$  and $H^1(y^{\alpha},D)$ as the weighted Sobolev space
\begin{equation}
\label{eq:H1y}
  H^1(y^{\alpha},D) := \left\{
    w \in L^2(y^\alpha,D)\;\colon\; |\nabla w | \in L^2(y^\alpha,D)
  \right\}
\end{equation}
equipped with the norm
\begin{equation}
  \label{eq:H_norm}
  \|w\|_{H^1(y^\alpha,D)} = \left(\|w\|^2_{L^2(y^\alpha,D)}+\|\nabla w\|^2_{L^2(y^\alpha,D)}\right)^{1/2}.
\end{equation}
To investigate \eqref{alpha_harm_intro} 
we include the homogeneous boundary condition on the lateral boundary 
$\partial_L \C$ by setting 
\begin{equation}
  \HL(y^{\alpha},\C) 
:= \left\{ 
   w \in H^1(y^\alpha,\C)\;:\; w = 0 \text{ on } \partial_L \C 
  \right\}.
\end{equation}
The bilinear form 
$\blfa{\C}: \HL(y^{\alpha},\C) \times  \HL(y^{\alpha},\C) \to \R$ 
defined by
\begin{equation}
\label{eq:blf-a} 
\blfa{\C}(v,w) 
=  
\int_\C y^\alpha ({\boldsymbol{\mathfrak A}} \nabla v \cdot \nabla w 
) \diff x' \diff y,
\end{equation}
is continuous and coercive on $\HL(y^{\alpha},\C)$.
The energy norm $\normC{\cdot}$ 
on $\HL(y^\alpha,\C)$ induced by the inner product 
$a_{\C}(\cdot,\cdot)$ is given by
\begin{equation}
\label{eq:norm-C} 
\normC{v}^2:= \blfa{\C}(v,v) \sim \|\nabla v\|^2_{L^2(y^\alpha,\C)}\;.
\end{equation}
For $w \in H^1(y^\alpha,\C)$ we denote by $\tr w$ its trace
on $\Omega \times \{0\}$, which 
connects the spaces $\HL(y^\alpha,\C)$ and $\Hs$ 
(cf.\ \cite[Prop.~{2.5}]{NOS}) via
\begin{equation}
\label{Trace_estimate}
\tr \HL(y^\alpha,\C) = \Hs,
\qquad
  \|\tr w\|_{\Hs} \leq C_{\tr} \| w \|_{\HL(y^\alpha,\C)}.
\end{equation}
With these definitions at hand, 
the weak formulation of \eqref{alpha_harm_intro} 
is to find
\begin{equation}
\label{eq:ue-variational} 
\ue \in  \HL(y^{\alpha},\C): \;
\forall v \in \HL(y^\alpha,\C) \ \colon \ 
\blfa{\C}(\ue,v) = d_s \langle f,\tr v \rangle . 
\end{equation}
\begin{remark}[regularity of $\ue$ for $s = 1/2$]
\label{rem:s=1/2}
In the special case $s = 1/2$ and $A = \operatorname{Id}$ 
in \eqref{fl=f_bdddom} 
the operator $ \mathfrak{L}$ 
in the CS extension \eqref{alpha_harm_intro} in $\calC$ 
coincides with the Laplacian in $\calC$. Therefore, 
the solution $\ue$ will, 
in general, exhibit \emph{algebraic singularities}
on $\partial \Omega$, even if $\partial \Omega$ is smooth.
\eremk
\end{remark}
%
\subsection{Balakrishnan Formula}
\label{S:BalkrFrm}
The second approach we take is via the \emph{spectral integral representation of fractional powers of elliptic operators} going back to \cite{Balakr1960}. 
For $0<s<1$ and $\calL = -\DIV(A\GRAD)$ with homogeneous Dirichlet boundary conditions
the bounded linear operator 
$\calL^{-s}:\mathbb{H}^{-s}(\Omega) \rightarrow \mathbb{H}^s(\Omega)$
admits the following representation with $c_B = \pi^{-1} \sin(\pi s)$:
\begin{equation}\label{eq:BlkrRep}
\begin{array}{rl}
\calL^{-s} 
& = \displaystyle  
c_B \int_0^\infty \lambda^{-s} (\lambda I + \calL)^{-1} \diff\lambda
=
c_B \int_{-\infty}^\infty e^{(1-s)y}(e^yI + \calL)^{-1} \diff y 
\\
& \displaystyle 
= 
c_B \int_{-\infty}^\infty e^{-sy} (I + e^{-y}\calL)^{-1} \diff y \;.
\end{array}
\end{equation}
The representations \eqref{eq:BlkrRep}
were used in \cite{BoPascFracRegAcc2017,BP:13} in conjunction with
an exponentially convergent, so-called sinc quadrature approximation of 
\eqref{eq:BlkrRep} (see, e.g., \cite{Stenger83} for details) 
and an $h$-version Finite Element projection in $\Omega$ 
to obtain numerical approximations of the 
fractional diffusion equation \eqref{fl=f_bdddom} in $\Omega$.
Here, we generalize the results in \cite{BoPascFracRegAcc2017,BP:13} to the $hp$-FEM,
establishing exponential convergence rates in polygonal domains $\Omega$ for the
resulting \BK 
for data $A$ and $f$ that are analytic in $\overline{\Omega}$ (cf.\ \eqref{eq:AnRegAc}) 
without boundary compatibility of $f$.


\subsection{Outline}
\label{S:outline}
The outline of the remainder of the paper is as follows.
The following 
Section~\ref{S:FEMy} describes the $hp$-FE spaces and Galerkin methods for \eqref{alpha_harm_intro} 
based on tensor products of discretizations in the $x$ and the $y$ variable. 

Section~\ref{S:diagonalization-abstract-setting} develops the diagonalization
of the $hp$-FE semi-discretization in the extended variable. 
In particular, in Section~\ref{S:hpFEM} we prove exponential
convergence of an $hp$-FE \emph{semidiscretization} in $(0,\infty)$.
The diagonalization reduces the semidiscrete approximation of the 
CS-extended, localized problem to a collection of 
decoupled, linear second order reaction-diffusion problems in $\Omega$.

Section~\ref{S:approx-sing-perturb} 
presents the exponential convergence results from 
\cite{banjai-melenk-schwab19-RD} of $hp$-FE discretizations of 
linear, second order singularly perturbed reaction-diffusion equations 
in $\Omega$ and establishes robust (with respect to the 
perturbation parameter $\eps$) exponential convergence results for these. 

Section~\ref{S:hpx} completes the proof of exponential convergence 
for the \EhpFEM by applying
the $hp$-FEM from Section~\ref{S:approx-sing-perturb}
in $\Omega$ for the reaction-diffusion problems obtained from the
diagonalization process in Section~\ref{S:diagonalization-abstract-setting}. 
Section~\ref{S:hpx} presents in fact two distinct $hp$-FE discretizations: 
a pure Galerkin method ({\bf Case B}) and 
a method based on discretizing 
after diagonalization each decoupled problem separately ({\bf Case A}). 
The latter approach features slightly better complexity estimates.

Section~\ref{S:ExpConvIII} is devoted to the analysis of the \BKp. 
There once more the numerical approximation of the fractional Laplacian
is reduced 
to the numerical solution of a sequence of \emph{local} 
linear, second order reaction-diffusion problems in $\Omega$. 
Applying exponential convergence bounds for sinc approximation
and for $hp$-FEM for reaction-diffusion problems in $\Omega$ 
in Section~\ref{S:approx-sing-perturb} from \cite{banjai-melenk-schwab19-RD},
once again exponential convergence for the resulting \BK 
for the spectral version of the fractional diffusion operator 
is established. 
As in Section~\ref{S:hpx}, we separately discuss the 
possibilities of approximating the solutions of the decoupled problems 
from the same space ({\bf Case B}) or from different spaces ({\bf Case A}).
Section~\ref{S:NumExp} has numerical experiments verifying
the theoretical convergence results. 
Section~\ref{S:Concl} has a summary and
outlines several generalizations and directions for further research.
In particular, we address in Section~\ref{S:FrcLapManif} the extension
to fractional diffusion on manifolds.
In Section~\ref{sec:Nwidth}, we discuss several exponential bounds
on Kolmogoroff $n$-widths of solution sets for 
spectral diffusion in polygons that follow from our results.
\subsection{Notation} 
\label{sec:Notat}
Constants 
$C$, $\gamma$, $b$ may be different in each occurence,
but are independent of critical parameters. 
We denote by $\widehat{S}:=(0,1)^2$ the reference square and by 
$\widehat{T}:=\{(\xi_1,\xi_2)\,|, 0 < \xi_1 < 1, \ 0 < \xi_2 < \xi_1\}$
the reference triangle. Sets of the form $\{x = y\}$, $\{x = 0 \}$, $\{x = y\}$
etc. refer to edges and diagonals of $\widehat{S}$ and analogously 
$\{y \leq x\} = \{(x,y) \in \widehat{S}\,|\, y \leq x\}$. 
${\mathbb P}_q$ denotes the space of polynomials of total degree $q$ and 
${\mathbb Q}_q$ the tensor product space of polynomial of degree $q$ 
in each variable separately. 
%
\section{$hp$-FEM Discretization}
\label{S:FEMy}
In this section, we introduce some $hp$-FEM space in both the $x$ and the $y$-variable
on which the \EhpFEM will be based. In particular, we introduce the geometric 
meshes $G^M_{geo,\sigma}$ that are used for the discretization in the $y$-variable. 
%
\subsection{Notation and FE spaces}
\label{S:NtFESpc}
\subsubsection{Meshes and FE spaces on $(0,\Y)$}
\label{S:GeoMes0Y}
%
Given a truncation parameter $\Y$ 
and 
a mesh $\calG^M = \{ I_m \}_{m=1}^M$ in $[0,\Y]$ consisting of $M$
intervals $I_m = [y_{m-1},y_m]$, with $0 = y_0 < y_1 < \cdots < y_M = \Y$,
we associate to $\calG^M$ a \emph{polynomial degree distribution} 
$\bmr = (r_1,r_2,\dots,r_M) \in \bbN^M$. 
We introduce the $hp$-FE space 
\begin{align*}
S^\bmr((0,\Y),\calG^M) 
& = 
\bigl\{ v_M \in H^1(\bbR)\,\colon\, \operatorname{supp} v \subset [0,\Y], \\
& \qquad  v_M|_{I_m} \in \mathbb{P}_{r_m}(I_m), I_m \in \calG^M,  
           m=1, \dots, M \bigr \},
\end{align*}
where $\mathbb{P}_r$ denotes the space of polynomials of degree $r$.
We will primarily work with the following piecewise polynomial 
space $S^\bmr_{\{\Y\}} ((0,\Y),\calG^M)\subset H^1(0,\infty)$ of functions 
that vanish on $[\Y,\infty)$:
\begin{equation}\label{eq:SrGM}
S_{\{\Y\}}^\bmr((0,\Y),\calG^M) = 
\left \{ v \in S^\bmr((0,\Y),\calG^M)\,\colon\, v(\Y) = 0\right\}.
\end{equation}
For constant polynomial degree $r_i=r \geq 1$, $i=1,\ldots,M$, we set
$S_{\{\Y\}}^r((0,\Y),\calG^M)$. 
Henceforth, we abbreviate 
\begin{equation}
\calM:= \operatorname*{dim} S^{\bmr}_{\{\Y\}}((0,\Y),\calG^M) \sim M^2 \;\;\mbox{as}\;\;M\to\infty\;.
\end{equation}
Of particular interest will be \emph{geometric meshes} 
$\calG^M_{geo,\sigma}$ on $[0,\Y]$,
with $M$ elements and grading factor 
$\sigma \in (0,1)$: $\{I_i\,|\,i=1,\ldots,M\}$ 
with elements $I_1 = [0,\Y \sigma^{M-1}]$ and 
$I_i = [\Y \sigma^{M-i+1},\Y \sigma^{M-i}]$ for $i=2,\ldots,M$.
On geometric meshes $\calG^M_{geo,\sigma}$ on $[0,\Y]$, 
we consider a \emph{linear polynomial degree vector $\bmr = \{ r_i \}_{i=1}^M$ 
with slope $\slope>0$} which is defined by 
\begin{equation}\label{eq:lindeg}
r_i := 1 + \lceil \slope (i-1) \rceil \} \;,\quad i=1,2,\ldots,M.
\end{equation}
For geometric meshes and linear degree vectors we set 
\begin{equation}
\calM_{geo} 
:= 
\operatorname*{dim} S^{\bmr}_{\{\Y\}}((0,\Y),\calG^M) \sim M^2 
\;\;\mbox{as}\;\;M\to\infty
\end{equation}
with constants implied in $\sim$ depending on $\slope>0$. 
\subsubsection{$hp$-FEM in $\Omega$}
\label{S:hpFEMOmega}
In the polygon $\Omega$, 
we consider Lagrangian FEM of uniform\footnote{We adopt uniform polynomial degree $q\geq 1$
here to ease notation and presentation. All approximation results admit
lower degrees in certain parts of the triangulations.  
This will affect, however, only constants in the error bounds, and
will not affect convergence rates in the ensuing exponential convergence
estimates.}
polynomial degree $q\geq 1$
based on regular triangulations of $\Omega$ denoted by $\calT$.
We admit both triangular and quadrilateral elements $K\in \calT$,
but \emph{do not} assume shape regularity. In fact, 
as we shall explain in Section~\ref{S:approx-sing-perturb}
ahead, 
\emph{anisotropic mesh refinement towards $\partial \Omega$} will be required
to resolve singularities at the singular support $\partial\Omega$ 
that are generically present in solutions of fractional PDEs 
(cf.\ Remark~\ref{rem:s=1/2}).
%
We introduce, for a regular 
(in the sense of \cite[Def.~{2.4.1}]{melenk02})
triangulation $\calT$ of $\Omega$ 
comprising curvilinear triangular or quadrilateral elements $K\in \calT$
with associated analytic element maps $F_K:\widehat{K} \rightarrow K$ 
(where $\widehat K \in \{\widehat T, \widehat S\}$ is either the reference
triangle or square depending on whether $K$ is a curvilinear triangle
or quadrilateral)
the FE space 
\begin{equation}
\label{eq:S^q_0}
S^q_0(\Omega,\calT)
= 
\left \{ v_h \in C(\bar \Omega): v_h|_{K}\circ F_K \in V_{q}(\widehat K) \quad 
\forall K \in \calT, \ v_h|_{\partial \Omega} = 0 \right \}.
\end{equation}
Here, for $q \ge 1$, the local polynomial space 
$V_q(\widehat K) = {\mathbb P}_q$ if $\widehat K = \widehat T$ 
and $V_q(\widehat K) = {\mathbb Q}_q$ if $\widehat K = \widehat S$.

\subsubsection{Tensor product $hp$-FE approximation}
\label{S:TPhpFE}
One $hp$-FE approximation of the extended problem 
\eqref{alpha_harm_intro} will be based on the
finite--dimensional \emph{tensor product spaces} of the form 
\begin{equation}\label{eq:TPFE}
\V^{q,\bmr}_{h,M}(\calT,\calG^M) 
:= 
S^q_0(\Omega,\calT) \otimes S_{\{ \Y\}}^\bmr((0,\Y),\calG^M)
\subset \HL(y^{\alpha},\C)\;,
\end{equation}
where $\calT$ is a regular triangulation of $\Omega$. 
To analyze this  method,
we consider semidiscretizations based on the following (infinite--dimensional, closed) 
Hilbertian tensor product space:  
\begin{equation}\label{eq:xySemiDis}
\begin{array}{l} 
\displaystyle
\displaystyle
\V^\bmr_{M}(\C_\Y) 
:= H^1_0(\Omega)\otimes S_{\{\Y\}}^\bmr((0,\Y),\calG^M) \subset \HL(y^{\alpha},\C)
\;.
\end{array}
\end{equation}
Here, the argument $\C_\Y$ indicates that spaces of functions 
supported in $\overline{\Omega} \times [0,\Y]$ are considered. 
Galerkin projections onto the spaces
$\V^{q,\bmr}_{h,M}(\calT^q , \calG^M)$ and 
$\V^\bmr_{M}(\C_\Y)$
with respect to the inner product 
$\blfa{\C}(\cdot,\cdot)$ are denoted by 
$G^{q,\bmr}_{h,M} $ and $G_M^\bmr$, respectively.
For the CS-extension $\ue$, i.e., the solution of (\ref{eq:ue-variational}), 
the Galerkin projections $G^{q,\bmr}_{h,M} \ue$
and $G^{\bmr}_{M} \ue$ are characterized by 
\begin{align}
\label{eq:fully-discrete}
a_\C(G^{q,\bmr}_{h,M} \ue, v)  =  d_s \langle f, \tr v\rangle \quad \forall v \in \V^{q,\bmr}_{h,M}(\calT,\calG^M) , \\
\label{eq:semi-discrete}
a_\C(G^{\bmr}_{M} \ue, v)   = d_s \langle f, \tr v\rangle \quad \forall v \in \V^{\bmr}_{M}(\calT,\calG^M). 
\end{align}
\section{Approximation based on semidiscretization in $y$}
\label{S:diagonalization-abstract-setting}
%
A key step in the $hp$-FE discretization in $(0, \Y)$ is, 
as in \cite{BMNOSS17_732},
the \emph{diagonalization} of the semidiscretized, truncated extension problem
with solution $G^{\bmr}_{M} \ue$ given by (\ref{eq:semi-discrete}). 
\subsection{Exponential Convergence of $hp$-FEM in $(0,\infty)$}
\label{S:hpFEM}
As in \cite{BMNOSS17_732,MPSV17},
we exploit the analytic regularity of the extended solution $\ue$ with respect 
to the extended variable $y$.  
It results in \emph{exponential convergence} of the 
$hp$-semidiscretization error  $\ue - G^{\bmr}_M \ue$ in $(0,\Y)$, if geometric
meshes $\calG^M_{geo,\sigma}$ and a truncation parameter $\Y \sim M$ are used.
%
\begin{lemma}[exponential convergence, \protect{ \cite[Lemma~{6.2}]{BMNOSS17_732}}]
\label{lemma:semidiscretization-error} 
Fix $c_1 < c_2$.  
Let $f \in {\mathbb H}^{-s+\nu}(\Omega)$ for some $\nu \in (0,s)$. 
Assume that $\Y$ satisfies $c_1 M \leq \Y \leq c_2 M$, and 
consider the geometric mesh $\calG^M_{geo,\sigma}$ on $(0,\Y)$
and the linear degree vector $\bmr$ with slope $\slope>0$. 
Let $\ue$ 
be given by (\ref{eq:ue-variational}) 
and $G^{\bmr}_M \ue$ be the Galerkin projection onto 
$\V^{\bmr}_{M}(\calT,\calG^M)$
given by 
(\ref{eq:semi-discrete}).  
Then there exist $C$, $b > 0$ 
(depending solely on $s$, $\calL$, $c_1$, $c_2$, $\sigma$, $\nu$, $\slope$)
such that 
\begin{equation}
\label{eq:lemma:semidiscretization-error-10}
\|\nabla (\ue - G^{\bmr}_M \ue)\|_{L^2(y^\alpha,\C)} 
\leq 
C e^{-b M} \|f\|_{{\mathbb H}^{-s+\nu}(\Omega)}
\;.
\end{equation}
Furthermore, (\ref{eq:lemma:semidiscretization-error-10}) also holds for constant polynomial
degree $\bmr  = (r,\ldots,r)$ if $c_3 M \leq r \leq c_4 M$ for some fixed $c_3$, $c_4>0$. 
The constant $b> 0$ then depends additionally on $c_3$, $c_4$. 
\end{lemma}
\begin{proof}
The statement is a slight generalization of 
\cite[Lemma 13]{BMNOSS17_732}.  
In \cite[Lemma 13]{BMNOSS17_732}, 
it is stated that the slope $\slope$ has to satisfy 
$\slope\ge \slope_{min}$ for some suitable $\slope_{min}>0$. 
Inspection of the proof shows that this condition can
be removed.  
Specifically, using 
\cite[Thm.~{8}, Eqn.~(78), Rem.~{16}]{apel-melenk17}
(or, referring alternatively to the extended preprint 
\cite[Thm.~{3.13}, Eqn.~{(3.21)}, Rem.~{3.14}]{apel-melenk17})
the result holds for any $\slope>0$, 
with constant $b = O(\slope)$ as $\slope\downarrow 0$.
The statement about the constant polynomial degree follows
from the case of the linear degree vector since 
a) $G^{\bmr}_M \ue$ is the Galerkin projection of $\ue$, 
b) the minimization property of Galerkin projections, and c)  
the fact that the space 
$S^r_{\{\Y\}}((0,\Y),\calG^M_{geo,\sigma})$ is a subspace of 
$S^{\bmr}_{\{\Y\}}((0,\Y),\calG^M_{geo,\sigma})$ provided
$\bmr$ is a linear degree vector with suitably chosen slope. 
\end{proof}
The error bound \eqref{eq:lemma:semidiscretization-error-10} 
shows that up to an exponentially small (with respect to $\Y$) 
error introduced by truncation of $(0,\infty)$ at $\Y$,
the solution $\ue$ can be approximated by the solution of
a local problem on the finite cylinder $\C_\Y$.

\subsection{Diagonalization}
\label{S:diag}
Diagonalization, as introduced in \cite{BMNOSS17_732}, refers 
to the observation that the solution $G^{\bmr}_M \ue$ of the semidiscrete 
problem (\ref{eq:semi-discrete}) 
can be expressed in terms of $\calM$ solutions $U_i\in H^1_0(\Omega)$,
of $\calM$ \emph{decoupled, linear local 2nd order reaction--diffusion problems in $\Omega$}.
As the eigenvalues $\mu_i$ in the corresponding eigenvalue problem
\eqref{eq:eigenvalue-problem} ahead govern the length scales in the
\emph{local} reaction-diffusion problems in $\Omega$ \eqref{eq:decoupled-problems}
(which, in turn, will be crucial in the mesh-design for the $hp$-FEM in $\Omega$),
it is of interest to know their asymptotic behavior.
We investigate this in Lemma~\ref{lemma:lambda} below.

Diagonalization is based on the 
explicit representation for the semidiscrete solution
$\ue_M$  obtained from the following generalized eigenvalue problem, 
introduced in \cite[Sec.~6]{BMNOSS17_732},
and proposed earlier in \cite{LynchR1964}, which reads:
find 
$(v,\mu) \in S^\bmr_{\{\Y\}}( (0,\Y), \calG^M) \setminus\{0\} \times {\mathbb R}$ 
such that 
\begin{equation}
\label{eq:eigenvalue-problem}
\forall w \in S^\bmr_{\{\Y\}}( (0,\Y), \calG^M)\;\colon\quad 
\mu \int_0^\Y y^\alpha v^\prime(y) w^\prime(y)\, \diff y = 
        \int_0^\Y y^\alpha v(y) w(y)\, \diff y \;. 
\end{equation} 
%
All eigenvalues $(\mu_i)_{i=1}^{\calM}$ of \eqref{eq:eigenvalue-problem}
are positive and $S^\bmr_{\{\Y\}}((0,\Y),\calG^M)$ 
has an orthonormal eigenbasis $(v_i)_{i=1}^\calM$ satisfying 
\begin{equation}
\label{eq:eigenbasis-normal}
\int_0^\Y y^\alpha v_i^\prime(y) v_j^\prime(y)\, \diff y = \delta_{i,j}, 
\qquad 
\int_0^\Y y^\alpha v_i(y) v_j(y)\, \diff y = \mu_i \delta_{i,j}. 
\end{equation}
We may expand the semidiscrete approximation $G^{\bmr}_M \ue$ as 
\begin{equation}\label{eq:U_M}
G^{\bmr}_M \ue(x',y)=: \sum_{i=1}^\calM U_i(x') v_i(y).
\end{equation}
The coefficient functions $U_i\in H^1_0(\Omega)$ 
satisfy a system of $\calM$ \emph{decoupled} linear reaction-diffusion 
equations in $\Omega$: 
for $i=1,\ldots,\calM$, find $U_i \in H^1_0(\Omega)$ such that
\begin{align}
\label{eq:decoupled-problems}
\forall V \in H^1_0(\Omega)\;\colon\quad 
\blfa{\mu_i,\Omega}(U_i,V) = d_s v_i(0) \langle f,V\rangle\;.
\end{align}
Here $v_i$ denotes the $i$-th eigenfunction of the eigenvalue problem 
\eqref{eq:eigenvalue-problem}, \eqref{eq:eigenbasis-normal}
and 
\begin{equation}
\label{eq:blfa-sg-per}
\blfa{\mu_i,\Omega}(U,V):= \mu_i \blfa{\Omega}(U,V) + \int_\Omega U V \diff x',
\end{equation}
with $\blfa{\Omega}$ as introduced in \eqref{eq:blfOmega}.
Due to the biorthogonality \eqref{eq:eigenbasis-normal} of 
the discrete eigenfunctions 
$v_i \in S^\bmr_{\{\Y\}}((0,\Y),\calG^M)$,
any 
$Z(x',y)= \sum_{i=1}^\calM V_i(x') v_i(y)$ with arbitrary 
$V_i\in H^1_0(\Omega)$ satisfies the energy (``Pythagoras'') identities
\begin{equation}
\label{eq:pythagoras}
\blfa{\C}(Z,Z)  = 
\blfa{\C_\Y}(Z,Z)  = 
\sum_{i=1}^\calM \|V_i\|^2_{\mu_i,\Omega}, 
\qquad \|V_i\|^2_{\mu_i,\Omega}:= \blfa{\mu_i,\Omega}(V_i,V_i). 
\end{equation}

The following bounds on the $\mu_i$ were shown in \cite[Lemma~{14}]{BMNOSS17_732}
for the special case of geometric meshes $\calG^M_{geo,\sigma}$ and linear degree
vectors: 
\begin{lemma}[properties of the eigenpairs, \protect{\cite[Lemma~{14}]{BMNOSS17_732}}]
\label{lemma:lambda}
Let $\{\calG^M_{geo,\sigma}\}_{M\geq 1}$ 
be a sequence of geometric meshes on $(0,\Y)$ and 
$\bmr$ a linear polynomial degree vector with slope $\slope>0$. 

Assume that the truncation parameter $\Y$ is chosen so that
$c_1 M\le \Y \le c_2 M$ for some constants $0<c_1, c_2 < \infty$  
that are independent of $M$.
	
Then,
there exists $C>1$ (depending on $c_1$, $c_2$ and on $\sigma\in (0,1)$) 
such that there holds for for every $M \in \bbN$ 
for the eigenpairs $(\mu_i,v_i)_{i=1}^{\calM_{geo}}$ given by 
\eqref{eq:eigenvalue-problem}, \eqref{eq:eigenbasis-normal} 
\begin{equation}\label{eq:RDevpBds}
\|v_i\|_{L^\infty(0,\Y)} \leq C M^{(1-\alpha)/2}, 
\qquad  C^{-1} (\slope^{-2} M^{-2} \sigma^{M})^2 \leq \mu_i \leq C M^2. 
\end{equation}
\end{lemma}

\subsection{Fully discrete approximation} \label{sec:fuldiscr}
The full discretization is obtained by approximating the functions 
$U_i$ of (\ref{eq:decoupled-problems}) from finite-dimensional spaces. 
Let $\calT_i$, $i=1,\ldots,\calM$, be regular triangulations in $\Omega$ 
and $q \in \bbN$. 
Let $\Pi^q_i:H^1_0(\Omega) \rightarrow S^q_0(\Omega,\calT_i)$ 
denote the Ritz projectors for the bilinear forms $\blfa{\mu_i,\Omega}$, 
which are characterized by 
\begin{equation}
\label{eq:galerkin-for-vi} 
\blfa{\mu_i,\Omega}(u - \Pi^q_i u,v) = 0 \qquad \forall v \in S^q_0(\Omega,\calT_i). 
\end{equation}

In terms of the projections $\Pi^q_i$ we can define the fully discrete approximation 
\begin{equation}
\label{eq:fully-discrete-1}
\ue_{h,M}(x,y):= \sum_{i=1}^{\calM} v_i(y) \Pi^q_i U_i(x). 
\end{equation}
By combining 
(\ref{eq:decoupled-problems}) and (\ref{eq:galerkin-for-vi}),
the functions $\Pi^q_i U_i \in S^q_0(\Omega, \calT_i)$ 
are explicitly and computably given as the solutions of 
\begin{equation}
\label{eq:PiqiUi}
\forall V \in S^q_0(\Omega,\calT_i)\ \colon\ 
a_{\mu_i,\Omega} (\Pi^q_i U_i, V) = d_s v_i(0) \langle f, V\rangle. 
\end{equation}
In view of (\ref{eq:pythagoras}), we have the following representation of the difference
between the semidiscrete approximation $G^\bmr_{M} \ue$ and the fully discrete approximation 
$\ue_{h,M}$:%
\begin{lemma}
\label{lemma:diagonalization-error}
Let $\ue_{h,M}$ be given by (\ref{eq:fully-discrete-1}). Then: 
\begin{equation}
\label{eq:diagonalization-error}
a_\C(G^\bmr_M \ue - \ue_{h,M}, G^\bmr_M \ue - \ue_{h,M}) 
 = \sum_{i=1}^{\calM} \|U_i - \Pi^q_i U_i\|^2_{\mu_i,\Omega}. 
\end{equation}
\end{lemma}
Concerning the meshes $\calT_i$, we distinguish two cases in this work: 
\begin{itemize}
\item[
{\bf Case A}:] 
The meshes $\calT_i$, $i=1,\ldots,\calM$, possibly differ from each other. 
\item[{\bf Case B}:] The meshes $\calT_i$, $i=1,\ldots,\calM$, coincide. 
That is, all coefficient functions $U_i$ in the semidiscrete solution 
\eqref{eq:U_M}
are approximated from one common $hp$-FE space $S^q_0(\Omega,\calT)$. 
\end{itemize}
In {\bf Case B} the 
approximation $\ue_{h,M}$ actually coincides with the Galerkin projection 
$G^{q,\bmr}_{h,M} \ue \in \V^{q,\bmr}_{h,M}(\calT,\calG^M) 
 = 
 S^q_0(\Omega,\calT) \otimes S^{\bmr}_{\{\Y\}}((0,\Y),\calG^M)$: 
\begin{lemma}[error representation, \protect{\cite[Lemma~{12}]{BMNOSS17_732}}]
\label{lemma:decoupled-problems}
Let $(\mu_i,v_i)_{i=1}^\calM$ be the eigenpairs given by 
\eqref{eq:eigenvalue-problem}, \eqref{eq:eigenbasis-normal}. 
For $i=1,\ldots,\calM$, 
let $U_i \in H^1_0(\Omega)$ be the solutions to \eqref{eq:decoupled-problems}.
Consider {\bf Case B} and 
let $\Pi^q_i:H^1_0(\Omega) \rightarrow S^q_0(\Omega,\calT)$ 
be the Galerkin projections given as in \eqref{eq:galerkin-for-vi},
with one common, regular triangulation $\calT$ of $\Omega$ for $i=1,\ldots,\calM$.
Let $G^{\bmr}_M \ue$ denote the solution to the semidiscrete problem 
\eqref{eq:semi-discrete}. 
Then the tensor product Galerkin approximation 
$G^{q,\bmr}_{h,M} \ue \in \V^{q,\bmr}_{h,M}(\calT,\calG^M) = S^q_0(\Omega,\calT) \otimes S^{\bmr}_{\{\Y\}}((0,\Y),\calG^M)$ 
satisfies 
\begin{align}
\label{eq:representation-by-reaction-diffusion-problems}
\ue_{h,M}(x',y) & = \sum_{i=1}^{\calM} v_i(y) \Pi^q_i U_i(x') , 
\\
\label{eq:error-representation}
\blfa{\C}(G^{\bmr}_{M}\ue - G^{q,\bmr}_{h,M} \ue, G^{\bmr}_{M} \ue - G^{q,\bmr}_{h,M} \ue) 
&= 
\sum_{i=1}^{\calM} \|U_i - \Pi^q_i U_i\|^2_{\mu_i,\Omega}. 
\end{align}
\end{lemma}
%
Lemma~\ref{lemma:decoupled-problems} shows that in {\bf Case B}, the Galerkin projection of $\ue$ 
into the tensor product space $\V^{q,\bmr}_{h,M}(\calT,\calG^M)$ 
coincides with the approximation $\ue_{h,M}$ defined in (\ref{eq:fully-discrete-1}) in terms of the decoupling
procedure. Hence, the decoupling
procedure is not essential for numerical purposes in {\bf Case B}, although it has algorithmic
advantages. In contrast, {\bf Case A} relies on the decoupling in an essential way. In both cases, 
the exponenial convergence result below will make use of the error estimates 
of Lemmas~\ref{lemma:diagonalization-error}, \ref{lemma:decoupled-problems} obtained by the 
diagonalization process. 

It is advisable to choose the spaces $\calT_i$ in case {\bf Case A} such that the functions 
$U_i$ can be approximated well from $S^q_0(\Omega, \calT_i)$ in the norm $\|\cdot\|_{\mu_i,\Omega}$. 
Correspondingly in {\bf Case B}, the commmon space $\calT$ should be chosen such that 
each $U_i$ can be approximated well from $S^q_0(\Omega,\calT)$. The bounds \eqref{eq:RDevpBds} 
indicate that, for large $M$, most of the reaction-diffusion problems \eqref{eq:decoupled-problems} 
are singularly perturbed.
Hence we design in the following Section~\ref{S:approx-sing-perturb}
$hp$-FE approximation spaces in $\Omega$ 
which afford exponential convergence rates that are \emph{robust} with respect to
the singular perturbation parameter. 
\section{$hp$-FE Approximation of singular perturbation problems}
\label{S:approx-sing-perturb}
In the exponential convergence rate analysis of tensorized 
$hp$-FEM for the CS extension (\EhpFEMp) 
as well as for the ensuing (see Section~\ref{S:ExpConvIII} ahead)
\BK approximation, 
a crucial role is played by \emph{robust exponential convergence 
rate bounds for $hp$-FEM for singularly perturbed, reaction-diffusion 
problems} in curvilinear polygonal domains $\Omega$.
Specifically, we consider the $hp$-FE approximation of 
the local reaction-diffusion problem in $\Omega$,
\begin{equation} \label{eq:sg-per}
-\varepsilon^2 \operatorname{div} \left( A(x') \nabla u^\varepsilon\right)
+ c(x') u^\varepsilon 
= f \quad \mbox{ in $\Omega$}, 
\qquad u^\varepsilon = 0 \quad\mbox{ on $\partial\Omega$}, 
\end{equation}
where we assume  $\operatorname*{ess\,inf}_{x'\in \Omega} c(x') \geq c_0 > 0$ and
\begin{equation}\label{eq:sg-per-asmp}
  \begin{aligned}    
&\text{ $A$, $c$, and $f$ are analytic on $\overline{\Omega}$ and}\\
&\text{$A$ is symmetric, uniformly positive definite.}
  \end{aligned}
\end{equation}
We note again that \eqref{eq:sg-per} does not imply 
any kind of boundary compatibility of $f$ at $\partial \Omega$
(cf.\ Remark~\ref{remk:compatibility}). 
We assume $\Omega$ to be scaled so that ${\rm diam}(\Omega) = O(1)$.
Then, for small $\varepsilon >0$, the boundary value problem \eqref{eq:sg-per} 
is a so-called ``elliptic-elliptic'' singular perturbation problem. 
Under the assumptions \eqref{eq:sg-per-asmp}, for every $\varepsilon>0$
problem \eqref{eq:sg-per} admits a unique solution $u^\varepsilon\in H^1_0(\Omega)$.
In general, $u^\varepsilon$ exhibits, for small $\varepsilon > 0$, 
\emph{boundary layers near $\partial\Omega$} 
whose robust numerical resolution (i.e., with error bounds whose constants are
independent of $\varepsilon$)
requires \emph{anisotropically refined meshes aligned with $\partial \Omega$}
(see \cite{RoosStynsTobiska2ndEd,melenk-schwab98,FstmnMM_hpBalNrm2017} 
 and the references there).
In addition, the corners of $\Omega$ induce point singularities 
in the (analytic in $\Omega$) solution $u^\varepsilon$. 
In the context of $hp$-FEM under consideration here,
their efficient numerical approximation mandates 
\emph{geometric mesh refinement near the corners}. 

In the present section, we consider the $hp$-FEM approximation of 
$u^\varepsilon$ that features exponential convergence for two different
types of meshes: 
a) \emph{geometric boundary layer meshes} in 
Section~\ref{sec:sing-approx-geo-mesh} and 
b) \emph{admissible boundary layer meshes}  in Section~\ref{sec:sing-approx-admissible-meshes}. 
In both cases the error estimates are of the form $O(e^{-b q} + e^{-b' L})$, 
with the constant hidden in $O(\cdot)$ independent of $\eps$, $L$, and $q$ and where 
$q$ is the polynomial degree employed and $L$ measures the number of layers
of geometric refinement towards the vertices or edges of $\Omega$. 
The difference in these two types of meshes is that ``admissible boundary layer meshes'' 
are strongly $\varepsilon$-dependent with geometric refinement towards the vertices
and only a \emph{single} layer of thin elements of width $O(q\varepsilon)$ 
near $\partial\Omega$ to resolve the boundary layer. 
The number of elements is then $O(L)$ leading to a number of degrees of freedom 
$N = O(L q^2)$. 
In contrast, {geometric boundary layer meshes} are based
on geometric, anisotropic refinement towards the edges and corners
of $\Omega$. As we show in \cite{banjai-melenk-schwab19-RD}, 
$hp$-FEM on such meshes afford exponential convergence for boundary layers with 
multiple scales.
The total number of elements in geometric boundary layer 
meshes with $L$ layers is $O(L^2)$.
Combined with local FE spaces of polynomial degree $q$,
this results in a number of degrees of freedom $N = O(L^2 q^2)$.  
Whereas \emph{admissible boundary layer meshes} are designed to approximate 
boundary layers of a single, given length scale $\eps$, 
geometric boundary layer meshes afford concurrent, robust and
exponentially convergent approximations of boundary layers with multiple 
length scales in $\Omega$. 
These arise, e.g.,  upon semidiscretization in the extended variable 
as is evident from \eqref{eq:decoupled-problems}.

\subsection{Macro triangulation. Geometric boundary layer mesh}
\label{sec:geo-bdy-layer-mesh}
\begin{figure}[th]
\psfragscanon
\psfrag{T}{$T$}
\psfrag{B}{$B$}
\psfrag{M}{$M$}
\psfrag{C}{$C$}
\begin{overpic}[width=0.25\textwidth]{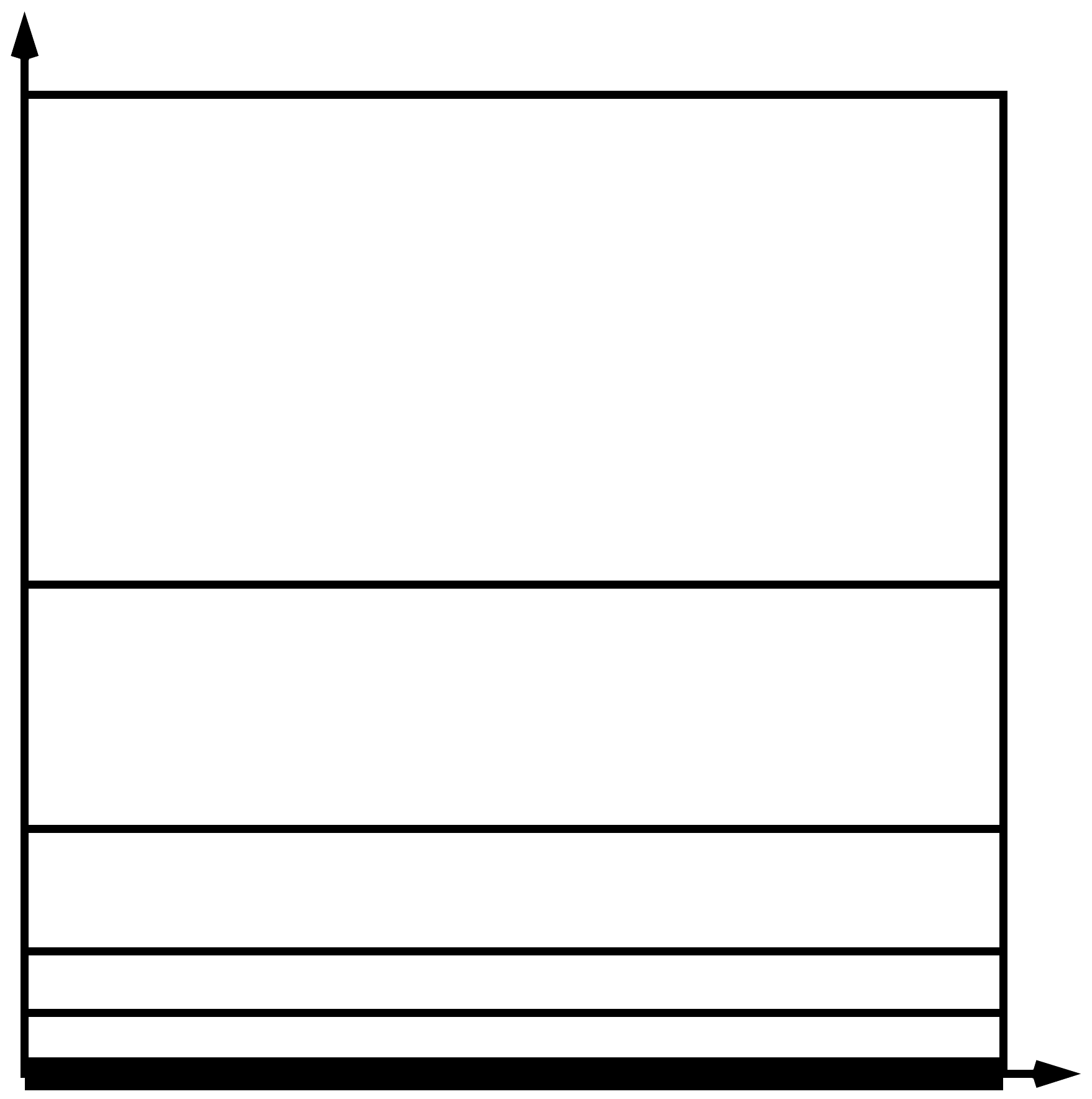}
\put(45,70){$ \check{\calT}^{\BL,L}_{geo,\sigma}$}
\put(100,-5){\small $\xh$}
\put(-7,95){\small $\yh$}
\end{overpic}
\hfill 
\begin{overpic}[width=0.25\textwidth]{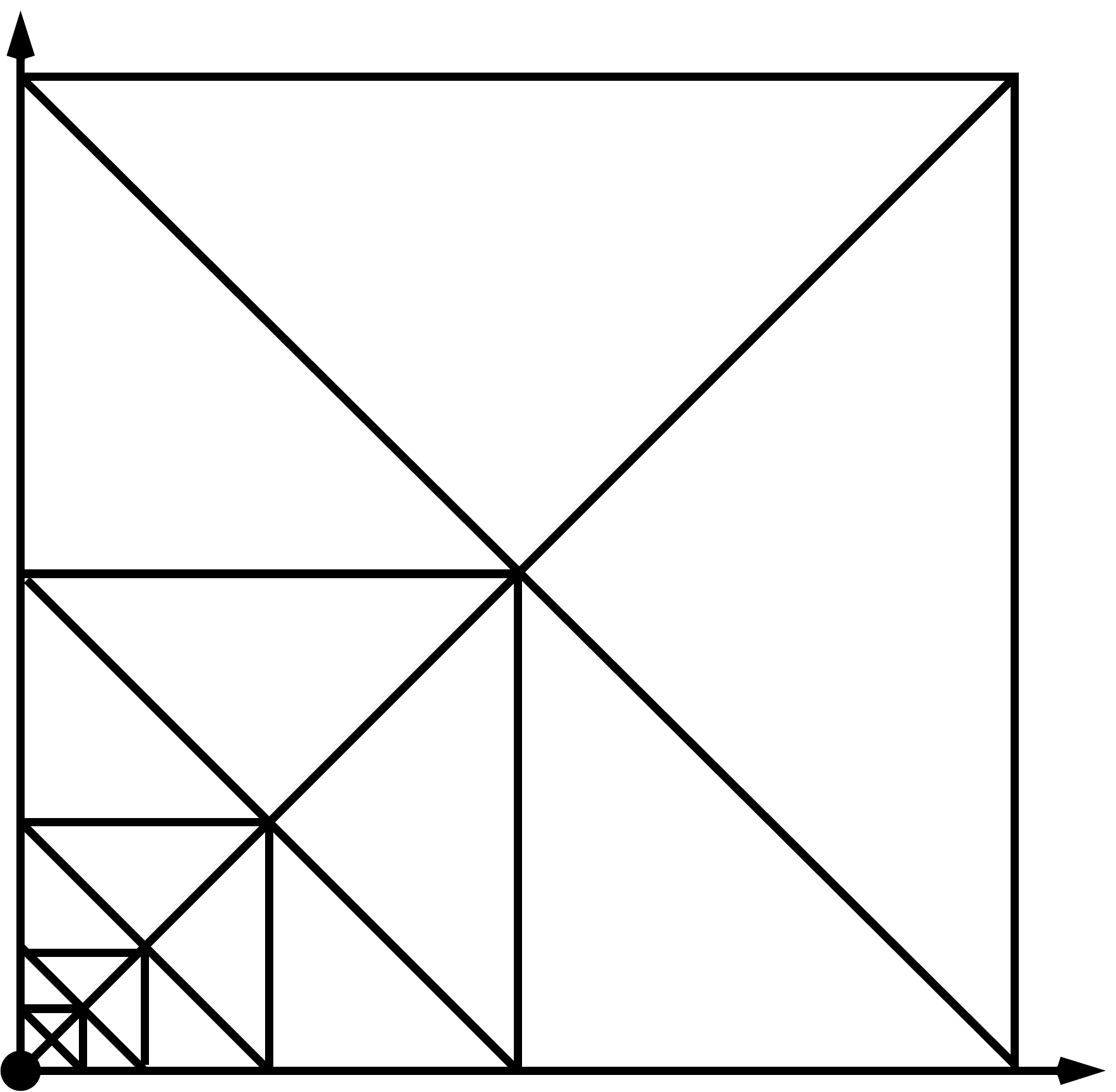}
\put(45,70){$ \check{\calT}^{\Co,n}_{geo,\sigma}$}
\put(100,-5){\small $\xh$}
\put(-7,95){\small $\yh$}
\end{overpic}
\hfill 
\begin{overpic}[width=0.25\textwidth]{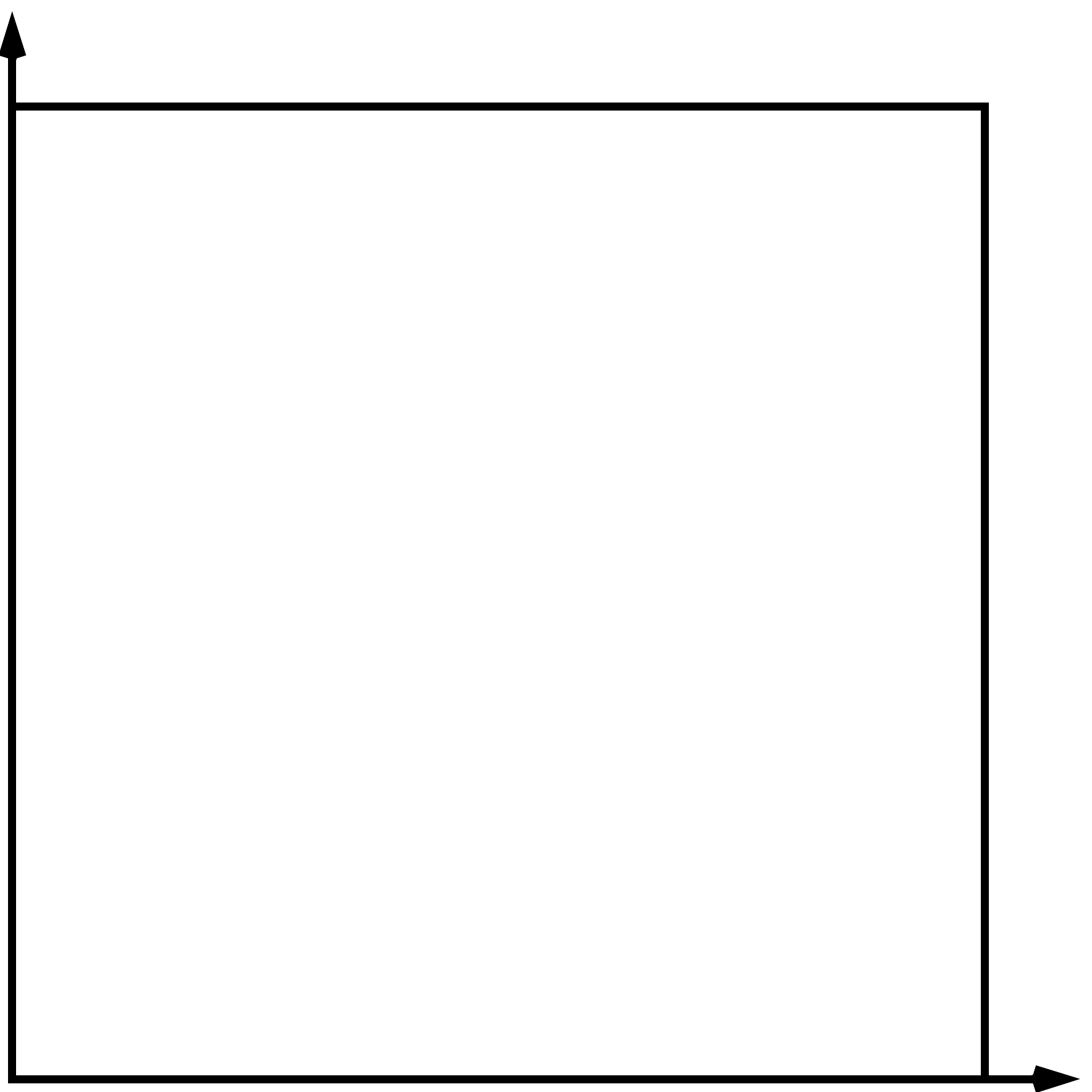}
\put(25,70){trivial patch}
\put(100,-5){\small $\xh$}
\put(-7,95){\small $\yh$}
\end{overpic}
\begin{overpic}[width=0.45\textwidth]{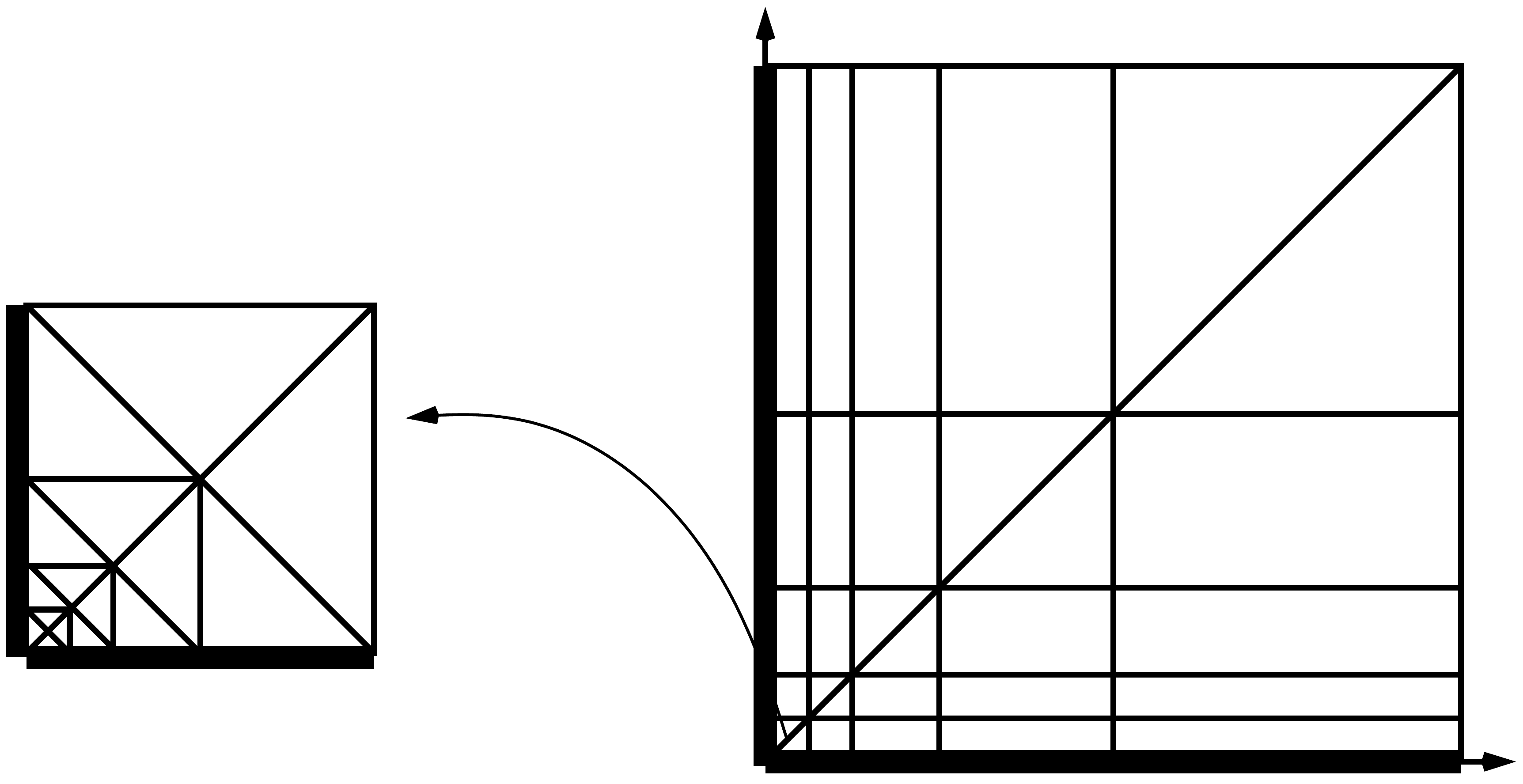}
\put(70,40){$ \check{\calT}^{\Te,L,n}_{geo,\sigma}$}
\put(0,0){$ \widehat S_1 = (0,\sigma^L)^2$}
\put(100,-5){\small $\xh$}
\put(45,45){\small $\yh$}
\end{overpic}
\hfill 
\begin{overpic}[width=0.45\textwidth]{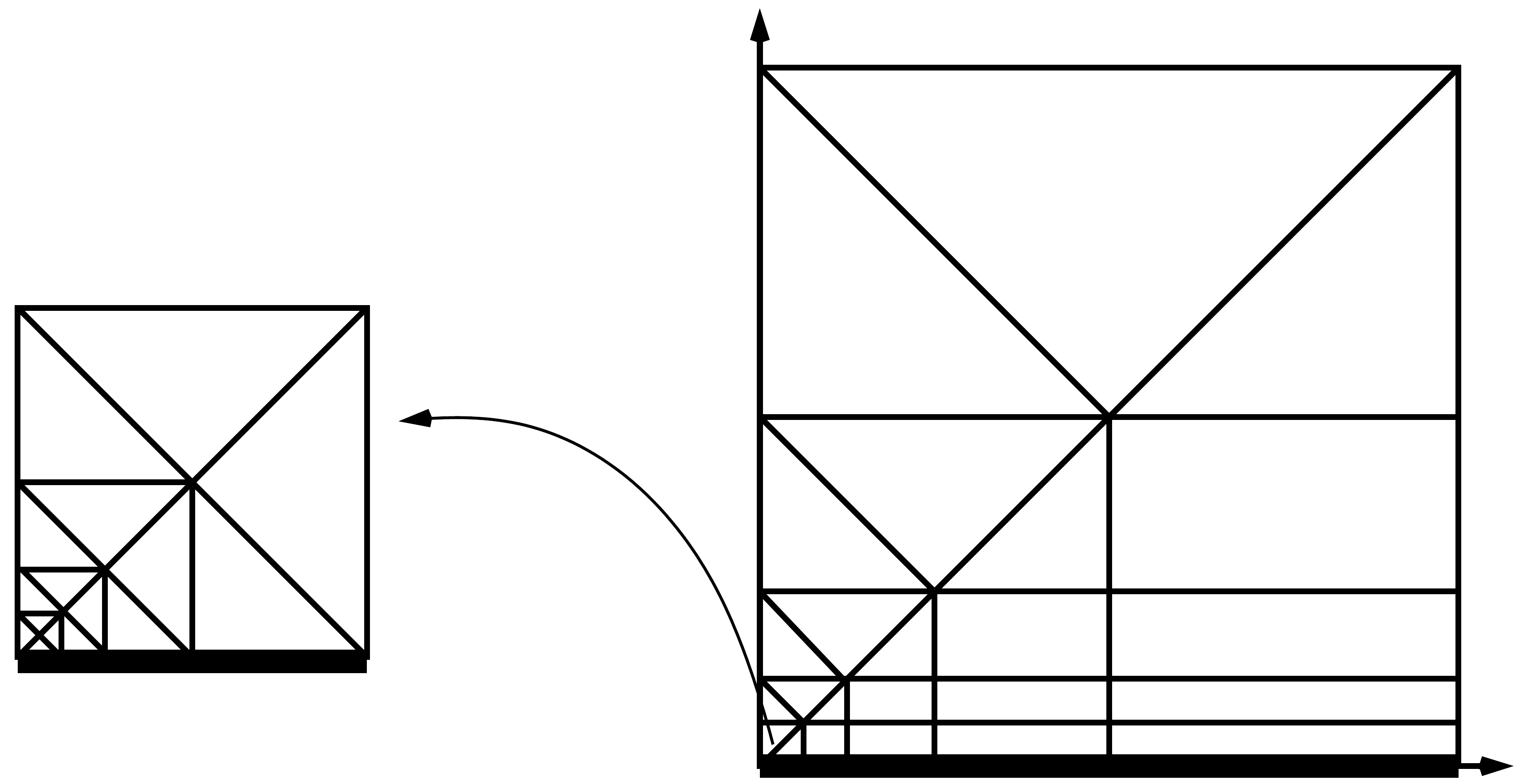}
\put(70,40){$ \check{\calT}^{\Mi,L,n}_{geo,\sigma}$}
\put(0,0){$ \widehat S_1 = (0,\sigma^L)^2$}
\put(100,-5){\small $\xh$}
\put(45,45){\small $\yh$}
\end{overpic}
\psfragscanoff
\caption{
\label{fig:patches} 
Catalog ${\mathfrak P}$ of reference refinement patterns from 
\cite{banjai-melenk-schwab19-RD}.
        Top row: reference boundary layer patch $\check{\calT}^{\BL,L}_{geo,\sigma}$
with $L$ layers of geometric refinement towards $\{\yh=0\}$;
        reference corner patch $\check{\calT}^{\Co,n}_{geo,\sigma}$
with $n$ layers of geometric refinement towards $(0,0)$; trivial patch.
Bottom row:
        reference tensor patch $\check{\calT}^{\Te,L,n}_{geo,\sigma}$
with $n$ layers of refinement towards $(0,0)$ and $L$ layers of refinement towards $\{\xh = 0\}$ and $\{\yh=0\}$;
        reference mixed patch $\check{\calT}^{\Mi,L,n}_{geo,\sigma}$
with $L$ layers of refinement towards $\{\yh=0\}$ and $n$ layers of refinement towards $(0,0)$.
Geometric entities shown in boldface indicate parts of $\partial \widehat S$
that are mapped to $\partial\Omega$.
These patch meshes are transported into the curvilinear polygon $\Omega$ shown
in Fig.~\ref{fig:curvilinear-polygon} via analytic patch maps $F_{K^{\M}}$.
}%
\end{figure}
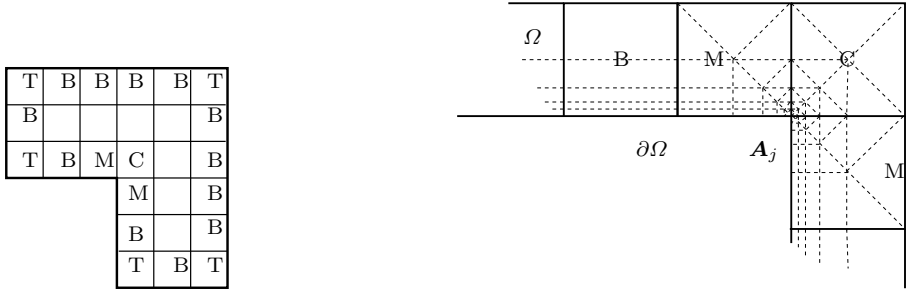
\begin{figure}[th]
\begin{overpic}[width=0.25\textwidth]{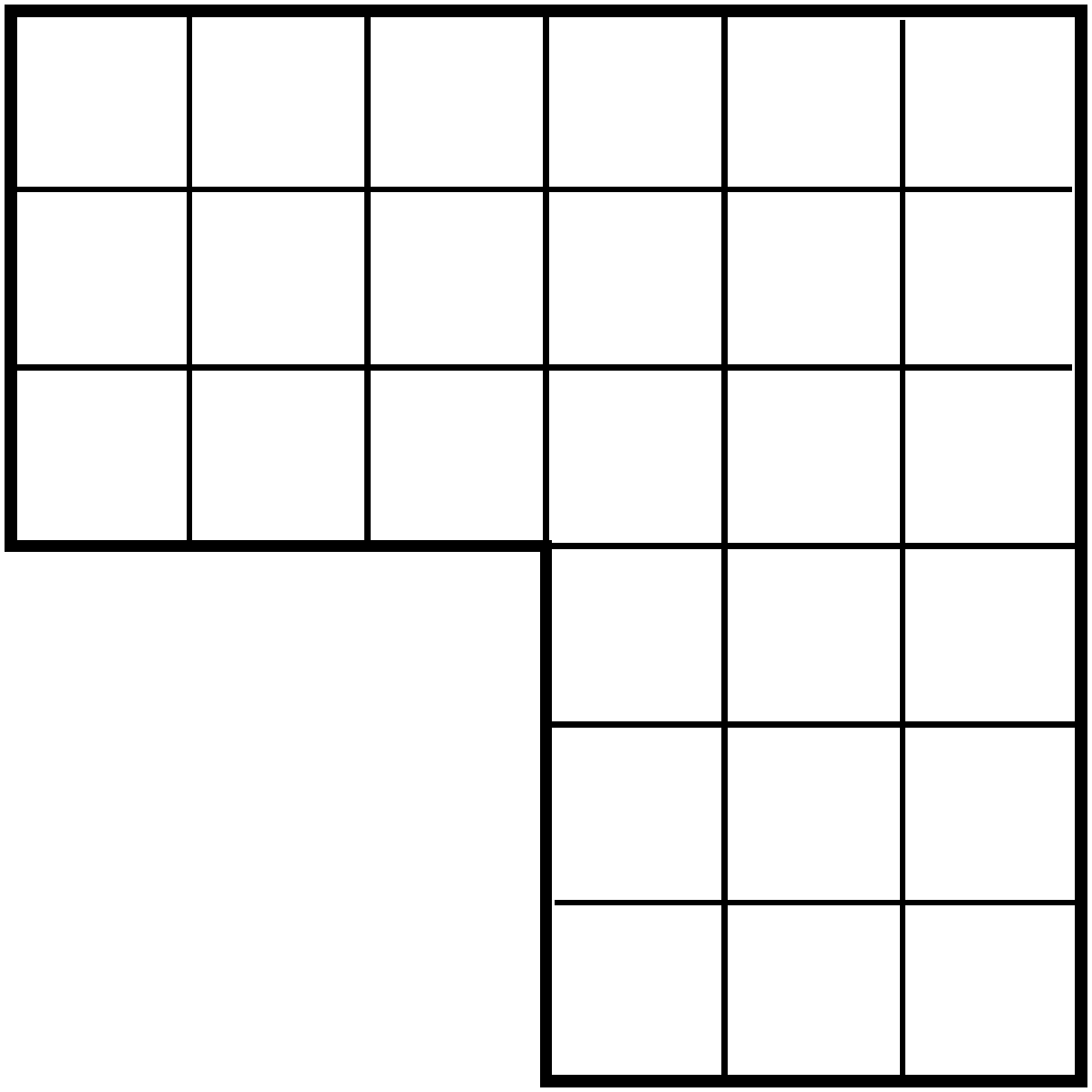}
\put(8,55){T}
\put(25,55){B}
\put(40,55){M}
\put(55,55){C}
\put(55,40){M}
\put(55,22){B}
\put(55,08){T}
\put(75,08){B}
\put(90,08){T}
\put(90,25){B}
\put(90,40){B}
\put(90,55){B}
\put(90,75){B}
\put(90,90){T}
\put(75,90){B}
\put(55,90){B}
\put(40,90){B}
\put(25,90){B}
\put(08,90){T}
\put(08,75){B}
\end{overpic}
\hfill 
\begin{overpic}[width=0.5\textwidth]{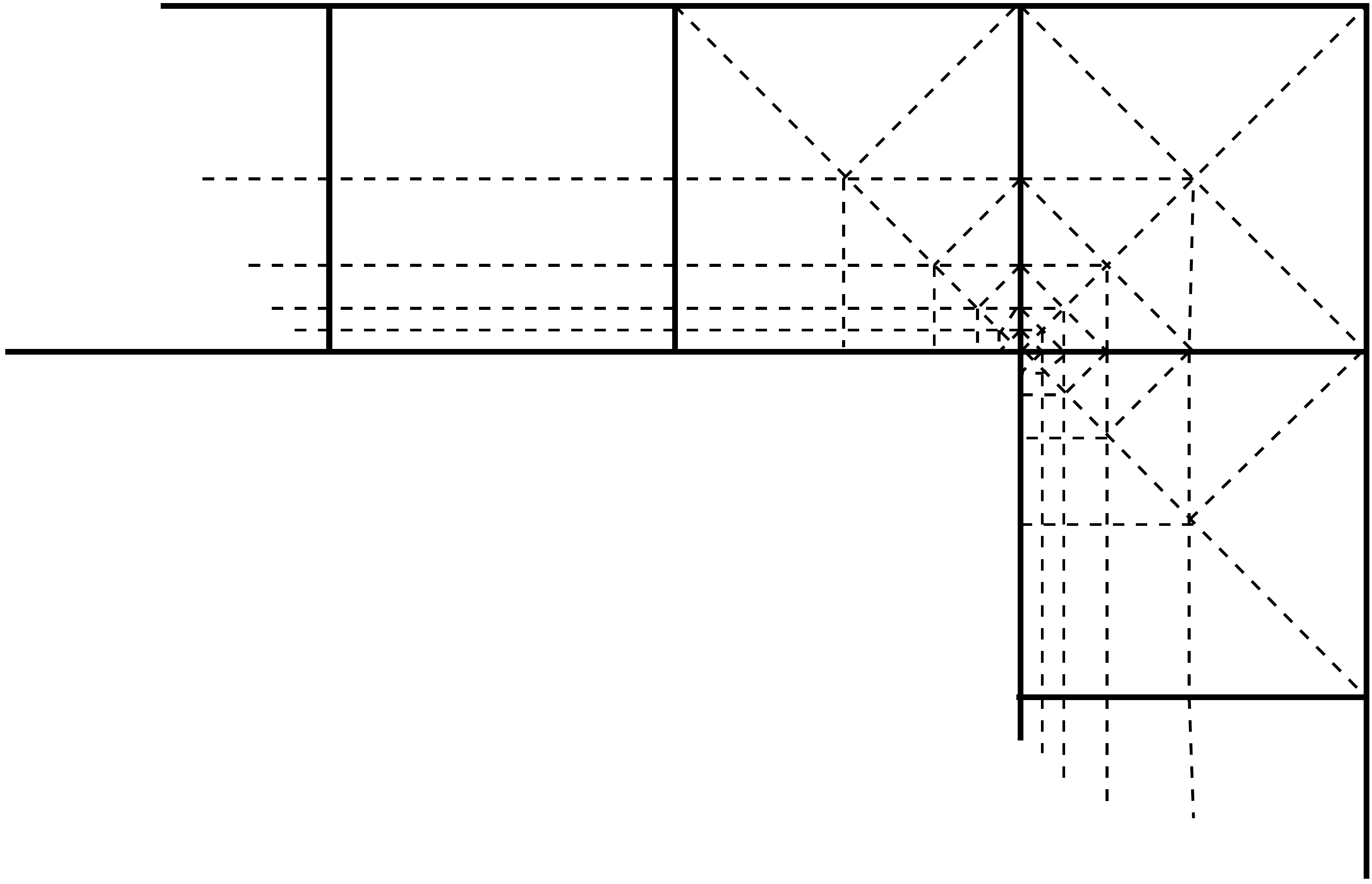}
\put(15,55){$\Omega$}
\put(35,50){B}
\put(55,50){M}
\put(85,50){C}
\put(95,25){M}
\put(40,30){$\partial\Omega$}
\put(65,30){$\bA_j$}
\end{overpic}

\caption{
\label{fig:patch-examples} 
Patch arrangement in $\Omega$ \cite{banjai-melenk-schwab19-RD}.
Left panel: example of L-shaped domain decomposed into 27 patches 
($T$, $B$, $M$, $C$ indicate Tensor, Boundary layer, Mixed, Corner patches,
empty squares stand for trivial patches). 
Right panel: Zoom-in near the reentrant corner $\bA_j$. 
Solid lines indicate patch boundaries, dashed lines mesh lines.
}
\end{figure}
We do not consider the most general meshes with anisotropic
refinement, but confine the $hp$-FE approximation theory 
to meshes generated as push-forwards of a small number of so-called
\emph{mesh patches}. 
This concept was used in the error analysis of 
$hp$-FEM for singular perturbations in \cite[Sec.~{3.3.3}]{melenk02} 
and in \cite{FstmnMM_hpBalNrm2017}. 
Specifically, we assume given a \emph{fixed macro-triangulation}
${\mathcal T}^{\M} = \{K^{\M} \,|\, K^{\M} \in {\mathcal T}^{\M}\}$
of $\Omega$ consisting of curvilinear quadrilaterals $K^{\M}$
with analytic \emph{patch maps} (to be distinguished from the
actual element maps) 
$F_{K^{\M}}:\widehat S=(0,1)^2 \rightarrow K^{\M}$
that satisfy the usual compatibility conditions.
I.e., ${\mathcal T}^{\M}$ does not have hanging nodes and, 
for any two distinct elements $K_1^{\M}, K_2^{\M} \in {\mathcal T}^{\M}$ that
share an edge $e$, their respective element maps induce compatible parametrizations
of $e$ (cf., e.g.,  \cite[Def.~{2.4.1}]{melenk02} for the precise conditions).
Each element of the fixed macro-triangulation ${\mathcal T}^{\M}$
is further subdivided according to one of the 
refinement patterns in Definition~\ref{def:admissible-patterns}
(see also \cite[Sec.~{3.3.3}]{melenk02} or \cite{FstmnMM_hpBalNrm2017}). 
The actual triangulation is then obtained by transplanting refinement patterns 
on the square reference patch into the physical domain $\Omega$
by means of the element maps $F_{K^{\M}}$ of the macro-triangulation. 
That is, 
for any element $K\in {\mathcal T}$, the element map $F_K$ is the 
concatenation of an affine map---which realizes the mapping from the 
reference square or triangle to the elements 
in the patch refinement pattern and will be denoted by $A_K$---
and the patch map (which will be denoted by $F_{K^\M}$), 
i.e., $F_K = F_{K^\M} \circ A_K : \hat K \rightarrow K$. 

The following refinement patterns 
were introduced in \cite[Def.~{2.1}, 2.3]{banjai-melenk-schwab19-RD}.
They are based on geometric refinement towards a vertex and/or an edge; 
the integer $L$ controls
the number of layers of refinement towards an edge whereas 
$ n \in \bbN$ measures the refinement towards a vertex. 
\begin{definition}[Catalog ${\mathfrak P}$ of refinement patterns,
\protect{\cite[Def.~{2.1}]{banjai-melenk-schwab19-RD}}]
\label{def:admissible-patterns}
Given $\sigma \in (0,1)$, $L$, $n \in {\mathbb N}_0$ with $n \ge L$
the catalog ${\mathfrak P}$ consists of the following patterns:
\begin{enumerate}
\item
The \emph{trivial patch}: The reference square $\widehat S = (0,1)^2$ is not further refined. 
The corresponding triangulation of $\widehat S$ consists of the single element: 
$\check{\mathcal T}^{trivial} = \{\widehat S\}$. 
\item
The \emph{geometric boundary layer patch $\check{\calT}^{\BL,L}_{geo,\sigma}$}: 
$\widehat S$ is refined anisotropically towards 
$\{\yh =0\}$ into $L$ elements as depicted
in Fig.~\ref{fig:patches} (top left). 
The mesh 
$\check{\calT}^{\BL,L}_{geo,\sigma}$ 
is characterized by the nodes $(0,0)$, $(0,\sigma^i)$, $(1,\sigma^i)$,
$i=0,\ldots,L$ and the corresponding rectangular elements generated by these nodes.
\item
The \emph{geometric corner patch $\check{\calT}^{\Co,n}_{geo,\sigma}$}: 
$\widehat S $ is refined isotropically towards $(0,0)$ as depicted 
in Fig.~\ref{fig:patches} (top middle). 
Specifically,
the reference geometric corner patch mesh $\check{\calT}^{C,n}_{geo,\sigma}$
in $\widehat S$ with geometric refinement towards $(0,0)$ and $n$ layers
is given by triangles determined by the nodes
$(0,0)$, and $(0,\sigma^i)$, $(\sigma^i,0)$,
$(\sigma^i,\sigma^i)$, $i=0,1,\ldots,n$. 
\item
The \emph{tensor product patch $\check{\calT}^{\Te,L,n}_{geo,\sigma}$}: 
$\widehat S$ is triangulated in $\widehat S_1:= (0,\sigma^L)^2$ and 
$\widehat S_2:= \widehat S \setminus \widehat S_1$ separately as 
depicted in Fig.~\ref{fig:patches} (bottom left). The triangulation of $\widehat S_1$
is a scaled version of $\check{\calT}^{\Co,n-L}_{geo,\sigma}$ characterized by the 
nodes $(0,0)$, $(0,\sigma^i)$, $(\sigma^i,0)$, $(\sigma^i,\sigma^i)$, $i=L,\ldots,n$. 
The triangulation of 
$\widehat S_2$ is characterized by the nodes 
$(\sigma^i,\sigma^j)$, $i$, $j=0,\ldots,L$. 
\item
The \emph{mixed patches $\check{\calT}^{\Mi,L,n}_{geo,\sigma}$}:
The triangulation consists of both anisotropic elements and isotropic elements
as depicted in Fig.~\ref{fig:patches} (bottom right) and is obtained by triangulating
the regions 
$\widehat S_1:= (0,\sigma^L)^2$, 
$\widehat S_2:= \bigl( \widehat S \setminus \widehat S_1\bigr)  
\cap \{\yh \leq \xh\}$, 
$\widehat S_3:= \widehat S \setminus \bigl (\widehat S_1 \cup \widehat S_2\bigr)$
separately.  
$\widehat S_1$ is a scaled version of $\check{\calT}^{\Co,n-L}_{geo,\sigma}$ 
characterized by the nodes 
$(0,0)$, $(0,\sigma^i)$, $(\sigma^i,0)$, $(\sigma^i, \sigma^i)$, $i=L,\ldots,n$. 
The triangulation of $\widehat S_2$ is given by the nodes 
$(\sigma^i,0)$, $(\sigma^i,\sigma^{j})$, $0 \leq i \leq L$, $i \leq j \leq L$
and consists of rectangles and triangles, and only the triangles abutt on the 
diagonal $\{\xh=\yh\}$. The triangulation of $\widehat S_3$ consists of triangles only 
given by the nodes 
$(0,\sigma^i)$, $(\sigma^i,\sigma^i)$, $i=0,\ldots,L$. 
\end{enumerate}
\end{definition}
\begin{remark}
\label{remk:further-refinement-patterns} 
We kept the list of possible patch refinement patterns in 
Definition~\ref{def:admissible-patterns} small in order
to reduce the number of cases to be discussed for the $hp$-FE error bounds.
A larger number of refinement patterns could 
facilitate greater flexibility in mesh generation. 
In particular, 
the reference patch meshes do not contain general quadrilaterals 
but only (axiparallel) rectangles; 
this restriction is not essential 
but leads to some simplifications in the $hp$-FE error analysis in \cite{banjai-melenk-schwab19-RD}. 

The addition of the diagonal line in the reference corner,
tensor, and mixed patches is done 
to be able to apply the regularity theory of \cite{melenk02} and 
probably not necessary in actual computations. We also mention that with 
additional constraints on the macro triangulation $\calT^\M$ 
the diagonal line could be dispensed with, \cite{banjai-melenk-schwab19-RD}. 
\eremk
\end{remark}

The following definition of the geometric boundary layer mesh $\Tg$
formalizes the requirement on the meshes that anisotropic refinement towards $\partial\Omega$
is needed as well as geometric refinement towards the corners. 
\begin{definition} [geometric boundary layer mesh, \protect{\cite[Def.~{2.3}]{banjai-melenk-schwab19-RD}}]
\label{def:bdylayer-mesh} 
Let ${\mathcal T}^{\M}$ be a fixed macro-triangulation consisting of quadrilaterals with 
analytic element maps
that satisfies \cite[Def.~{2.4.1}]{melenk02}. 

Given $\sigma \in (0,1)$, $L$, $n \in {\mathbb N}_0$ with $n \ge L$, 
a mesh $\Tg$ is called a \emph{geometric boundary layer mesh} 
if the following conditions hold:
\begin{enumerate}
\item 
$\Tg$ is obtained by refining each element $K^{\M} \in {\mathcal T}^{\M}$
according to the finite catalog ${\mathfrak P}$ 
of structured patch-refinement patterns specified in Definition~\ref{def:admissible-patterns}, 
governed by the parameters $\sigma$, $L$, and $n$.
\item 
$\Tg$ is a regular triangulation of $\Omega$, i.e., 
it does not have hanging nodes. 
Since the element maps for the refinement patterns
are assumed to be affine, this requirement ensures that the 
resulting triangulation satisfies 
\cite[Def.~{2.4.1}]{melenk02}.  
\end{enumerate}
For each macro-patch $K^\M \in {\mathcal T}^{\M}$, 
exactly one of the following cases is possible: 
\begin{enumerate}
\setcounter{enumi}{2}
\item 
\label{item:def-geo-3}
$\overline{K^\M} \cap \partial\Omega = \emptyset$. 
Then the trivial patch is selected as the reference patch. 
\item 
\label{item:def-geo-4} 
$\overline{K^\M} \cap \partial \Omega$ is a single point. Then two cases
can occur: 
\begin{enumerate}[(a)]
\item 
$\overline{K^\M} \cap \partial \Omega =  \{\bA_j\}$ for a vertex $\bA_j$ of 
$\Omega$. Then 
the corresponding reference patch is the corner patch
$\check{\calT}^{\Co,n}_{geo,\sigma}$ with $n$ layers of refinement towards 
the origin $\bO$. 
Additionally, $F_{K^\M}(\bO) = \bA_j$. 
\item 
$\overline{K^\M} \cap \partial\Omega = \{\bP\}$, 
where the boundary point $\bP$ is not a vertex of $\Omega$. 
Then the refinement pattern is the
corner patch $\check{\calT}^{\Co,L}_{geo,\sigma}$ with 
$L$ layers of geometric mesh refinement towards $\bO$. 
Additionally, 
it is assumed that $F_{K^\M}(\bO)  = \bP \in \partial\Omega$. 
\end{enumerate}
\item 
\label{item:def-geo-5}
$\overline{K^\M} \cap \partial \Omega = \overline{e}$ for an edge $e$ of 
$K^\M$ and neither  endpoint of $e$ is a vertex of $\Omega$. Then
the refinement pattern is the boundary layer patch $\check {\calT}^{\BL,L}_{geo,\sigma}$ and additionally 
$F_{K^\M}(\{\widehat y = 0\}) \subset \partial\Omega$. 
\item 
\label{item:def-geo-6}
$\overline{K^\M} \cap \partial \Omega = \overline{e}$ for an edge $e$ of 
$K^\M$ and exactly one endpoint of $e$ is a vertex $\bA_j$ of $\Omega$. Then
the refinement pattern is the mixed layer patch
$\check {\calT}^{\Mi,L,n}_{geo,\sigma}$ and additionally 
$F_{K^\M}(\{\widehat y = 0\}) \subset \partial\Omega$ as well as 
$F_{K^\M}(\bO) = \bA_j$. 
\item 
\label{item:def-geo-7}
Exactly two edges of a macro-element $K^\M$ are situated on $\partial\Omega$. 
Then the refinement pattern 
is the tensor patch $\check {\calT}^{\Te,L,n}_{geo,\sigma}$. 
Additionally, 
it is assumed that 
$F_{K^\M}(\{\widehat y = 0\}) \subset \partial\Omega$, 
$F_{K^\M}(\{\widehat x = 0\}) \subset \partial\Omega$, and 
$F_{K^\M}(\bO) = \bA_j$ for a vertex $\bA_j$ of $\Omega$. 
\end{enumerate}
Finally, 
the following technical condition ensures the 
existence of certain meshlines: 
\begin{enumerate}
\setcounter{enumi}{7}
\item 
\label{item:def-geo-mesh-9}
For each vertex $\bA_j$ of $\Omega$, introduce a set of lines 
$$\ell = \bigcup_{K^\M \colon \bA_j \in \overline{K^\M}} 
\{\, F_{K^\M}(\{\widehat y = 0\}), F_{K^\M}(\{\widehat x = 0\}), F_{K^\M}(\{\widehat x = \widehat y\})\, \}.
$$
Let $\Gamma_j$, $\Gamma_{j+1}$ be the two boundary arcs of $\Omega$ that meet at $\bA_j$. 
Then there exists a line $e \in \ell$ such that 
the interior angles $\angle(e,\Gamma_j)$ and $\angle(e,\Gamma_{j+1})$ 
are both less than $\pi$. 
\end{enumerate}
\end{definition}
\begin{example}
Fig.~\ref{fig:patch-examples} (left and middle) shows an example 
of an $L$-shaped domain with macro triangulation and 
suitable refinement patterns. 
\eremk
\end{example}
\begin{remark}
\label{remk:geo-mesh-forL=1}
For fixed $L$ and increasing $n$, the meshes $\Tg$ are geometrically 
refined towards the vertices of $\Omega$. 
These meshes are classical geometric meshes for elliptic problems in corner domains 
as introduced in \cite{babuska-guo86a,babuska-guo86b} and discussed in 
\cite[Sec.~{4.4.1}]{phpSchwab1998}. 
\eremk
\end{remark}
%
\subsection{$hp$-FE approximation of singularly perturbed problems on geometric boundary layer meshes}
\label{sec:sing-approx-geo-mesh}
The principal result \cite[Thm.~{4.1}]{banjai-melenk-schwab19-RD} on 
robust exponential convergence of $hp$-FEM for \eqref{eq:sg-per}
reads as follows: 
\begin{proposition}[\protect{\cite[Thm.~{4.1}]{banjai-melenk-schwab19-RD}}]
\label{prop:singular-approx}
Let $\Omega \subset {\mathbb R}^2$ be a curvilinear polygon with 
$J$ vertices as described in Section~\ref{sec:GeoPrel}. 
Let $A$, $c\ge c_0>0$, $f$ satisfy (\ref{eq:sg-per-asmp}). 
Denote by $\{\Tg\}_{L\geq 0, n \ge L}$ 
a sequence of geometric boundary layer meshes
in the sense of Definition~\ref{def:bdylayer-mesh}. 
Fix $c_1 > 0$. 

Then 
there are constants $C$, $b >0 $, $\beta \in [0,1)$ (depending solely on the data 
$A$, $c$, $f$, $\Omega$, on the parameter $c_1$, and on the analyticity properties 
of the patch-maps of the macro-triangulation ${\mathcal T}^\M$) 
such that the following holds: 
If $\varepsilon \in (0,1]$ and $L$ 
satisfy the \emph{(boundary layer) scale resolution condition}
\begin{equation} \label{eq:L-eps-resolution}
{\sigma^L} \leq c_1 \varepsilon  
\end{equation}
then, for any $q$, $n \in {\mathbb N}$, 
the solution $u^\varepsilon \in H^1_0(\Omega)$ of \eqref{eq:sg-per} can be approximated
from $S^q_0(\Omega,\Tg)$ such that
\begin{align}
\label{eq:thm:singular-approx-10}
\inf_{v \in S^{q}_0(\Omega,\Tg)} &
\left(
\|u^\varepsilon - v\|_{L^2(\Omega)} + 
\varepsilon \|\nabla( u^\varepsilon - v)\|_{L^2(\Omega)} 
\right)
\leq 
C q^{9} 
\left[\varepsilon^\beta \sigma^{(1-\beta)n}  + e^{-b q}\right], 
\\
\label{eq:thm:singular-approx-20}
& N:=\operatorname{dim} S^{q}_0(\Omega,\Tg)
\leq C \left( L^2 q^2 \operatorname{card} {\mathcal T}^\M  
+ n q^2 J\right).
\end{align}
\end{proposition}

Proposition~\ref{prop:singular-approx} is restricted to $\varepsilon \in (0,1]$. 
For $\varepsilon \ge 1$, the solution $u^\varepsilon$ of 
\eqref{eq:sg-per} does not have boundary layer but 
merely corner singularities. Hence, by Remark~\ref{remk:geo-mesh-forL=1}
meshes with fixed $L$ are appropriate. In particular, the boundary 
layer scale resolution condition
(\ref{eq:L-eps-resolution}) is not required:  
\begin{proposition} \label{prop:singular-approx-eps>1}
Assume the hypotheses on $\Omega$ and the data $A$, $c$, $f$ 
as in Proposition~\ref{prop:singular-approx}. 
Let $\{\Tg\}_{L\ge 0, n \ge L}$ be a sequence of 
geometric boundary  layer  meshes\footnote{
No boundary layer refinement/ resolution is required here, i.e.,
``ordinary'', corner refined geometric mesh sequences will suffice.}.

There are constants $C$, $b>0$, $\beta \in [0,1)$
(depending solely on $A$, $c$, $f$, $\sigma\in (0,1)$, 
and the analyticity properties of the macro-triangulation) 
such that the solution $u^\varepsilon$ of (\ref{eq:sg-per}) satisfies 
\begin{equation}
\forall L,n \in {\mathbb N}_0, \ q\in {\mathbb N}\colon\quad 
\inf_{v \in S^{q}_0(\Omega,\Tg)}
\|u^\varepsilon - v\|_{H^1(\Omega)} \leq C \varepsilon^{-2} q^9 
\left( \sigma^{(1-\beta)n}   + e^{-b q}\right), 
\end{equation}
and $\operatorname{dim} S^q_0(\Omega,\Tg)$ satisfies 
(\ref{eq:thm:singular-approx-20}).
\end{proposition}
\subsection{$hp$-FE approximation of singularly perturbed 
            problems on admissible meshes $\Tmin(\varepsilon)$ in $\Omega$}
\label{sec:sing-approx-admissible-meshes}
In Proposition~\ref{prop:singular-approx}, the solution $u^\varepsilon$ 
is approximated on patchwise geometric meshes. These meshes are able to 
capture boundary layers (and corner layers) on a whole range of singular
perturbation parameters $\varepsilon$: 
as long as a lower bound for $\varepsilon$ is known
and provided that geometric mesh refinement resolves all scales,
robust exponential convergence is assured.

On the other hand, if there is a single, explicitly known scale 
$\varepsilon$ then the ``minimal, admissible boundary layer meshes'' 
$\Tmin(\varepsilon):= \calT(\min\{\kappa_0,\lambda p \varepsilon\},L)$
of \cite[Def.~{2.4.4}]{melenk02} (see also \cite[Fig.~{11}]{SSX98_321}
or \cite[Fig.~{2}]{MMCAXeno_Balanced2016}), 
which are designed to resolve a single, 
explicitly known length scale with $hp$-FEM may be employed. 
In contrast to the geometric boundary layer meshes of 
Def.~\ref{def:bdylayer-mesh}, 
these ``minimal'' boundary-fitted meshes are $\varepsilon$-dependent.
\begin{proposition}[\protect{\cite[Thm.~{2.4.8} in conjunction with Thm.~{3.4.8}]{melenk02}}] 
\label{prop:thm-2.4.8}
Let $\Omega \subset {\mathbb R}^2$ be a curvilinear polygon with
$J$ vertices as described in Section~\ref{sec:GeoPrel}. Let $A$, $c\ge c_0>0$, $f$ 
satisfy (\ref{eq:sg-per-asmp}). 

Consider,  for $\kappa_0>0$ determined by $\Omega$,
the two-parameter family 
${\mathcal T}(\kappa,L)$, $(\kappa,L) \in (0,\kappa_0] \times {\mathbb N}$, 
of \emph{admissible meshes} in the sense of 
\cite[Def.~{2.4.4}]{melenk02},\cite[Def.~{3.1}, Figs.~1, 2]{FstmnMM_hpBalNrm2017}.
Let $u^\varepsilon$ be the solution of (\ref{eq:sg-per}).  

Then there are constants $b$, $\lambda_0$ 
independent of $\varepsilon \in (0,1]$ 
such that for every $\lambda \in (0,\lambda_0]$ there is $C > 0$ such that 
for every $q\ge 1$, $L \ge 0$
there holds the error bounds
\begin{align} 
\label{eq:prop:thm-2.4.8-a}
& \inf_{v \in S^q_0(\Omega,\Tmin(\varepsilon))} 
\|u^\varepsilon - v\|_{L^2(\Omega)} 
     + \varepsilon \|\nabla( u^\varepsilon - v)\|_{L^2(\Omega)} 
 \leq C q^6 \left[ e^{-b \lambda q} + \varepsilon e^{-b L}\right], 
\\
\label{eq:prop:thm-2.4.8-b}
& \quad \Tmin(\varepsilon): = \calT(\min\{\kappa_0,\lambda q \varepsilon\},L), 
\\
\label{eq:prop:thm-2.4.8-c}
& \quad 
  N := \operatorname{dim} 
S^q_0(\Omega, \Tmin(\varepsilon)) 
 \leq C L q^2. 
\end{align}
In particular, for $L \sim q$, one has with $C$, $b'$ independent of $\varepsilon$
\begin{align*}
\inf_{v \in S^q_0(\Omega,\Tmin(\varepsilon))}
\|u^\varepsilon - v\|_{L^2(\Omega)} + 
\varepsilon \|\nabla( u^\varepsilon - v)\|_{L^2(\Omega)} 
\leq C \exp(-b' N^{1/3}). 
\end{align*}
\end{proposition}
For $\varepsilon \ge 1$ these admissible boundary layer meshes are 
the well-known geometric meshes with $L$ layers of geometric refinement 
as introduced in \cite{babuska-guo86a,babuska-guo86b} and discussed in 
\cite[Sec.~{4.4.1}]{phpSchwab1998}. 
These  geometric, corner-refined meshes 
are similar to the meshes $\Tg$ with fixed $L = 0$ discussed in 
Remark~\ref{remk:geo-mesh-forL=1}.
In particular, the minimal boundary layer meshes $\Tmin(\varepsilon)$ 
for $\varepsilon \ge 1$ do not really depend on 
$\varepsilon$, $\lambda$, and $q$. 
However, for consistency of notation, we keep the notation 
$\Tmin(\varepsilon)$ in the following result, which covers the case 
$\varepsilon \ge 1$. We need this result since 
the range \eqref{eq:RDevpBds} of eigenvalues $\mu_i$ involves also
eigenvalues $\mu_i \ge 1$.
\begin{proposition} \label{prop:eps>1}
Under the assumptions of Proposition~\ref{prop:thm-2.4.8},
there exist constants $b$, $C>0$ such that
\begin{align*} 
&
\forall \varepsilon \geq 1, \forall q,L\in {\mathbb N}:
\quad 
\inf_{v \in S^q_0(\Tmin(\varepsilon))}
\|u^\varepsilon - v\|_{H^1(\Omega)} 
 \leq C \varepsilon^{-2} q^6 \left[ e^{-b \lambda q} +  e^{-b L}\right], \\
	& \quad \mbox{where}\quad 
	 N := \operatorname{dim} S^q_0(\Omega,\Tmin(\varepsilon)) \leq C L q^2. 
\end{align*}
In particular, for $L \sim q$, there are constants $b'$, $C>0$ such that 
\begin{align*}
\forall \varepsilon \geq 1, q,L\in {\mathbb N}: 
\quad
\inf_{v \in S^q_0(\Omega,\Tmin(\varepsilon))} 
\|u^\varepsilon - v\|_{H^1(\Omega)} 
\leq C \varepsilon^{-2} \exp(-b' N^{1/3}). 
\end{align*}
\end{proposition}
It is worth pointing out the following differences between the 
approximation on geometric boundary layer meshes $\Tg$ and on 
the minimal admissible boundary layer meshes $\Tmin$: 
a) the use of the mesh $\Tg$ requires the scale resolution condition 
\eqref{eq:L-eps-resolution}. 
It requires $L \gtrsim |\ln \varepsilon|$ 
so that the approximation result Proposition~\ref{prop:singular-approx}
depends (weakly) on $\varepsilon$. 
b) Selecting $n \simeq L \simeq q$ 
in Proposition~\ref{prop:singular-approx} yields convergence 
$O(\exp(-b \sqrt[4]{N}))$ whereas the choice $L \simeq q$ in 
Proposition~\ref{prop:thm-2.4.8} yields the better convergence behavior 
$O(\exp(-b' \sqrt[3]{N}))$. 
c) The meshes $\Tmin$ are designed to approximate
a \emph{single} scale well whereas the meshes $\Tg$ are capable to 
resolve a range of scales. 
d) The meshes $\Tmin$ rely on a suitable choice
of the parameter $\lambda$ whereas the geometric boundary layer meshes $\Tg$ 
do not have parameters that need to be suitably chosen. 
\section{Exponential Convergence of \EhpFEMp}
\label{S:hpx}
Based on the $hp$ semidiscretization in the extended variable
combined with the diagonalization in Section~\ref{S:diagonalization-abstract-setting},
we use the $hp$-approximation results 
from Section~\ref{S:approx-sing-perturb} 
to prove exponential convergence of $hp$-FEM for the 
CS-extended problem \eqref{eq:ue-variational}.  

As is revealed by the diagonalization \eqref{eq:decoupled-problems}, 
the $y$-semidiscrete solution $G^{\bmr}_{M} \ue$ 
contains $\calM$ separate length scales associated with the eigenvalues 
$\mu_i$, $i=1,...,\calM$. 
The solutions $U_i \in H^1_0(\Omega)$ 
of the resulting  $\calM$ many independent, linear second-order
reaction-diffusion problems in $\Omega$ exhibit
both, boundary layers and corner singularities. 

In {\bf Case A}, which we discuss in Section~\ref{S:ExpConvI}, 
we employ for each $i$ a ``minimal'' $hp$-FE space in $\Omega$
that resolves boundary- and corner layers appearing in the $U_i$ 
due to possibly large/small values of $\mu_i$. Mesh design principles 
for such ``minimial'' FE spaces that may resolve a single scale of a 
singularly perturbed problem have already been presented
in, e.g.,  \cite{schwab-suri96,SSX98_321,melenk97,melenk-schwab98,melenk02}; 
the specific choice $\Tmin(\varepsilon)$ has been 
discussed in 
Propositions~\ref{prop:thm-2.4.8}, \ref{prop:eps>1}
and will be used in our analysis. 

In {\bf Case B}, which we discuss in Section~\ref{S:ExpConvII}, 
we discretize these decoupled, reaction-diffusion problems
by \emph{one common $hp$-FEM in the bounded polygon $\Omega$},
which employs both, geometric corner refinement as well as
geometric boundary refinement, as in \cite{melenk-schwab98,melenk02}.
Due to the need to obtain FE solutions for \emph{all} $\mu_i$ in 
\emph{one common FE space in $\Omega$}, however (in order that the 
sum \eqref{eq:representation-by-reaction-diffusion-problems}
belong to a tensor product $hp$-FE space),
our analysis will provide \emph{one} $hp$-FE space in $\Omega$
which will resolve \emph{all} boundary and corner layers 
due to small parameters $\mu_i$ near $\partial \Omega$.
As we shall show, in {\bf Case B} the total number of DOFs 
is larger than in {\bf Case A}.

%
\subsection{Exponential Convergence I: Diagonalization and Minimal Meshes}
\label{S:ExpConvI}
The robust exponential convergence result Proposition~\ref{prop:thm-2.4.8}
allows us to establish, in conjunction with the diagonalization 
\eqref{eq:eigenvalue-problem}--\eqref{eq:U_M},
a first exponential convergence result 
in {\bf Case A} of Section \ref{S:diagonalization-abstract-setting}.
We consider the following numerical scheme, which relies
on the ``minimal boundary layer meshes'' 
$\Tmin(\varepsilon)$ from 
	\cite[Sec.~{2.4.2}]{melenk02} already discussed 
in Proposition~\ref{prop:thm-2.4.8}: 
\begin{enumerate}[(1)]
\item Select $\Y$ with $c_1 M \leq \Y \leq c_2 M$ and consider the space 
$S^{\bmr}_{\{\Y\}}((0,\Y),\calG^M_{geo,\sigma})$ for the geometric 
mesh $\calG^M_{geo,\sigma}$ on $(0,\Y)$ with $M$ elements and a
linear degree vector $\bmr$ with slope $\slope>0$. 
\item
Solve the eigenvalue problem \eqref{eq:eigenvalue-problem}, \eqref{eq:eigenbasis-normal}. 
\item Select $\lambda > 0$. 
Define $U^{q,L}_i \in S^{q}_0(\Tmin(\sqrt{\mu_i}))$ as the solution of 
\begin{equation}
\label{eq:UqL} 
\forall v \in 
S^{q}_0(\Tmin(\sqrt{\mu_i}))\ \colon\ 
a_{\mu_i,\Omega} (U^{q,L}_i,v) = d_s v_i(0) \langle f,v\rangle. 
\end{equation}
\item 
Define the approximation $\ue^{q,L}(x,y):= 
\sum_{i=1}^{\calM_{geo}} v_i(y) U^{q,L}_i(x).$
\end{enumerate}

For the approximation error $\ue - \ue^{q,L}$ we have: 
\begin{theorem}\label{thm:ExpCnvI}
%
Let $\Omega \subset {\mathbb R}^2$ be a curvilinear polygon with
$J$ vertices as described in Section~\ref{sec:GeoPrel}. Let $A$, $f$ satisfy (\ref{eq:AnRegAc}) 
and let $A$ be uniformly symmetric positive definite on $\Omega$. 
Fix positive constants $c_1$, $c_2$, and $\slope$. 

Then 
there are constants $C$, $b$, $b'$, $b''$, $\lambda_0>0$ 
(depending on $\Omega$, $A$, $c$, and the parameters characterizing the mesh family $\Tmin$)
such that for any $\lambda \in (0,\lambda_0]$ there holds for all $q$, $M\ge 1$, $L \ge 0$
\begin{align}
\label{eq:ErrBdB}
\| u - \tr \ue^{q,L} \|_{\Hs} 
&\lesssim
\|\nabla(\ue - \ue^{q,L}) \|_{L^2(y^\alpha,\C)}  
\\
\nonumber 
& 
\leq C M^{2-\alpha} q^6 \left[ \exp(-b \lambda q) + \exp(-b' L)\right] 
+ \exp(-b^{\prime\prime} M). 
\end{align}
In particular, for $q \simeq L \simeq M \simeq p$, 
denoting
$\ue^p:= \ue^{q,L}$ with this choice of $q$ and $L$, 
and 
the total number of degrees of freedom
$N =  \sum\limits_i \operatorname{dim} S^q_0(\Omega, \Tmin(\sqrt{\mu_i}))$,
\begin{equation}\label{eq:ErrBdA}
\| u - \tr \ue^{p} \|_{\Hs} 
\lesssim
\|\nabla(\ue - \ue^p) \|_{L^2(y^\alpha,\C)} 
\lesssim 
\exp(-bp) 
\simeq 
\exp(-b''' \sqrt[5]{N}) 
\;,
\end{equation}
where the constant $b'''$ depends additionally on the implied constants
in $q \simeq L \simeq M$. 
\end{theorem}
\begin{remark}
\label{remk:minmesh}
The approximation result (\ref{eq:ErrBdA}) still holds 
if the linear degree vector $\bmr$ in the definition of 
$\ue^{q,L}$ is replaced with a constant polynomial degree $r \sim M$. 
\eremk
\end{remark}
\begin{proof}
\emph{Step 1 (semidiscretization error):}
The analyticity of $f$ on $\overline{\Omega}$ implies $f \in {\mathbb H}^{-s+\nu}(\Omega)$
for any $\nu \in (0,1/2+s)$. Hence, by 
\eqref{eq:lemma:semidiscretization-error-10}, 
the semidiscretization error $\ue - G^{\bmr}_M \ue$ satisfies 
for suitable $b > 0$ independent of $M$
\begin{equation}
\label{eq:semidisc-err}
\|\ue - G^{\bmr}_M \ue\|_\C \lesssim e^{-b M}. 
\end{equation}
%

\noindent
\emph{Step 2 (representation of $G^{\bmr}_M \ue$)}: The semidiscrete approximation $G^\bmr_M\ue$
may be expressed in terms of the eigenbasis $\{ v_j \}_{j=1}^M$ 
in \eqref{eq:eigenvalue-problem}, \eqref{eq:eigenbasis-normal} as 
$$
 (G^\bmr_M \ue)(x',y) = \sum_{i=1}^{\calM_{geo}} v_i(y) U_i(x') \;,
$$
where the function $U_i$ solve by \eqref{eq:decoupled-problems} 
\begin{equation*}
\forall V \in H^1_0(\Omega)\; \colon\;
a_{\mu_i,\Omega}(U_i,V) = d_s v_i(0) \langle f,V\rangle. 
\end{equation*}
%
\noindent
\emph{Step 3}: 
For every $i=1,\ldots,\calM_{geo}$, and for every $q\in \bbN$,
approximate the $U_i\in H^1_0(\Omega)$ by its Galerkin approximation 
$U^{q,L}_i \in 
S^q_0(\Omega,\Tmin(\sqrt{\mu_i}))\subset H^1_0(\Omega)$.
That is, $U^{q,L}_i  = \Pi_i^q U_i $ is the 
$a_{\mu_i,\Omega}(\cdot,\cdot)$-projection of $U_i$ given by 
\eqref{eq:galerkin-for-vi}. It is the best approximation to $U_i$ in the 
corresponding energy norm and satisfies 
$$
\| U_i - \Pi_i^q U_i \|_{\mu_i,\Omega} 
= 
\min_{V\in  S^q_0(\Omega,\Tmin(\sqrt{\mu_i}))} \| U_i - V \|_{\mu_i,\Omega}
\;.
$$
By linearity of $\Pi_i^q$ and the analyticity of $f$ 
Propositions~\ref{prop:thm-2.4.8}, \ref{prop:eps>1} (depending on 
whether $\mu_i\leq 1$ or $\mu_i > 1$) and 
Lemma~\ref{lemma:lambda} 
\begin{equation*}
\| U_i - \Pi_i^q U_i \|_{\mu_i,\Omega} 
\lesssim |v_i(0)| q^6 \left(e^{-b \lambda q}+ \sqrt{\mu_i} e^{-b' L} \right)
\lesssim \calM_{geo}^{(1-\alpha)/2} q^6 \left(e^{-b \lambda q}+ \sqrt{\mu_i} e^{-b' L} \right). 
\end{equation*}

\noindent
\emph{Step 4 (Proof of \eqref{eq:ErrBdB})}:
With the approximations $U^{q,L}_i$ the approximation 
$\ue^{q,L}$ of $\ue$ is given by 
$$
\ue^{q,L}(x',y) := \sum_{i=1}^{\calM_{geo}} v_i(y) U^{q,L}_i(x') \in \HL(y^\alpha,\C).
$$
%
%
%
{}From \eqref{eq:pythagoras} we get for 
$Z = \ue_M - \ue^{q,L}  = \sum_{i=1}^{\calM_{geo}} v_i(y) (U_i(x) - U^{q,L}_i(x))$
$$
\| Z \|_{\C}^2 = a_\C(Z,Z) = \sum_{i=1}^{\calM_{geo}} \| U_i - U^q_i \|^2_{\mu_i,\Omega} 
\lesssim \calM_{geo}^{2-\alpha} q^{12} \left[ \exp(-2b\lambda q) + \exp(-2 b' L)
\right] 
\;.
$$
We note that $\calM_{geo} \sim M^2$. Combining this last estimate with 
(\ref{eq:semidisc-err}) yields the second estimate in \eqref{eq:ErrBdB}. 
The first estimate in \eqref{eq:ErrBdB} expresses the continuity of the trace
operator at $y = 0$. 
%
%

\noindent
\emph{Step 5 (complexity estimate)}: 
Using that $\calM_{geo} = O(M^2) = O(q^2)$ and the fact that 
$\operatorname{dim} S^q_0(\Omega, \Tmin(\sqrt{\mu_i}))
\leq C L q^2 = O(q^3)$ as well as the assumption 
$q \simeq L \simeq M \simeq p$, 
we arrive at a total problem size $N = O(q^5)$. Absorbing algebraic factors in the 
exponentially decaying one in \eqref{eq:ErrBdB} then yields 
\eqref{eq:ErrBdA}. 
\end{proof}
%
\subsection{Exponential Convergence II: Geometric Boundary Layer Meshes}
\label{S:ExpConvII}
In this section, we show that exponential convergence of a 
Galerkin method for \eqref{eq:ue-variational} can be achieved by a suitable 
choice of meshes $\calT$ and $\calG^M$ in the tensor product space 
$\V^{q,\bmr}_{h,M}(\calT,\calG^M)$ of \eqref{eq:TPFE}. That is, 
we place ourselves in {\bf Case B} in Section~\ref{S:diag}. For the discretization
in $y$, we select again the spaces 
$S^{\bmr}_{\{\Y\}}((0,\Y), \calG^M_{geo,\sigma})$ 
with $\Y \sim M$ and the linear degree vector $\bmr$ with slope $\slope$. 
The $hp$-FE discretization in $\Omega$
is based on the space $S^q_0(\Omega,\Tg)$
with the geometric boundary layer mesh $\Tg$ in Definition~\ref{def:bdylayer-mesh}. 
Recall that $G^{q,\bmr}_{h,M}\ue$ denotes the Galerkin projection of the solution 
$\ue $
onto $\V^{q,\bmr}_{h,M}(\Tg,\calG^M_{geo,\sigma}) 
= S^q_0(\Omega,\Tg) \otimes S^{\bmr}_{\{\Y\}}((0,\Y),\calG^M_{geo,\sigma})$. 
In Theorem~\ref{thm:ExpCnvII} below, we will 
focus on the case $q \simeq L \simeq M$ and the corresponding Galerkin projection 
is denoted $\ue^p_{TP}$. 
\begin{remark} 
\begin{enumerate}[(i)]
\item In contrast to the procedure 
of {\bf Case A} in the preceeding Section~\ref{S:ExpConvI}, precise knowledge
of the length scales $\sqrt{\mu_i}$ is not necessary. 
\item 
The diagonalization procedure may be carried out numerically
and results in decoupled reaction-diffusion problems,
affording parallel numerical solution.
\item
The linear degree vector $\bmr$ could be replaced with a constant
degree $r \sim M$, and Theorem~\ref{thm:ExpCnvII} will still hold. 
\eremk
\end{enumerate}
\end{remark}
For the tensor-product $hp$-FEM in $\C_\Y$ 
we also have an exponential convergence result: 
\begin{theorem}\label{thm:ExpCnvII}
Let $\Omega \subset {\mathbb R}^2$ be a curvilinear polygon with
$J$ vertices as described in Section~\ref{sec:GeoPrel}. Let $A$, $f$ satisfy
(\ref{eq:AnRegAc}) and let $A$ be uniformly symmetric positive definite on $\Omega$. 
Fix a slope $\slope > 0$.  
Set 
\begin{equation}
\label{eq:thm:ExpCnvII-10}
\Y \simeq L \simeq M \simeq n \simeq q =:p
\end{equation}
With these choices, 
denote by $\ue^p_{TP}$ the 
Galerkin projection of $\ue$ onto the tensor product $hp$-FE space 
$\V^{q,\bmr}_{h,M}(\Tg,\calG^M_{geo,\sigma}) 
= 
S^q_0(\Omega,\Tg) \otimes S^{\bmr}_{\{\Y\}}((0,\Y),\calG^M_{geo,\sigma})$. 

Then 
$N:= \operatorname{dim} \V^{q,\bmr}_{h,M}(\Tg,\calG^M_{geo,\sigma}) = O(p^6)$
and
there are constants $C$, $b$, $b' >0$ depending only on $\Omega$, $A$, $f$, 
the macro triangulation $\calT^\M$ underlying the geometric boundary
layer meshes $\Tg$, the slope $\slope$, and the implied constants in 
(\ref{eq:thm:ExpCnvII-10}) such that 
$$
\| u - \tr \ue^p_{TP} \|_{\Hs}
\lesssim
\|\nabla(\ue - \ue^p_{TP})\|_{L^2(y^\alpha,\C)}
\lesssim
\exp(-bp)
\simeq
\exp(-b' \sqrt[6]{N})
\;.
$$
\end{theorem}
\begin{proof}
The proof of this result is structurally 
along the lines of the proof of Theorem~\ref{thm:ExpCnvI}.
We omit details and
merely indicate how the scale resolution condition \eqref{eq:L-eps-resolution} 
is now accounted for.
We note that for fixed $\Y$ and 
$M' < M$, we have that the spaces $S^{\bmr}_{\{\Y\}}((0,\Y),\calG^{M'})$ and 
$S^{\bmr}_{\{\Y\}}((0,\Y),\calG^{M}_{geo,\sigma})$ satisfy 
$S^{\bmr}_{\{\Y\}}((0,\Y),\calG^{M'}_{geo,\sigma}) \subset 
S^{\bmr}_{\{\Y\}}((0,\Y),\calG^{M}_{geo,\sigma})$ (if the same slope $\slope$ 
for the linear degree vector is chosen). Hence, the Galerkin error 
for the approximation from the space 
$S^{q}_0(\Omega,\Tg) \otimes 
S^{\bmr}_{\{\Y\}}((0,\Y),\calG^{M}_{geo,\sigma}) $
is smaller than that from 
$S^{q}_0(\Omega,\Tg) \otimes 
S^{\bmr}_{\{\Y\}}((0,\Y),\calG^{M'}_{geo,\sigma}) $, and 
we therefore focus on bounding the approximation error for 
$S^{q}_0(\Omega,\Tg) \otimes 
S^{\bmr}_{\{\Y\}}((0,\Y),\calG^{M'}_{geo,\sigma}) $. 
We select $M'$ of the form $M'  = \lfloor \eta M\rfloor $ 
for some $\eta$ to be chosen below.
For ease of notation, we simply set $M' = \eta M$. 

By Lemma~\ref{lemma:lambda} we have that the smallest length scale 
of the singlarly perturbed problems for the space 
$S^q_0(\Tg) \otimes S^{\bmr}_{\{\Y\}}(\calG^{M'}_{geo,\sigma}) \}$ 
is 
$\min_i \mu_i \gtrsim {M'}^{-2} \sigma^{2M'}$
and that scale resolution condition \eqref{eq:L-eps-resolution} therefore 
reads 
\begin{equation}
\label{eq:scale-res-modified}
\sigma^L \leq c_1 \min_{i} \sqrt{\mu_i} 
\leq c_1 {M'}^{-1} \sigma^{M'}
\leq c_1 {\eta M}^{-1} \sigma^{\eta M}
\end{equation}
Since $L \simeq M$, we see that \eqref{eq:scale-res-modified}
can be satisfied for some fixed $c_1$ provided $\eta$ is suitably chosen. 
The approximation of $\ue$ from 
$S^q_0(\Tg) \otimes S^{\bmr}_{\{\Y\}}(\calG^{M'}_{geo,\sigma})$ 
then follows by arguments very similar to those of the proof of 
Theorem~\ref{thm:ExpCnvI}. 
\end{proof}
%
\section{Exponential Convergence of \BKp} 
\label{S:ExpConvIII}
The $hp$-since BK FEM is based on exponentially convergence, so-called 
``sinc'' quadratures 
to the Balakrishnan formula
\begin{equation}
\label{eq:dunford-integral}
\calL^{-s} = 
c_B \int_{-\infty}^\infty e^{-sy} (I + e^{-y}\calL)^{-1} \diff y \;.
\end{equation}
as described in Section~\ref{S:BalkrFrm} and \eqref{eq:BlkrRep} 
(see \cite{BonitoEtAl_FracSurv2017,BP:13,BoPascFracRegAcc2017} and the references there).
We briefly review the corresponding exponential convergence results in Section~\ref{S:SincAppr}.
The numerical realization of the sinc quadrature approximation of \eqref{eq:dunford-integral} 
leads again to the numerical solution of decoupled, local linear reaction-diffusion
problems in $\Omega$. These boundary value problems are again singularly perturbed.
Accordingly, 
we discuss two classes of $hp$-FE approximations for their numerical solution: 
In Section~\ref{sec:SincBKcaseA}, we discuss {\bf Case A}, which is based on 
the \emph{minimal boundary layer meshes} $\Tmin$ in $\Omega$.
In Section~\ref{sec:SincBKcaseB}, we detail {\bf Case B}, where \emph{geometric boundary
layer meshes} in $\Omega$ are employed.
The latter allow one common $hp$-FEM for all values of parameters arising from the 
sinc quadrature approximation of \eqref{eq:dunford-integral}.
\subsection{Sinc quadrature approximation}
\label{S:SincAppr}
The above integral (\ref{eq:dunford-integral}) 
can be discretized by so-called ``sinc'' quadratures 
(see, e.g., \cite{Stenger83,BoPascFracRegAcc2017}). 
To that end, 
we define for $K\in \bbN$
\begin{equation}\label{eq:SincPrm}
y_j := jK^{-1/2} = jk, \;\; |j| \leq K, \;\; k:=1/\sqrt{K}\;. 
\end{equation}
For $f\in L^2(\Omega)\subset \mathbb{H}^{-s}(\Omega)$ for every $0<s<1$, 
the (semidiscrete) 
\emph{sinc quadrature approximation $u_M$ of $u = \calL^{-s}f \in \Hs$}
as represented in \eqref{eq:dunford-integral} reads 
with $\eps_j := e^{-y_j/2} = e^{j/(2\sqrt{K})}$, $|j|\leq K$:
\begin{equation}\label{eq:Sincu}
u_K = Q^{-s}_k(\calL) f 
:=
c_B k \sum_{|j|\leq K} \eps_j^{2s} \left(I + \eps_j^2 \calL\right)^{-1} f 
\;.
\end{equation}
We note that for any $k$, we have that 
$Q^{-s}_k(\calL): 
\mathbb{H}^{-1}(\Omega)\rightarrow\mathbb{H}^1(\Omega)$ is a 
bounded linear map.  By the continuous embeddings 
$\mathbb{H}^1(\Omega) \subseteq \mathbb{H}^s(\Omega) 
\subseteq \mathbb{H}^{-s}(\Omega) \subseteq \mathbb{H}^{-1}(\Omega)$,
also 
$Q^{-s}_k(\calL): \mathbb{H}^{-s}(\Omega) \rightarrow \mathbb{H}^s(\Omega)$  
is a bounded linear map for any $0<s<1$.
The semidiscretization error 
$\calL^{-s} - Q^{-s}_k(\calL)$ is bound in \cite{BoPascFracRegAcc2017}:
\begin{proposition}
[\protect{\cite[Thm.~{3.2}]{BoPascFracRegAcc2017}}] 
\label{prop:SincErrBLP17}
For $f\in L^2(\Omega)$
and for every $0 \leq \beta < s$ with $0<s<1$ denoting the exponent
of the fractional diffusion operator in \eqref{fl=f_bdddom},
there exist constants $b$, $C>0$ (depending on $\beta$, $s$, $\Omega$, and $\calL$) 
such that for every $k>0$ as in \eqref{eq:SincPrm} holds,
with $D(\calL^\beta) = \mathbb{H}^{2 \beta}(\Omega)$
where $\mathbb{H}^{\sigma}(\Omega)$ is as in \eqref{def:Hs},
\begin{equation}\label{eq:SincErrBd}
\| (\calL^{-s} - Q^{-s}_k(\calL))f \|_{D(\calL^{\beta})}
\leq 
C\exp(-b/k) \| f \|_{L^2(\Omega)}
\;. 
\end{equation}
\end{proposition}
\begin{remark}\label{remk:ChoicSincParm}
Sinc approximation formulas such as \eqref{eq:Sincu}
have a number of parameters which can be optimized in various ways.
The error bound in Proposition \ref{prop:SincErrBLP17} is merely
one particular choice (the so-called ``balanced'' choice of parameters), 
which is sufficient for the exponential sinc error bound \eqref{eq:SincErrBd}.
Other choices yield analogous (exponential) sinc error bounds, 
with possibly better numerical values for the constants $b$, $C>0$ 
in \eqref{eq:SincErrBd}. We point out that we make such a choice 
in our numerical examples in \eqref{eq:Sinc_practice} and refer
to \cite[Rem.~{3.1}]{BoPascFracRegAcc2017} for details.  
\eremk
\end{remark}
%
\subsection{$hp$-FE approximation in $\Omega$}
\label{S:SincHP}
%
The sinc approximation error bound \eqref{eq:SincErrBd} implies 
exponential convergence of the sinc quadrature sum \eqref{eq:Sincu}, 
which we write as
\begin{equation}\label{eq:SincSum}
Q^{-s}_k(\calL) f
= 
c_Bk \sum_{|j| \leq K}  \eps_j^{2s} w_j\;.
\end{equation}
Here, the $w_j \in H^1_0(\Omega)$ 
are solutions of the $2K+1$ reaction-diffusion problems
\begin{equation}\label{eq:RDPbwj}
\eps_j^2 \calL w_j + w_j = f\quad\mbox{in}\quad \Omega, \quad w_j|_{\partial \Omega} = 0\;,
\quad |j| \leq K\;.
\end{equation}
With the bilinear form
$a_{\varepsilon_j^2,\Omega}(\cdot,\cdot)$ from \eqref{eq:blfa-sg-per},
their variational formulations reads:
find $w_j \in H^1_0(\Omega)$ such that 
\begin{equation}\label{RDPBweak}
\forall v\in H^1_0(\Omega)\;\colon \quad 
a_{\eps_j^2,\Omega}(w_j,v) = (f,v)\;. 
\end{equation}
The reaction diffusion problems \eqref{eq:RDPbwj} 
are again of the type \eqref{eq:decoupled-problems} 
for which exponentially convergent $hp$-FE approximations were
presented in Section \ref{S:hpx}, from \cite{melenk02} and \cite{banjai-melenk-schwab19-RD}.
A \emph{fully discrete \BK approximation} is constructed 
by replacing 
$w_j$ in \eqref{eq:SincSum}, \eqref{eq:RDPbwj} by one of the $hp$-FE approximations
discussed in Section~\ref{S:hpx}. 
As in the case of the \EhpFEMp, 
also for the \BK one can distinguish
{\bf Case A}, in which 
each problem \eqref{RDPBweak} is discretized using a different 
$hp$-FE space, and 
{\bf Case B}, 
where all problems \eqref{RDPBweak} are discretized by the same 
$hp$-FE space in $\Omega$.
\subsubsection{Case A}
\label{sec:SincBKcaseA}
We discretize the singularly perturbed problems \eqref{RDPBweak} with 
length scales $\varepsilon_j$ using the spaces $\Tmin(\varepsilon_j)$. 
That is, denoting the resulting approximations generically by $w_j^{hp}$, 
defined by:  
for $|j|\leq K$, 
find $w_j^{hp} \in S^q_0(\Omega,\Tmin(\varepsilon_j))$
such that
\begin{equation}
\label{eq:BK-CaseA}
\forall v \in S^q_0(\Omega, \Tmin(\varepsilon_j))\, \colon\,
a_{\eps_j^2,\Omega}(w_j^{hp},v) = \langle f,v\rangle  \;. 
\end{equation}
The $hp$-FE approximations $w_j^{hp}$ are well-defined.
Replacing in \eqref{eq:SincSum} the $w_j$ by their $hp$-FE approximations,
we obtain the \BK  approximation 
of the (inverse of) the fractional diffusion operator $\calL^s$:
\begin{equation}\label{eq:SincSumhp}
Q^{-s}_k(\calL_{hp}) f
:= 
c_Bk \sum_{|j| \leq K}  \eps_j^{2s} w^{hp}_j\;.
\end{equation}

To bound the error $\| u - Q^{-s}_k(\calL_{hp}) f \|_{\mathbb{H}^s(\Omega)}$, we write
\begin{align*}
\| u - Q^{-s}_k(\calL_{hp}) f \|_{\mathbb{H}^s(\Omega)}
& \leq  \displaystyle
\| (\calL^{-s} - Q^{-s}_k(\calL)) f  \|_{\mathbb{H}^s(\Omega)}
+ 
\| Q^{-s}_k(\calL - \calL_{hp}) f  \|_{\mathbb{H}^s(\Omega)} 
\;.
\end{align*}
For the first term, the sinc approximation error,
we use the error bound \eqref{eq:SincErrBd} with $\beta = s/2$. 
Using 
$D(\calL^{s/2}) = \mathbb{H}^s(\Omega)$
for $0<s<1$, 
we obtain from \eqref{eq:SincErrBd} and from $k \simeq K^{-1/2}$
(see \eqref{eq:SincPrm}) the bound 
\begin{equation}\label{eq:SnchpT1}
\| (\calL^{-s} - Q^{-s}_k(\calL)) f  \|_{\mathbb{H}^s(\Omega)}
\leq C\exp(-b\sqrt{K}) \| f\|_{L^2(\Omega)}
\;.
\end{equation}
To bound the second term, definition \eqref{eq:SincSum}
and the triangle inequality imply
\begin{equation}\label{eq:SnchpSumBd}
\| Q^{-s}_k(\calL - \calL_{hp}) f  \|_{\mathbb{H}^s(\Omega)}
\lesssim
k
\sum_{|j| \leq K }
\eps_j^{2s} \| w_{j} - w_{j}^{hp} \|_{\mathbb{H}^s(\Omega)}
\;.
\end{equation}

To invoke the $hp$-error bound \eqref{eq:thm:singular-approx-10}
with the norm \eqref{eq:pythagoras},
we use the interpolation inequality
\begin{equation}\label{eq:InterpIneq}
\forall 0<s<1 \;\exists C_s>0 \;\forall w\in H^1_0(\Omega):
\quad 
\| w \|_{\mathbb{H}^s(\Omega)}
\leq C_s 
\| w \|^{1-s}_{L^2(\Omega)} 
\| \nabla w \|^s_{L^2(\Omega)} \;.
\end{equation}
We apply this to each term in  \eqref{eq:SnchpSumBd} and,
using the definition \eqref{eq:pythagoras} of the norm 
$\| w \|_{\eps^2,\Omega}$,
and $\| w \|_{\eps^2}^2 
:= \eps^2 \| \nabla w \|^2_{L^2(\Omega)} + \| w \|^2_{L^2(\Omega)} 
\simeq 
\left( \eps \| \nabla w \|_{L^2(\Omega)} + \| w \|_{L^2(\Omega)}  \right)^2$ 
for all $\eps \geq 0$,
arrive at 
\begin{align*}
\| Q^{-s}_k(\calL - \calL_{hp}) f  \|_{\mathbb{H}^s(\Omega)}
&\lesssim  \displaystyle 
k
\sum_{|j|\leq K} 
\eps_j^s 
\| w_{j} - w_{j}^{hp} \|^{1-s}_{L^2(\Omega)} 
\left(
\eps_j 
\| \nabla(w_{j} - w_{j}^{hp}) \|_{L^2(\Omega)}
\right)^s 
\\
&\lesssim  \displaystyle 
k
\sum_{|j|\leq K}
\eps_j^{s}
\|  w_{j} - w_{j}^{hp} \|_{\eps_j^2}
\;.
\end{align*}
We split the summation indices 
as $I_+ \cup I_-$, i.e.,  
$I_+ := \{|j|\leq K\} \cap \{j>0\}$ 
and 
$I_- := \{|j|\leq K\} \cap \{ j\leq 0\}$.

As $\eps_j =\exp(j/(2\sqrt{K}))$, 
$j\in I_-$ implies $0<\eps_j \leq 1$ 
and 
$j\in I_+$ to $1<\eps_j \leq \exp(\sqrt{K}/2)$.
With Proposition~\ref{prop:eps>1}, 
we estimate the sum over $j\in I_+$ according to 
\begin{align}
\nonumber 
\sum_{j\in I_+} \eps_j^{s} \|  w_{j} - w_{j}^{hp} \|_{\eps_j^2} 
&\lesssim \displaystyle 
q^6 \sum_{j\in I_+} \eps_j^{s-1} \left(\exp(-bq)+\exp(-b'L)\right)
\\
\nonumber 
& = \displaystyle 
q^6 \left(\exp(-bq)+\exp(-b'L)\right) \sum_{j\in I_+} \exp(j(s-1)/(2\sqrt{K}))
\\
\label{eq:I_+}
&\lesssim \displaystyle 
\sqrt{K} q^6 \left( \exp(-bq)+\exp(-b'L) \right) 
\;.
\end{align}
We estimate the sum over $j\in I_-$ (i.e., $0<\eps_j \leq 1$)
with Proposition~\ref{prop:thm-2.4.8} 
\begin{align}
\nonumber 
\displaystyle
\sum_{j\in I_-} \eps_j^{s} \|  w_{j} - w_{j}^{hp} \|_{\eps_j^2} 
&\lesssim \displaystyle 
 q^6 \sum_{j\in I_-} \eps_j^{s} \left[ \exp(-b\lambda q ) + \eps_j \exp(-b'L) \right]
\\
\nonumber 
&\lesssim \displaystyle 
	q^6 \left[ \exp(-b\lambda q ) + \exp(-b'L) \right] \sum_{j=1}^{K} \exp(-sj/(2\sqrt{K}))
\\
&\simeq  \displaystyle 
	\sqrt{K} q^6 \left[ \exp(-b\lambda q ) + \exp(-b'L) \right]
\;.
\label{eq:I_-}
\end{align}
We select $L\simeq q$ (i.e., the number $L$ of mesh-layers 
proportional to the polynomial degree $q\geq 1$) 
with proportionality constant independent of $\eps_j$. 
Furthermore, we note $k = 1/\sqrt{K}$ and select $K \simeq q$ so that 
%
%
%
\begin{equation}\label{eq:SnchpT2}
\| Q^{-s}_k(\calL - \calL_{hp}) f  \|_{\mathbb{H}^s(\Omega)}
\lesssim q^6 \exp(-b\lambda q).  
\end{equation}
Combining the error bounds \eqref{eq:SnchpT1} and \eqref{eq:SnchpT2},
and suitably adjusting the constant $b>0$ in the exponential bounds, 
we arrive at 
\begin{equation}\label{eq:SnchpErrBd}
\| u - Q^{-s}_k(\calL_{hp}) f \|_{\mathbb{H}^s(\Omega)}
\lesssim \exp(-bq) \;.
\end{equation}
Given that the approximation $u_{K,hp}$ involves the solution
of $O(K) = O(q^2)$ reaction-diffusion problems, each of which requires
$O(q^3)$ DOF, the error bound \eqref{eq:SnchpErrBd}
in terms of the total number of degrees of freedom $N_{DOF}$
reads
\begin{equation}\label{eq:SnchpErrBd2}
\| u - Q^{-s}_k(\calL_{hp}) f \|_{\mathbb{H}^s(\Omega)}
\leq 
C \exp(-b\sqrt[5]{N_{DOF}}) 
\end{equation}
with constants $b$, $C>0$ that are independent of $N_{DOF}$. 
We have thus shown: 
\begin{theorem}
\label{thm:BKminMesh}
Let $\Omega$ be a curvilinear polygon as defined in Section~\ref{sec:GeoPrel}, 
let $A$, $f$ satisfy \eqref{eq:AnRegAc}, and let $A$ be uniformly symmetric positive 
definite on $\Omega$. 
Let $u$ be the solution to \eqref{fl=f_bdddom}, and let 
its fully discrete approximation be given by the \BK 
approximation \eqref{eq:SincSum} in conjunction with  
the $hp$-FE approximation of $w_j^{hp}$ in \eqref{eq:BK-CaseA} 
with the $hp$-FE spaces $S^q_0(\Omega, \Tmin(\varepsilon_j))$
on the minimal boundary layer meshes $\Tmin(\varepsilon_j)$.
Choose further the parameters $q \simeq L \simeq K$ 
and 
let $N = \sum\limits_{|j| \leq K}\operatorname{dim} S^q_0(\Omega, \Tmin(\varepsilon_j))$ 
denote the total number of degrees of freedom.

Then there exists a $\lambda_0$ 
(depending on $\Omega$, $A$, $c$, and the parameters characterizing the mesh family $\Tmin$)
such that 
for any $\lambda \in (0,\lambda_0]$ there are constants $C$, $b > 0$ (depending additionally
on the implied constants in $q \simeq L \simeq K$) such that 
$$
\|\calL^{-s} f - Q_k^{-s}(\calL_{hp}) f\|_{\Hs} \leq C \exp(-b \sqrt[5]{N}).
$$
\end{theorem}
\subsubsection{Case B}
\label{sec:SincBKcaseB}
Instead of approximating the problems \eqref{RDPBweak} from 
individual spaces, one may approximate them from the same $hp$-FE space in $\Omega$.
Specifically, we define the approximations
$w_j^{hp}$ by: 
Find $w_j^{hp} \in S^q_0(\Omega,\Tg)$ such that 
\begin{equation}
\label{eq:BK-CaseB}
\forall v \in S^q_0(\Omega, \Tg)\, \colon\,
a_{\eps_j^2,\Omega}(w_j^{hp},v) = \langle f,v\rangle  \;. 
\end{equation}
%
\begin{theorem} \label{thm:BKTPMesh}
Let $\Omega$ be a curvilinear polygon as defined in Section~\ref{sec:GeoPrel}
and assume that $A$, $f$ satisfy \eqref{eq:AnRegAc} and that $A$ is uniformly symmetric positive definite. 
Let $u$ be solution to \eqref{fl=f_bdddom}, and 
let its discrete approximation 
$Q_k^{-s}(\calL_{hp}) f$ 
be given by 
\eqref{eq:SincSum} in conjunction with 
\eqref{eq:BK-CaseB}. Fix $c_1 > 0$.
Let $K \simeq q \simeq L  = n$ 
and let $N = (2K+1)\operatorname{dim}  S^q_0(\Omega, \Tg) \sim q^6 $ 
denote the total number of degrees of freedom. 

Then, under the scale resolution condition
\begin{equation}
\label{eq:scale-res-BK} 
\sigma^L \leq c_1 e^{-K/2}
\end{equation}
there are constants 
$C$, $b >0$ (depending on $\Omega$, $A$, $f$, $c_1$, $\sigma$, and the 
analyticity properties of the macro triangulation)
such that
\begin{equation*}
\|\calL^{-s} f - Q_k^{-s}(\calL_{hp} f)\|_{\Hs} \leq C \exp(-b \sqrt[6]{N}).
\end{equation*}
\end{theorem}
\begin{proof}
The proof follows the arguments of Theorem~\ref{thm:BKminMesh}. 
Since $e^{-K/2} =\min_{j} \varepsilon_j$ is the smallest scale,  
the scale resolution condition (\ref{eq:scale-res-BK}) ensures that 
Proposition~\ref{prop:singular-approx} is applicable. 
\end{proof}
%
\section{Numerical experiments} 
\label{S:NumExp}
We consider the problem \eqref{fl=f_bdddom} with 
diffusion coefficient $A = I$, \ie $\mathcal{L}^s = (-\Delta)^s$.  
The domain $\Omega$ is chosen as either the unit square $\Omega_1 = (0,1)^2$, the  so-called $L$-shaped 
polygonal domain $\Omega_2 \subset \R^2$ determined by the vertices
$\{ (0,0), (1,0), (1,1), (-1,1), (-1,-1), (0,-1)\}$,
or the square domain with a slit $\Omega_3 = (-1,1)^2 \setminus (-1,0] \times \{0\}$. 
As we are in particular interested in smooth, but 
possibly non-compatible data $f$ in all the numerical examples we take
\begin{equation}
  \label{eq:nonsmooth_f}
  f(x_1,x_2) \equiv 1 \qquad \mbox{in } \Omega.
\end{equation} 
Notice that, in this case, $f$ is analytic on $\overline{\Omega}$ 
but $f\in \Ws$ only for $ s > 1/2 $ due to boundary incompatibility
(cf.\ Remark~\ref{remk:compatibility}).
The exact solution is not known, so that
the error is estimated numerically with reference to
an accurate numerical solution. 
The error measure is always the functional
\begin{equation}
  \label{eq:err_defn}
  e(\tilde u) = \left|d_s\int_\Omega f (u^{\text{fine}} - \tilde u )\diff x'\right|^{1/2},
\end{equation}
where $u^{\text{fine}}$ is the numerical solution obtained on a fine mesh. 
Note that for the Galerkin method on the cylinder $\C$ 
(i.e., \EhpFEM in {\bf Case B})
this error measure is equivalent to the energy norm if $u^{\text{fine}}$ is replaced by the exact solution $u$:
\[
  \|u-\tr{\ue^p}\|^2_{\Hs}
  \lesssim \|\nabla (\ue- \ue^p)\|^2_{L^2(y^\alpha,\C)}
  = d_s\int_\Omega f (u - \tr{\ue^p} )\diff x',
  \]
  where $\ue^p$ denotes the discrete solution in $\C_\Y$.

In Figure~\ref{fig:meshes} we show examples of the meshes used for the three domains. 
These are constructed using the Netgen/NGSolve package \cite{netgen1}. 
For the square domain $\Omega = \Omega_1$ the resulting mesh 
is the geometric boundary layer mesh $\calT^{L,L}_{geo, \sigma_x}$  
with $L = 4$ and $\sigma_x = 1/4$. 
The same parameters are used in Netgen/NGSolve to construct the meshes for the other two domains, 
with the resulting meshes diverging from the strict definition of $\calT^{L,L}_{geo, \sigma_2}$ 
near the re-entrant corners since these meshes are not constructed 
using mesh patches but instead by applying directly geometric refinement 
towards edges and vertices.
Nevertheless we denote these meshes also by $\calT^{L,L}_{geo, \sigma_x}$ 
and make use of the finite element spaces $S^q_0(\Omega,\calT^{L,L}_{geo,\sigma_x})$. 

  \begin{figure}[th]
    \centering
    \includegraphics[width=0.32\textwidth]{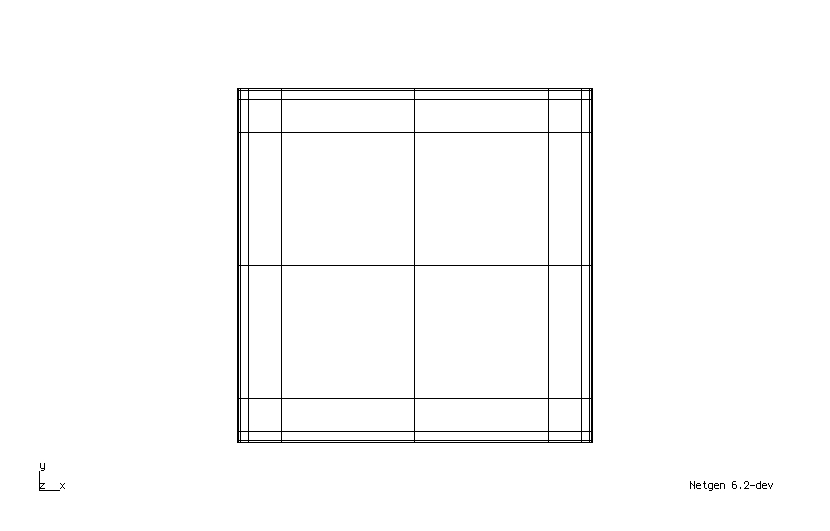}
    \includegraphics[width=0.32\textwidth]{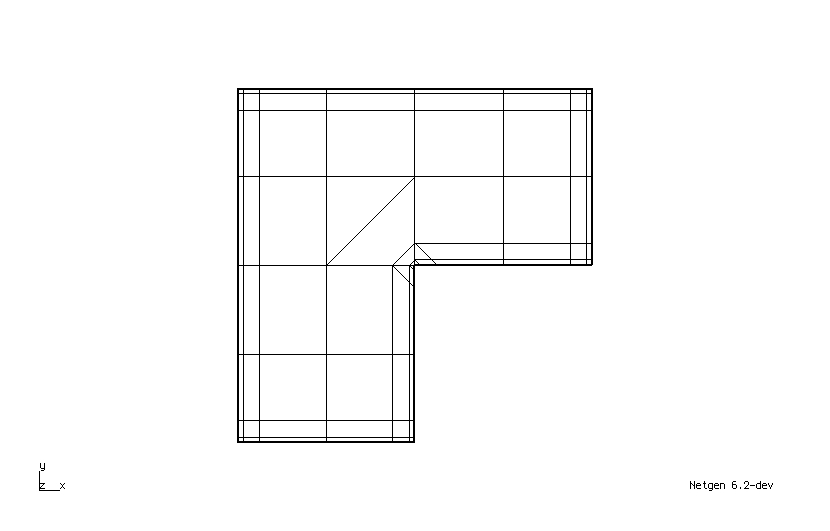}
        \includegraphics[width=0.32\textwidth]{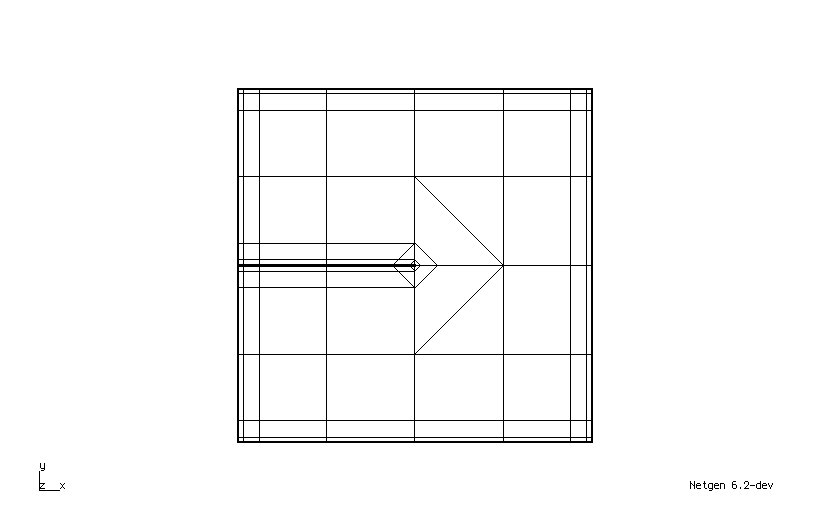}
    \caption{Examples of geometric boundary layer meshes generated 
             by Netgen/NGSolve \cite{netgen1} that are used for the three domains 
             $\Omega = \Omega_j$, $j = 1,2,3$, ordered from left to right. 
    \label{fig:meshes}}
  \end{figure}

  Given a polynomial order $p \geq 1$, in both approaches the finite element space in $\Omega$ is  $S^q_0(\Omega,\calT^{L,L}_{geo,\sigma_x})$ with uniform polynomial degree $q = p$, number of levels $L = p$ and $\sigma_x = 1/4$. Next, we describe the parameters used in the $hp$-FEM on $(0,\Y)$ and the quadrature in the Balakrishnan formula.

  For the extended problem, on the geometric mesh $\calG^M_{geo,\sigma_y}$ in $(0,\Y)$ as defined in Section \ref{S:GeoMes0Y} we use FE-spaces $S_{\{\Y\}}^{\bmr}((0,\Y),\calG^M_{geo,\sigma_y})$.  Given a  polynomial degree $p \geq 1$, in the definition of these spaces we use
$\Y = \tfrac12 p$, $\sigma_y = 1/4$ and $M = \operatorname{round}(0.79\, p/s)$\footnote{{The choice $M = \operatorname{round}(0.79\, p/s)$ resulted from 
equilibrating (an upper bound) for the semidiscretization error associated 
with  $[y_0,y_1]$ and $[y_{M-1},y_M]$}}, and a uniform degree vector $\bmr = (p,\dots,p)$. 

For simplicity in the analysis of the sinc quadrature, 
we used a symmetric approximation \eqref{eq:Sincu}. 
For the numerical experiments in order to obtain a more efficient scheme 
we have followed \cite{BoPascFracRegAcc2017} to define the quadrature as 
 \begin{equation}
   \label{eq:Sinc_practice}
 Q_k^{-s}(\mathcal{L}) f := \frac{k \sin (\pi s)}{\pi}\sum_{\ell = -K_1}^{K_2} e^{-sy_\ell} ( I + e^{-y_\ell}\mathcal{L})^{-1} f,
\end{equation}
with $y_\ell =\ell k$ and the number of quadrature points chosen as
\[
K_1 =\left \lceil \frac{\pi^2}{2(1-s)k^2}\right\rceil\;, \qquad
K_2 = \left\lceil \frac{\pi^2}{s k^2}\right\rceil\;.
\]
For the given polynomial order $p \geq 1$, we set $k = \frac{4}{3}p^{-1}$.
 
%

We now compare the convergence of the two schemes. We plot the error against the polynomial degree $p$ and against $N_{\text{ls}}^{1/2}$, where $N_{\text{ls}}$ is the number of linear systems that need to be solved. The convergence curves for the square domain $\Omega_1$ are shown in Figure~\ref{fig:square_conv}, for the L-shaped domain $\Omega_2$  in Figure~\ref{fig:lshape_conv}, and for the slit domain $\Omega_3$ in Figure~\ref{fig:slit_conv}. 
For all three domains we clearly see exponential convergence as the polynomial order is increased. 
Also, the \EhpFEM requires significantly fewer linear systems 
to be solved to achieve the same accuracy as the \BK.
We should, however, also note that the eigenvalue problem \eqref{eq:eigenvalue-problem} becomes  ill-conditioned for increasing $p$ and much higher accuracy than the one shown in the above figures cannot be obtained using our approach for the extension problem. No such accuracy limitations could be seen for the sinc approach.

\begin{figure}[th]
  \centering
  \includegraphics[width=0.45\textwidth]{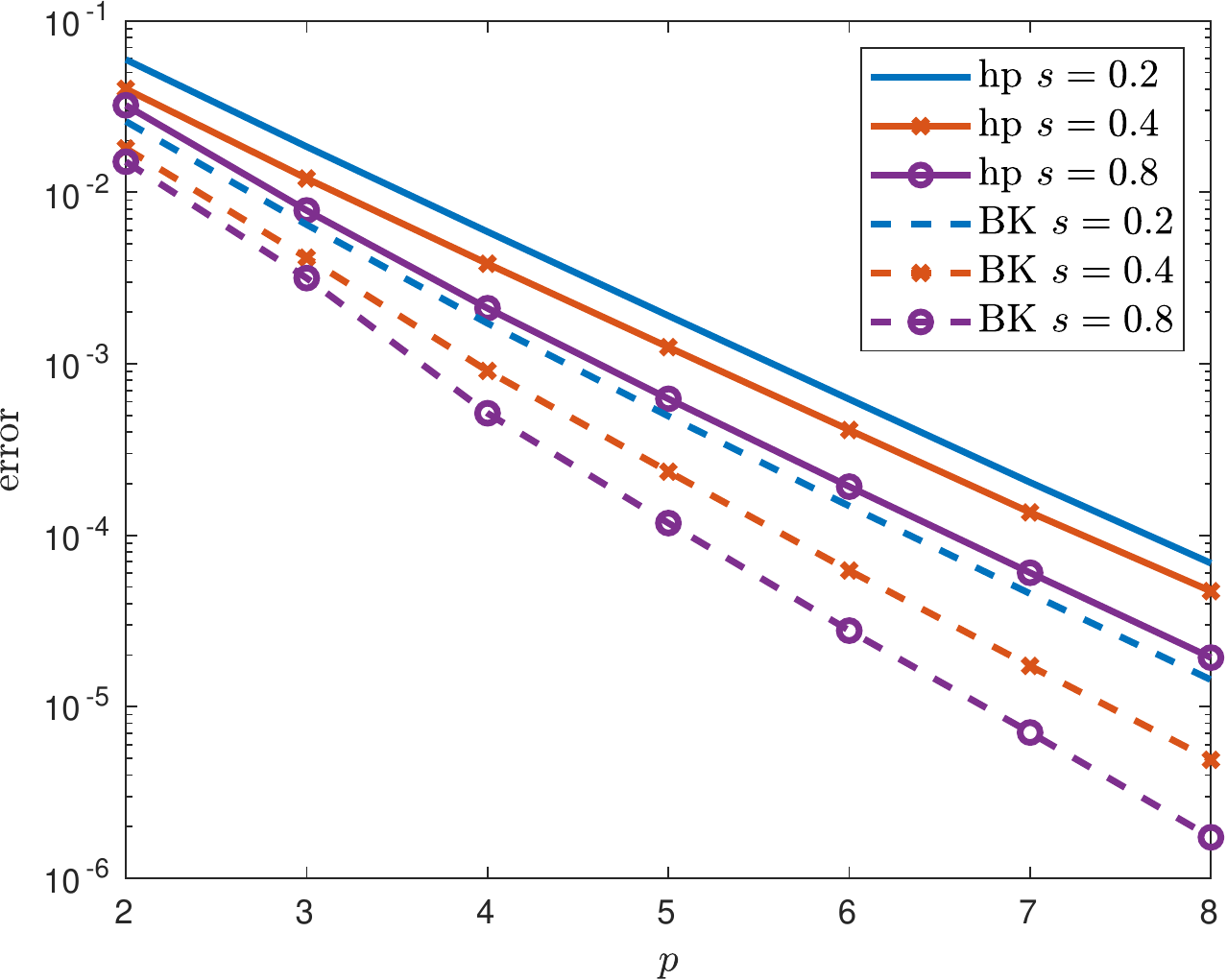}
  \includegraphics[width=0.45\textwidth]{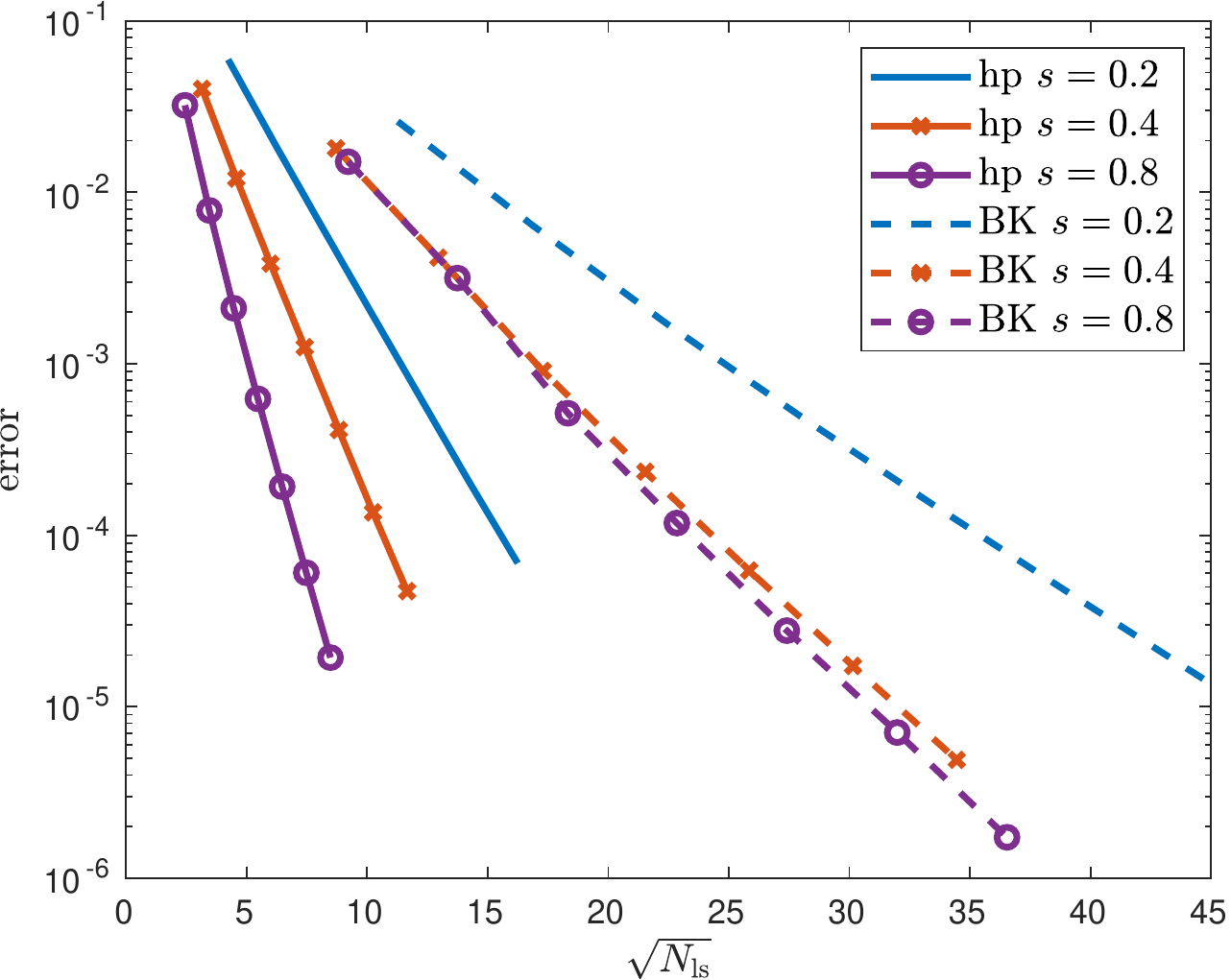}
  \caption{Error convergence of the 
\EhpFEM and the \BK for the domain $\Omega_1$, 
depicted versus the polynomial degree $p$ and $N_{\text{ls}}^{1/2}$, 
where $N_{\text{ls}}$ denote the number of linear systems that need to be solved. 
Solid lines correspond to \EhpFEM and dashed lines to the \BKp.
Results for $s = 0.2$, $s = 0.4$ and $s = 0.8$ are shown.}
  \label{fig:square_conv}
\end{figure}

\begin{figure}[th]
  \centering
  \includegraphics[width=0.45\textwidth]{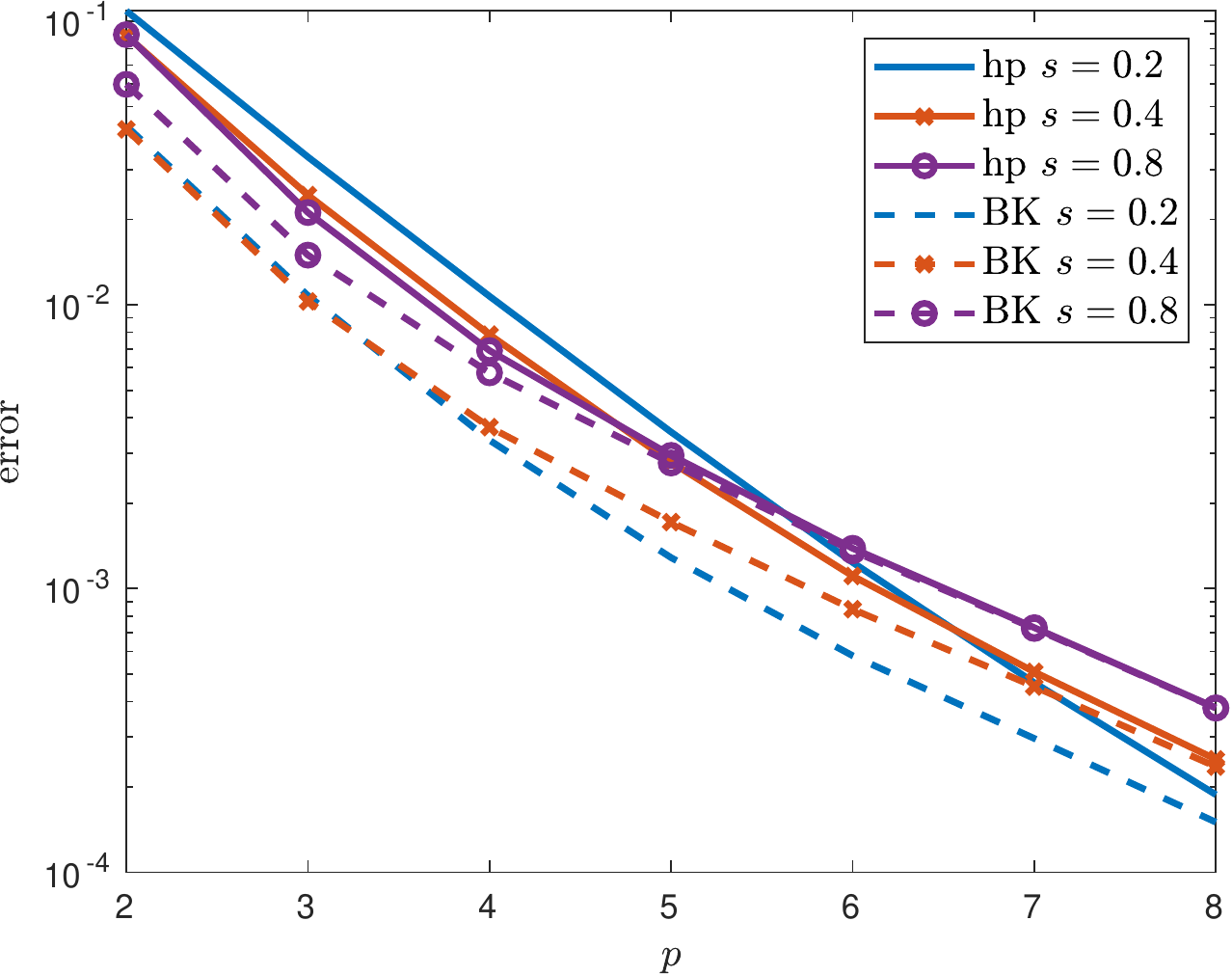}
  \includegraphics[width=0.45\textwidth]{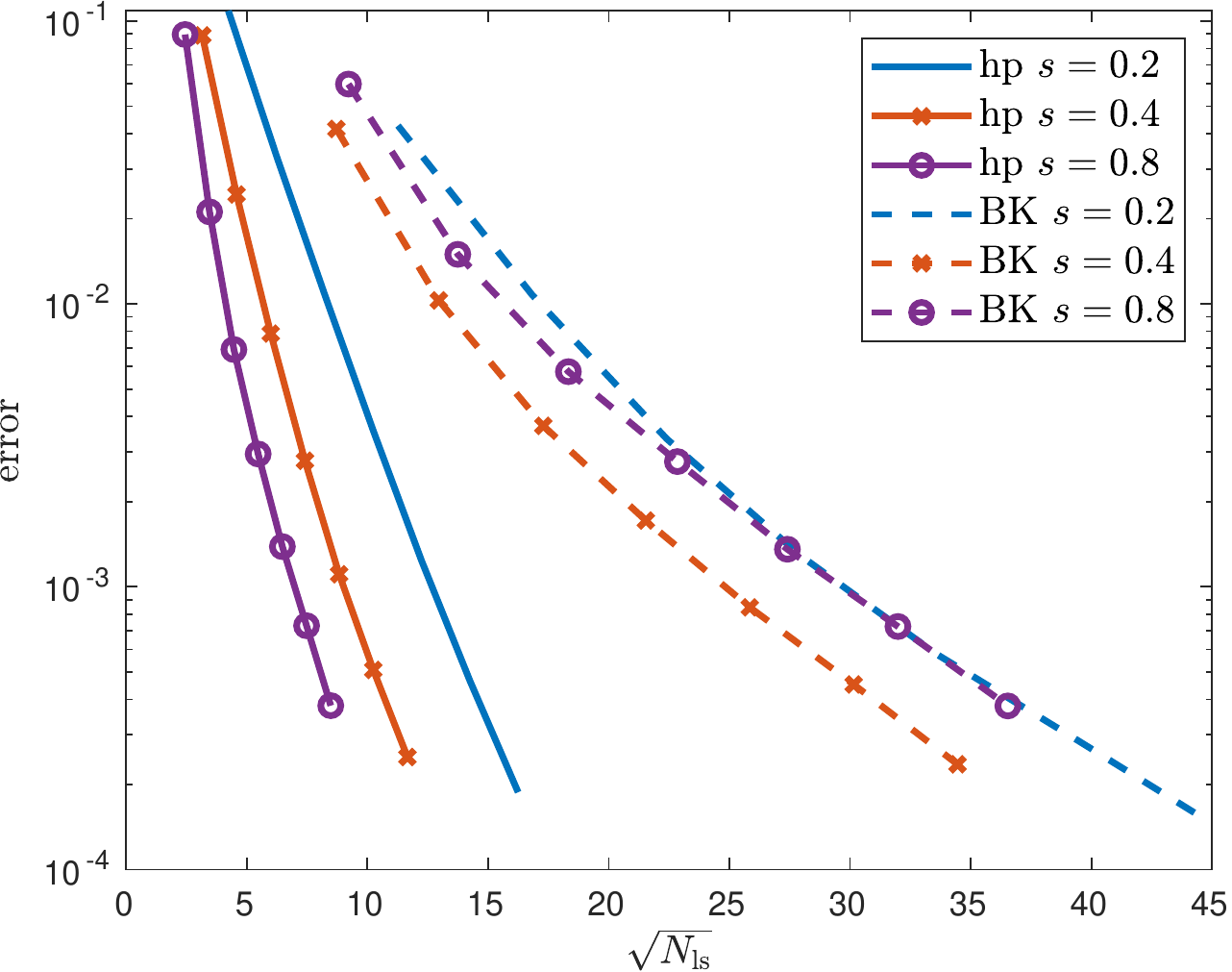}
  \caption{Error convergence  of the \EhpFEM and the \BK
           for L-shaped domain $\Omega_2$, depicted versus the polynomial degree 
           $p$ and $N_{\text{ls}}^{1/2}$, the square root of the number of linear systems to be solved. }
  \label{fig:lshape_conv}
\end{figure}

\begin{figure}[th]
  \centering
  \includegraphics[width=0.45\textwidth]{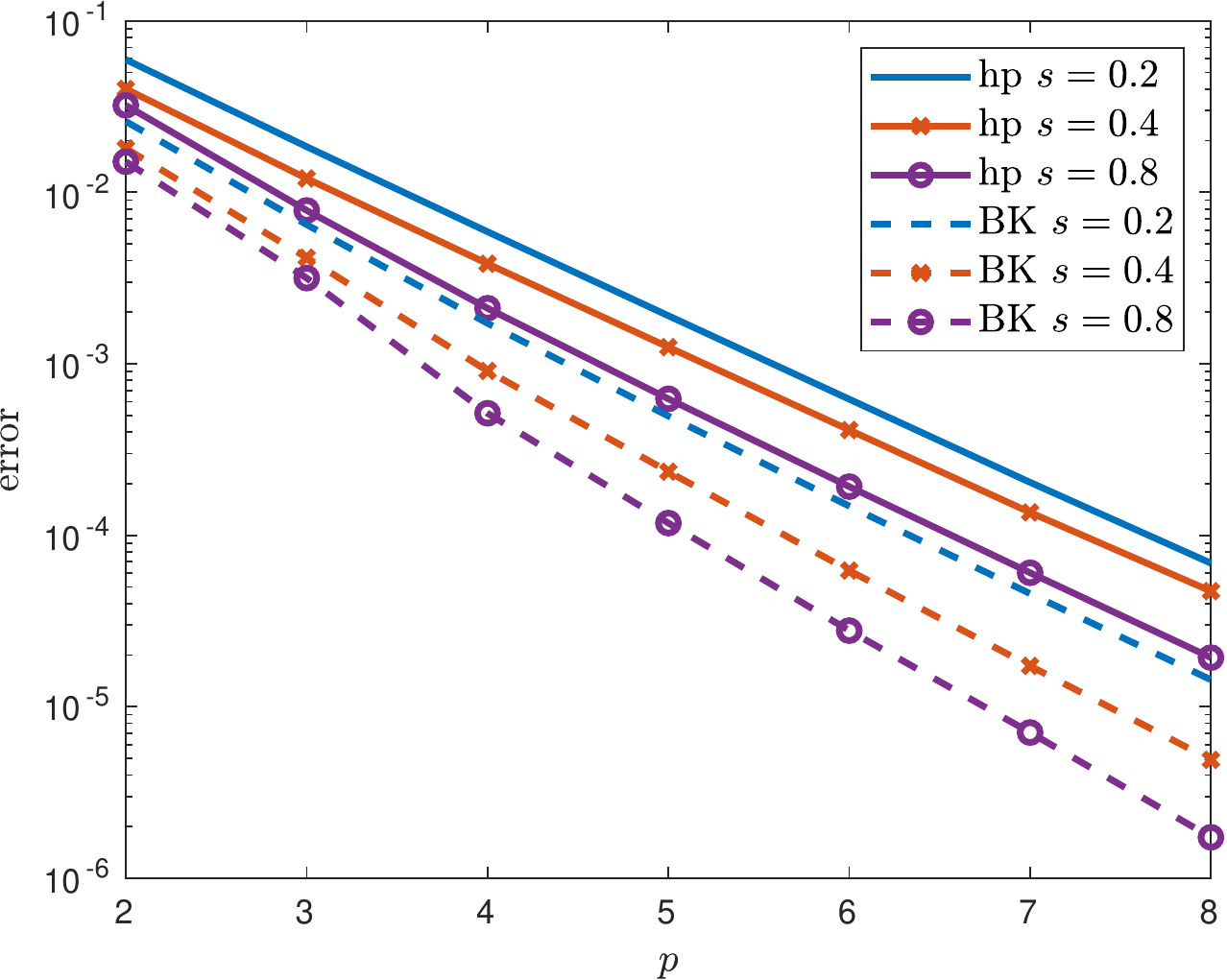}
  \includegraphics[width=0.45\textwidth]{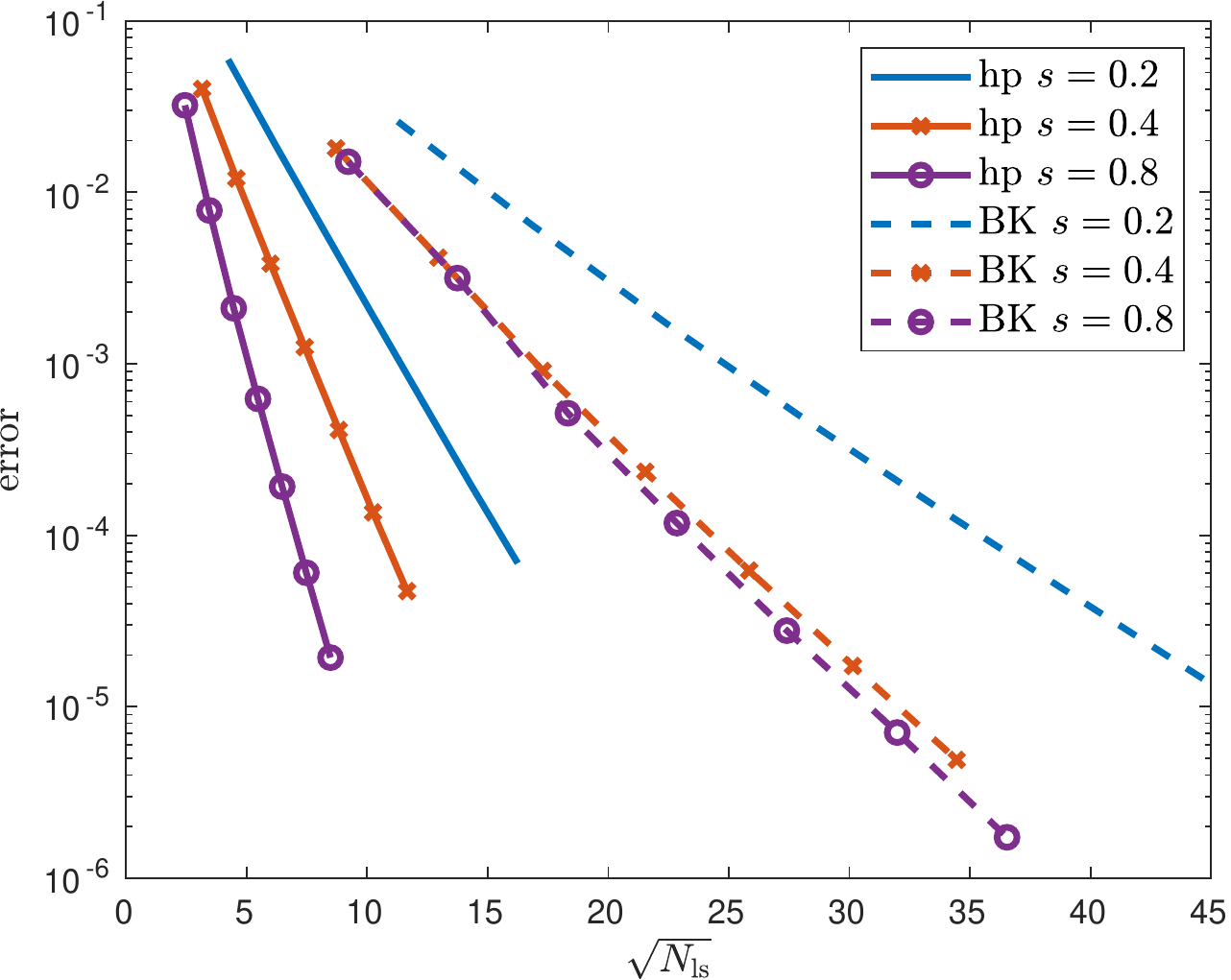}
  \caption{Error convergence of the \EhpFEM and the \BK for the 
           slit domain $\Omega_3$, depicted versus the polynomial degree $p$ and 
           $N_{\text{ls}}^{1/2}$, the square root of the number of linear systems to be solved. }
  \label{fig:slit_conv}
\end{figure}

\section{ Extensions and Conclusions}
\label{S:Concl}

\subsection{Fractional Diffusion on Manifolds}
\label{S:FrcLapManif}

We describe next 
\emph{fractional surface diffusion operators on analytic manifolds}, that are of interest in some application.
It exploits the admissibility of nonconstant, analytic coefficient $A(x')$ in the diffusion operator $\calL$.
The numerical schemes and  their analysis as described above can be extended to this setting as well.

Let $\Gamma \subset \R^3$ denote a compact, orientable analytic manifold 
(e.g. \cite{Aubin1998Riemannian}). 
We think of bounded, analytic surfaces such as the unit sphere $\bbS^2 \subset \R^3$.
Let $\Gamma$ be covered by a finite atlas of analytic charts $\{\chi_j\}_{j=1}^J$.
In a generic analytic chart $\chi$ of $\Gamma$, consider the polygonal domain
$\widetilde{\Gamma} = \chi(\Omega)\subset \Gamma$ 
where the parameter domain $\Omega \subset \R^2$ 
of the chart $\chi$ is a curvilinear polygon in the sense of Section \ref{sec:GeoPrel}.
On $\Gamma$, introduce the surface (Lebesgue)measure $\mu$.
On $\widetilde{\Gamma}$, for given $\tilde{f} \in L^2(\Gamma,\mu;\R)$, 
consider the
\emph{Dirichlet problem for the surface diffusion operator $\widetilde{\calL}$}:
find $u_\Gamma$ such that
\begin{equation}\label{eq:SurfDiff}
\widetilde{\calL} u_\Gamma
:= 
-\DIV_\Gamma (\tilde{A} \nabla_\Gamma u_\Gamma) = \tilde{f} 
\quad \mbox{on}\quad \widetilde{\Gamma}\;,
\qquad 
{u_\Gamma}|_{\partial \widetilde{\Gamma} } = 0 
\;.
\end{equation}
Here, the ``diffusion coefficient'' $\tilde{A}$ in 
\eqref{eq:SurfDiff} is a symmetric, uniformly in $\Gamma$ positive definite 
linear map acting on the tangent bundle of $\Gamma$, and $\nabla_\Gamma$ and 
$\DIV_\Gamma$ denote the surface gradient and divergence differential operators on $\Gamma$,
respectively (see \cite{Aubin1998Riemannian}).
With Sobolev spaces on $\Gamma$ invariantly defined in the usual fashion
(e.g., \cite{Aubin1998Riemannian}), 
the surface diffusion operator $\widetilde{\calL}$ in \eqref{eq:SurfDiff}
extends to a boundedly invertible, self-adjoint operator 
$\widetilde{\calL}: H^1_0(\widetilde{\Gamma};\mu) \to H^{-1}(\widetilde{\Gamma};\mu) =  ( H^1_0(\widetilde{\Gamma};\mu) )^*$ 
(duality with respect to $L^2(\widetilde{\Gamma};\mu) \simeq (L^2(\widetilde{\Gamma};\mu))^*$)
whose inverse $\widetilde{\calL}^{-1}$ is a compact, self-adjoint operator on $L^2(\widetilde{\Gamma};\mu)$.
The spectral theorem implies that $\widetilde{\calL}^{-1}$ admits a countable sequence of eigenpairs 
$(\tilde{\lambda}_k,\tilde{\varphi}_k)_{k\geq 1}$ whose eigenvectors $\tilde{\varphi}_k$ 
can be normalized so that they constitute an ONB of $L^2(\widetilde{\Gamma};\mu)$.
With the ONB $\{ \tilde{\varphi}_k \}_{k\geq 1}$, 
fractional Sobolev spaces on $\widetilde{\Gamma}$ can 
be defined as in \eqref{def:Hs}, i.e. for $0<s<1$,
\begin{equation}\label{def:MHs}
{\mathbb H}^s(\tilde{\Gamma}) 
:= 
\left\{ w = \sum_{k=1}^\infty w_k \tilde{\varphi}_k: 
        \| w \|_{{\mathbb H}^s(\tilde{\Gamma})}^2 = \sum_{k=1}^{\infty} \tilde{\lambda}_k^s w_k^2 < \infty \right\} \;.
\end{equation}
The space ${\mathbb H}^s(\tilde{\Gamma})$ can be characterized by (real) interpolation:
There holds ${\mathbb H}^s(\tilde{\Gamma}) = (H^1_0(\widetilde{\Gamma};\mu), L^2(\widetilde{\Gamma};\mu))_{s,2}$ 
for $0<s<1$.
As in \eqref{fl=f_bdddom}, with the family $(\tilde{\lambda}_k,\tilde{\varphi}_k)_{k\geq 1}$
we may define the spectral fractional Laplacian
$\widetilde{\calL}^s = (\calL,I)_{s,2}$ 
by interpolation of linear operators (e.g. \cite{SGKreinIntrpOp71}). 
The arguments in \cite{CS:07} extend verbatim the localization \eqref{alpha_harm_intro} to 
the present setting. In particular, the spectral fractional diffusion operator on
$\widetilde{\Gamma}$ with homogeneous Dirichlet boundary conditions on $\partial \widetilde{\calL}$
admits a localization on the cylinder $\widetilde{\calC} = \widetilde{\Gamma} \times (0,\infty)$.

Pulling back the problem \eqref{eq:SurfDiff} via $\chi$ into the (Euclidean)
chart domain $\Omega\subset \R^2$, 
the Dirichlet problem for the fractional power $s\in (0,1)$ of the surface
diffusion \eqref{eq:SurfDiff} in $\widetilde{\Gamma} = \chi(\Omega)$ 
reduces to \eqref{alpha_harm_intro} where the bilinear form 
\eqref{eq:blfOmega} and diffusion coefficient $A$ are given by 
$$
A(x') = G(x')^{\top}(\tilde{A}\circ\chi)(x') G(x') \;,\quad x'\in \Omega,
$$
with $G(x') = D\chi(x'):\Omega \to \R^{3\times2}$ denoting 
the (assumed analytic in $\overline{\Omega}$) 
metric of $\M$ in chart $\chi$.
The real-analyticity of compositions, sums and product of
real-analytic functions implies that $A(x')$ satisfies
\eqref{eq:AnRegAc} in $\Omega$, so that the ensuing
mathematical results also apply to \eqref{fl=f_bdddom}
with \eqref{eq:SurfDiff}.
\subsection{$N$-widths of solution sets}
\label{sec:Nwidth}
The $hp$-approximation rate bounds for either the \EhpFEM 
(Theorems~\ref{thm:ExpCnvI}, \ref{thm:ExpCnvII}) 
and the \BK (Theorems~\ref{thm:BKminMesh}, \ref{thm:BKTPMesh}) 
imply 
exponential bounds on $N$-widths of solution sets of \eqref{fl=f_bdddom} 
in a curvilinear polygon $\Omega$ as defined in Section~\ref{sec:GeoPrel},
with the data $A$ and $f$ satisfying the conditions in Section~\ref{sec:FracDiff}.
Such bounds are well-known to determine the rate of convergence of 
so-called \emph{reduced basis methods} (see \cite{RBM} and the references there).

We recall that, for a normed linear space $X$
(with norm $\| \circ \|_X$) and for a compact subset $\calK \subset X$,
the $N$-width of $\calK$ in $X$ 
is given by
\begin{equation}\label{eq:KlmNwid}
d_N(\calK,X) = \inf_{E_N} \sup_{f\in \calK} \inf_{g\in E_N} \| f - g \|_X.
\end{equation}
Here, the first infimum is taken over all subspaces $E_N$ of $X$
of dimension $N\in \bbN$.
Subspace sequences $\{ E_N \}_{N\geq 1}$ that attain the rates
of $d_N(\calK,X)$ in \eqref{eq:KlmNwid} as $N\to \infty$
can be realized numerically by (generally non-polynomial)
so-called \emph{reduced bases} (see, e.g., \cite{RBM} and the references there).
Here, we fix a set $G \subset {\mathbb C}^2$ containing 
$\overline{\Omega}$ and 
choose $\calA \subset L^2(\Omega)$ as the
set of functions $f:\Omega\to \bbR$
that admit a holomorphic extension to 
$G$ with $\| f \|_{L^\infty(\Omega)} \leq 1$.
Then $\calK := \calL^{-s}\calA \subset \Hs$ 
is a compact subset by the continuity of $ \calL^{-s}: {\mathbb H}^{-s}(\Omega) \to \Hs$
and the compact embedding $L^2(\Omega) \subset {\mathbb H}^{-s}(\Omega)$.
We choose $X=\Hs$  in \eqref{eq:KlmNwid}. 

Then, from Theorem~\ref{thm:BKTPMesh} and the fact that 
$Q_k^{-s}(\calL_{hp}) f \in S^q_0(\Omega, \Tg)$, 
with the choices of parameters in Theorem~\ref{thm:BKTPMesh},
$N = {\rm dim}(S^q_0(\Omega, \Tg) = O(q^4)$,
for $\calK$ as above and $E_N = S^q_0(\Omega, \Tg)$ 
follows the (constructive) bound 
\begin{equation}\label{eq:HsNwidth1}
d_N(\calK,X) \lesssim \exp(-b\sqrt[4]{N})
\end{equation}
for some constant $b>0$ independent of $N$.

We also mention that the argument in \cite{JMMnwidth} 
can be adapted to the setting of \eqref{fl=f_bdddom} in Section~\ref{sec:GeoPrel},
resulting in the (sharp) nonconstructive bound
\begin{equation}\label{eq:HsNwidth2}
d_N(\calK,X) \lesssim \exp(-b\sqrt{N})\;.
\end{equation}
We refer to 
\cite{HarbChenNar19,danczul2020reduced,bonito2020reduced}
for numerical approximation of \eqref{fl=f_bdddom} using reduced basis methods. 
\subsection{Conclusions}
\label{sec:Concl}
For the Dirichlet problem of the spectral, 
fractional diffusion operator $\calL^s$ with $0<s<1$
in a bounded, polygonal domain $\Omega \subset \mathbb{R}^2$, 
we proposed two $hp$-FE discretizations. 
The first discretization, already considered in \cite{MPSV17,BMNOSS17_732},
is based on the CS-extension upon $hp$-FE 
\emph{semi-discretization in the extended variable}. 
Subsequent diagonalization leads to a decoupled system \eqref{eq:decoupled-problems} 
of $\M$ linear and local, 
singularly perturbed second order reaction-diffusion problems in $\Omega$.
Invoking analytic regularity results for these
problems from \cite{MelCS_RegSingPert,melenk02}, 
and 
\emph{robust exponential convergence} of $hp$-FEM for reaction-diffusion
problems in polygons from 
\cite{melenk-schwab98,melenk02,banjai-melenk-schwab19-RD}, 
an exponential convergence rate bound
$C\exp(-b\sqrt[6]{N_{DOF}})$ with respect to the total number of degrees of freedom, $N_{DOF}$, 
which are used in the tensor-product $hp$-FE discretization, 
is established in the fractional Sobolev norm $\mathbb{H}^s(\Omega)$.
We add that the variational semi-discretization in 
Section~\ref{S:diagonalization-abstract-setting}
    with respect to the extruded variable $y$ offers the possibility for 
    {\sl residual a posteriori} error estimation. 

The second discretization is based on the spectral integral representation
of $\calL^{-s}$ due to Balakrishnan \cite{Balakr1960}. 
A sinc quadrature discretization \cite{Stenger83} 
approximates the spectral integral by an (exponentially convergent
\cite{Stenger83,BP:13}) finite linear combination of solutions of decoupled 
elliptic reaction diffusion problems in $\Omega$ with analytic input data. 
Drawing once more on analytic regularity
and robust exponential convergence of $hp$-FEM  
\cite{MelCS_RegSingPert,melenk-schwab98,melenk02,banjai-melenk-schwab19-RD},
we prove exponential convergence also for this approach.
A computable {\sl a posteriori} bound for the semidiscretization error 
    incurred for the \BK approach does not seem to be available currently.

The theoretical convergence rate bounds are verified in a series of numerical
experiments. These show, in particular, that exponential convergence 
is realized in the practical range of discretization parameters.
They also indicate a number of practical issues, such as conditioning or 
algorithmic steering parameter selection, which are beyond the scope of 
the mathematical convergence analysis.
We point out that the proposed algorithms and the exponential convergence
results extend in several directions: 
besides homogeneous Dirichlet boundary conditions, 
also mixed, Dirichlet-Neumann boundary conditions, and operators with a nonzero first order term
could be considered. In either case, the proposed algorithms extend readily.
The main result is the construction of $hp$-FE discretizations with
robust exponential convergence rates
for spectral fractional diffusion in polygonal domains $\Omega\subset \bbR^2$. 
Similar results hold in bounded intervals $\Omega\subset \bbR^1$ 
(we refer to \cite{BMNOSS17_732} for details).
In polyhedral $\Omega\subset \bbR^3$, the present line of analysis is also applicable;
however, exponential convergence and analytic regularity of $hp$-FEM
for reaction-diffusion problems in space dimension $d=3$ does not appear to be 
available to date.
We considered fractional powers only for self-adjoint, second-order elliptic
divergence-form differential operators $\calL w = - \DIV( A \GRAD w )$ in $\Omega$.
The arguments for the \BK extend to non-selfadjoint operators
which include first-order terms via \cite{BoPascFracRegAcc2017},
provided suitable $hp$-FEM for advection-reaction-diffusion problems in $\Omega$
are available (e.g. \cite{MS99_325}).

The present analysis is indicative for achieving high, algebraic rate $p+1-s$
of convergence in $\mathbb{H}^{s}(\Omega)$ by $h$-version FEM 
of fixed order $p\geq 1$ in $\Omega$. As in $hp$-FEM, this will require
anisotropic mesh refinement aligned with $\partial\Omega$, 
ie., so-called ``boundary-layer'' meshes. 
Several constructions are available 
(see, e.g., \cite{SSX98_321} for so-called
``exponential boundary layer meshes'' and \cite{RoosStynsTobiska2ndEd} for
  so-called ``Shiskin meshes'').
We refrain from developing details for this approach which can be analyzed 
along the lines of the present paper.

\begin{thebibliography}{10}

\bibitem{HarbChenNar19}
Harbir Antil, Yanlai Chen, and Akil Narayan.
\newblock Reduced basis methods for fractional {L}aplace equations via
  extension.
\newblock {\em SIAM J. Sci. Comput.}, 41(6):A3552--A3575, 2019.

\bibitem{apel-melenk17}
T.~Apel and J.M. Melenk.
\newblock Interpolation and quasi-interpolation in $h$- and $hp$-version finite
  element spaces.
\newblock In E.~Stein, R.~de~Borst, and T.J.R. Hughes, editors, {\em
  Encyclopedia of Computational Mechanics}, pages 1--33. John Wiley \& Sons,
  Chichester, UK, second edition, 2018.
\newblock extended preprint at {\tt
  http://www.asc.tuwien.ac.at/preprint/2015/asc39x2015.pdf}.

\bibitem{Aubin1998Riemannian}
Thierry Aubin.
\newblock {\em Some nonlinear problems in {R}iemannian geometry}.
\newblock Springer Monographs in Mathematics. Springer-Verlag, Berlin, 1998.

\bibitem{babuska-guo86a}
I.~Babu{\v s}ka and B.Q. Guo.
\newblock The $h-p$ version of the finite element method. {P}art 1: {T}he basic
  approximation results.
\newblock {\em Computational Mechanics}, 1:21--41, 1986.

\bibitem{babuska-guo86b}
I.~Babu{\v s}ka and B.Q. Guo.
\newblock The $h-p$ version of the finite element method. {P}art 2: {G}eneral
  results and applications.
\newblock {\em Computational Mechanics}, 1:203--220, 1986.

\bibitem{Balakr1960}
A.~V. Balakrishnan.
\newblock Fractional powers of closed operators and the semigroups generated by
  them.
\newblock {\em Pacific J. Math.}, 10:419--437, 1960.

\bibitem{banjai-melenk-schwab19-RD}
L.~Banjai, J.M. Melenk, and Ch. Schwab.
\newblock {$hp$-FEM} for reaction-diffusion equations. {II}: Robust exponential
  convergence for multiple length scales in corner domains.
\newblock Technical Report 2020-28, Seminar for Applied Mathematics, ETH
  Z{\"u}rich, Switzerland, 2020.

\bibitem{BMNOSS17_732}
Lehel Banjai, Jens~M. Melenk, Ricardo~H. Nochetto, Enrique Ot\'{a}rola,
  Abner~J. Salgado, and Christoph Schwab.
\newblock Tensor {FEM} for spectral fractional diffusion.
\newblock {\em Found. Comput. Math.}, 19(4):901--962, 2019.

\bibitem{BoPascFracRegAcc2017}
A.~{Bonito}, W.~{Lei}, and J.~E. {Pasciak}.
\newblock {On sinc quadrature approximations of fractional powers of regularly
  accretive operators}.
\newblock {\em J.\ Num.\ Math.}, 27(2):57--68, 2019.

\bibitem{BP:13}
A.~Bonito and J.E. Pasciak.
\newblock Numerical approximation of fractional powers of elliptic operators.
\newblock {\em Math. Comp.}, 84(295):2083--2110, 2015.

\bibitem{BonitoEtAl_FracSurv2017}
Andrea Bonito, Juan~Pablo Borthagaray, Ricardo~H. Nochetto, Enrique
  Ot\'{a}rola, and Abner~J. Salgado.
\newblock Numerical methods for fractional diffusion.
\newblock {\em Comput. Vis. Sci.}, 19(5-6):19--46, 2018.

\bibitem{bonito2020reduced}
Andrea Bonito, Diane Guignard, and Ashley~R. Zhang.
\newblock Reduced basis approximations of the solutions to spectral fractional
  diffusion problems, 2020.
\newblock arXiv:1905.01754.

\bibitem{CT:10}
X.~Cabr\'e and J.~Tan.
\newblock Positive solutions of nonlinear problems involving the square root of
  the {L}aplacian.
\newblock {\em Adv. Math.}, 224(5):2052--2093, 2010.

\bibitem{CS:07}
L.~Caffarelli and L.~Silvestre.
\newblock An extension problem related to the fractional {L}aplacian.
\newblock {\em Comm. Part. Diff. Eqs.}, 32(7-9):1245--1260, 2007.

\bibitem{CafStinga16}
L.A. Caffarelli and P.R. Stinga.
\newblock Fractional elliptic equations, {C}accioppoli estimates and
  regularity.
\newblock {\em Ann. Inst. H. Poincar\'e Anal. Non Lin\'eaire}, 33(3):767--807,
  2016.

\bibitem{danczul2020reduced}
Tobias Danczul and Joachim Sch{\"o}berl.
\newblock A reduced basis method for fractional diffusion operators {II}, 2020.
\newblock arXiv:2005.03574.

\bibitem{FstmnMM_hpBalNrm2017}
M.~Faustmann and J.M. Melenk.
\newblock Robust exponential convergence of {$hp$}-{FEM} in balanced norms for
  singularly perturbed reaction-diffusion problems: corner domains.
\newblock {\em Comput. Math. Appl.}, 74(7):1576--1589, 2017.

\bibitem{KarkJMM19}
Michael Karkulik and Jens~Markus Melenk.
\newblock {$\mathcal H$}-matrix approximability of inverses of discretizations
  of the fractional {L}aplacian.
\newblock {\em Adv. Comput. Math.}, 45(5-6):2893--2919, 2019.

\bibitem{SGKreinIntrpOp71}
S.~G. Kre\u{\i}n.
\newblock Interpolation of linear operators, and properties of the solutions of
  elliptic equations.
\newblock In {\em Elliptische {D}ifferentialgleichungen, {B}and {II}}, pages
  155--166. Schriftenreihe Inst. Math. Deutsch. Akad. Wissensch. Berlin, Reihe
  A, Heft 8. Akademie-Verlag, Berlin, 1971.

\bibitem{AinsworthEtAl_FracSurv2018}
A.~{Lischke}, G.~{Pang}, M.~{Gulian}, F.~{Song}, C.~{Glusa}, X.~{Zheng},
  Z.~{Mao}, W.~{Cai}, M.~M. {Meerschaert}, M.~{Ainsworth}, and G.~E.
  {Karniadakis}.
\newblock What is the fractional {L}aplacian? {A} comparative review with new
  results.
\newblock {\em J. Comput. Phys.}, 404:109009, 62, 2020.

\bibitem{LynchR1964}
Robert~E. Lynch, John~R. Rice, and Donald~H. Thomas.
\newblock Direct solution of partial difference equations by tensor product
  methods.
\newblock {\em Numer. Math.}, 6:185--199, 1964.

\bibitem{mcLean}
W.~McLean.
\newblock {\em Strongly elliptic systems and boundary integral equations}.
\newblock Cambridge University Press, Cambridge, 2000.

\bibitem{MPSV17}
Dominik Meidner, Johannes Pfefferer, Klemens Sch\"{u}rholz, and Boris Vexler.
\newblock {$hp$}-finite elements for fractional diffusion.
\newblock {\em SIAM J. Numer. Anal.}, 56(4):2345--2374, 2018.

\bibitem{MS99_325}
Jens Melenk and Christoph Schwab.
\newblock An hp finite element method for convection-diffusion problems in one
  dimension.
\newblock {\em IMA Journal of Numerical Analysis}, 19(3):425--453, 1999.

\bibitem{melenk97}
J.M. Melenk.
\newblock On the robust exponential convergence of {$hp$} finite element method
  for problems with boundary layers.
\newblock {\em IMA J. Numer. Anal.}, 17(4):577--601, 1997.

\bibitem{JMMnwidth}
J.M. Melenk.
\newblock On {$n$}-widths for elliptic problems.
\newblock {\em J. Math. Anal. Appl.}, 247(1):272--289, 2000.

\bibitem{melenk02}
J.M. Melenk.
\newblock {\em {$hp$}-finite element methods for singular perturbations},
  volume 1796 of {\em Lecture Notes in Mathematics}.
\newblock Springer-Verlag, Berlin, 2002.

\bibitem{melenk-schwab98}
J.M. Melenk and Ch. Schwab.
\newblock $hp$ {FEM} for reaction-diffusion equations. {I}. {R}obust
  exponential convergence.
\newblock {\em SIAM J. Numer. Anal.}, 35(4):1520--1557, 1998.

\bibitem{MelCS_RegSingPert}
J.M. Melenk and Ch. Schwab.
\newblock Analytic regularity for a singularly perturbed problem.
\newblock {\em SIAM J. Math. Anal.}, 30(2):379--400, 1999.

\bibitem{MMCAXeno_Balanced2016}
J.M. Melenk and C.~Xenophontos.
\newblock Robust exponential convergence of {$hp$}-{FEM} in balanced norms for
  singularly perturbed reaction-diffusion equations.
\newblock {\em Calcolo}, 53(1):105--132, 2016.

\bibitem{NOS}
R.H. Nochetto, E.~Ot\'arola, and A.J. Salgado.
\newblock A {PDE} approach to fractional diffusion in general domains: a priori
  error analysis.
\newblock {\em Found. Comput. Math.}, 15(3):733--791, 2015.

\bibitem{RBM}
Alfio Quarteroni, Andrea Manzoni, and Federico Negri.
\newblock {\em Reduced basis methods for partial differential equations},
  volume~92 of {\em Unitext}.
\newblock Springer, Cham, 2016.
\newblock An introduction, La Matematica per il 3+2.

\bibitem{RoosStynsTobiska2ndEd}
Hans-G\"org Roos, Martin Stynes, and Lutz Tobiska.
\newblock {\em Robust numerical methods for singularly perturbed differential
  equations}, volume~24 of {\em Springer Series in Computational Mathematics}.
\newblock Springer-Verlag, Berlin, second edition, 2008.
\newblock Convection-diffusion-reaction and flow problems.

\bibitem{RosOton2016Surv}
Xavier Ros-Oton.
\newblock Nonlocal elliptic equations in bounded domains: a survey.
\newblock {\em Publ. Mat.}, 60(1):3--26, 2016.

\bibitem{netgen1}
J.~Sch\"{o}berl.
\newblock Netgen an advancing front 2d/3d-mesh generator based on abstract
  rules.
\newblock {\em J. Comput. Visual. Sci.}, 1:41--52, 1997.

\bibitem{phpSchwab1998}
Ch. Schwab.
\newblock {\em {$p$- and $hp$-Finite Element Methods}}.
\newblock Numerical Mathematics and Scientific Computation. The Clarendon
  Press, Oxford University Press, New York, 1998.
\newblock Theory and applications in solid and fluid mechanics.

\bibitem{schwab-suri96}
Ch. Schwab and M.~Suri.
\newblock The {$p$} and {$hp$} versions of the finite element method for
  problems with boundary layers.
\newblock {\em Math. Comp.}, 65(216):1403--1429, 1996.

\bibitem{SSX98_321}
Ch. Schwab, M.~Suri, and C.A. Xenophontos.
\newblock The $hp$ {Finite Element Method} for problems in mechanics with
  boundary layers.
\newblock {\em Comp. Meth. Appl. Mech. Engg.}, 157(3-4):311--333, 1998.

\bibitem{Stenger83}
Frank Stenger.
\newblock {\em Numerical methods based on sinc and analytic functions},
  volume~20 of {\em Springer Series in Computational Mathematics}.
\newblock Springer-Verlag, New York, 1993.

\bibitem{ST:10}
P.R. Stinga and J.L. Torrea.
\newblock Extension problem and {H}arnack's inequality for some fractional
  operators.
\newblock {\em Comm. Partial Differential Equations}, 35(11):2092--2122, 2010.

\end{thebibliography}

\def\cprime{$'$}

\end{document}